\documentclass[ruled]{article}
\usepackage{geometry}
\usepackage[latin9]{inputenc}
\usepackage{color}
\usepackage{mathtools}

\usepackage{todonotes}
\usepackage{amsmath}
\usepackage{mathrsfs}
\usepackage{accents}
\usepackage{float}
\usepackage{graphicx}
\usepackage{caption}
\usepackage{subcaption}
\usepackage{enumitem}  
\usepackage[unicode=true,pdfusetitle, bookmarks=true,bookmarksnumbered=false,bookmarksopen=false, breaklinks=false,pdfborder={0 0 1},backref=false,colorlinks=true] {hyperref}
\usepackage{setspace}
\usepackage{amsthm}
\usepackage{bbm}
\usepackage{amssymb}
\usepackage[authoryear]{natbib}
\hypersetup{linkcolor=blue,citecolor=blue}

\makeatletter

\theoremstyle{plain}

\theoremstyle{plain}

\floatstyle{ruled}
\newfloat{algorithm}{tbp}{loa}
\providecommand{\algorithmname}{Algorithm}
\floatname{algorithm}{\protect\algorithmname}

\usepackage{algorithm}
\usepackage{dsfont}
\usepackage{algpseudocode}
\usepackage{mathtools}

\DeclareMathOperator*{\argmin}{argmin}

\newcommand{\setZ}{\mathsf{Z}}

\newcommand{\F}{\mathcal{F}}

\newcommand{\bigO}{\mathcal{O}}
\newcommand{\smallo}{{\scriptscriptstyle\mathcal{O}}} 
\newcommand{\R}{\mathbb{R}}
\renewcommand{\P}{\mathbb{P}}
\newcommand{\E}{\mathbb{E}}

\newcommand{\ind}{\mathbbm{1}}

\newtheorem{theorem}{Theorem}
\newtheorem{lemma}{Lemma}
\newtheorem{proposition}{Proposition}

\newtheorem{corollary}{Corollary}

\newtheorem{assumption}{Assumption}

\makeatother

\title{Gradient Descent  for Convex and Smooth Noisy Optimization}
\date{}
\author{
  Feifei Hu and Mathieu Gerber \vspace{0.2cm }\\
  School of Mathematics \vspace{0.2cm }\\
  University of Bristol}

\begin{document}
\maketitle

\begin{abstract}
We study the use of gradient descent with   backtracking line search (GD-BLS)  to solve the noisy optimization problem $\theta_\star:=\argmin_{\theta\in\R^d} \E[f(\theta,Z)]$, imposing  that the objective function $F(\theta):=\E[f(\theta,Z)]$ is strictly convex but not necessarily $L$-smooth. Assuming that $\E[\|\nabla_\theta f(\theta_\star,Z)\|^2]<\infty$, we first prove that sample average approximation  based on GD-BLS allows to estimate $\theta_\star$ with an error of size  $\bigO_\P(B^{-0.25})$, where $B$ is the available computational budget. We then show that we can improve upon this rate by stopping the optimization process earlier when the gradient of the objective function is sufficiently close to zero, and use the residual computational budget to optimize, again with GD-BLS, a finer approximation of $F$. By  iteratively applying this strategy $J$ times we establish that we can estimate $\theta_\star$  with an error of size $\bigO_\P(B^{-\frac{1}{2}(1-\delta^{J})})$, where $\delta\in(1/2,1)$  is  a user-specified  parameter. More generally, we show that if $\E[\|\nabla_\theta f(\theta_\star,Z)\|^{1+\alpha}]<\infty$ for some  known $\alpha\in (0,1]$ then this approach, which can be seen as a retrospective approximation algorithm with a fixed computational budget,  allows to learn $\theta_\star$ with an error of size $\bigO_\P(B^{-\frac{\alpha}{1+\alpha}(1-\delta^{J})})$, where  $\delta\in(2\alpha/(1+3\alpha),1)$ is a tuning parameter.   Beyond knowing $\alpha$, achieving the aforementioned convergence rates do not require to tune the algorithms' parameters according to the specific functions  $F$ and $f$ at hand, and  we exhibit a simple noisy optimization problem for which stochastic gradient is not guaranteed to converge while the algorithms discussed in this work are.

\end{abstract}

\section{Introduction}\label{sec:Intro}

\subsection{Set-up and problem formulation}

We consider the problem of computing
\begin{align}\label{eq:optim_prob}
\theta_\star:=\argmin_{\theta\in\R^d} F(\theta),\quad F(\theta)=\E[f(\theta,Z)]
\end{align}
where $Z$ is a random variable taking its values in a Polish space $\setZ$ and where $f:\R^d\times\setZ\rightarrow\R$ is a measurable function such that $\nabla_\theta f(\theta,z)$ exists for all $(\theta,z)\in\Theta\times \setZ$. We assume that only the function $f$ can be evaluated point-wise and that to solve \eqref{eq:optim_prob} we have at our disposal a sequence $(Z_i)_{i\geq 1}$ of i.i.d.~copies of $Z$. All the random variables are assumed to be defined on a complete probability space $(\Omega,\F,\P)$ and the function $F$ is assumed to be strictly convex on $\R^d$, and thus the optimization problem \eqref{eq:optim_prob} is   well-defined. Formal assumptions on $F$ are given in Section \ref{sub:assumptions}, and   we stress that in this work    the function $\theta\mapsto f(\theta,Z)$ is not supposed to be $\P$-almost surely convex.

Due to their ease of implementation and low computational cost, stochastic gradient (SG) algorithms  arguably constitute the main set of tools used in practice for solving this type of optimization problems \citep{patel2022global}. Stochastic gradient methods are for instance widely applied  for training machine learning algorithms \citep{bottou2018optimization} or for performing statistical inference in large datasets \citep{Toulis2017}. In order to guarantee that a SG algorithm is stable when deployed on an unbounded search space it is commonly assumed  that the variance of the norm of the stochastic gradient $\nabla_\theta f(\theta,Z)$ is bounded and that the function $F$ is $L$-smooth, that is that the gradient of $F$ is $L$-Lipschitz  \citep[see for instance][for a review of   convergence results for SG]{patel2022stopping,garrigos2023handbook}. Under the additional assumption that $F$ is twice differentiable, the condition that $F$ is $L$-smooth can be replaced by the weaker assumption that $F$ is $(L_0,L_1)$-smooth, in the sense that $\|\nabla^2 F(\theta)\|\leq L_0+L_1\|\nabla F(\theta)\|$ for all $\theta\in\R^d$ and some finite constants $L_0$ and $L_1$ \citep[see][Assumption 3]{zhang2019gradient}. The assumption that the variance of $\|\nabla_\theta f(\theta,Z)\|$ is bounded can also be weakened. Notably, letting $C$ be a finite constant, it can instead be assumed that, for all $\theta\in\R^d$, this variance is bounded by $C+C\|\nabla F(\theta)\|^2$    (see e.g.~Assumption 4.3 in \citealp{bottou2018optimization} and   Assumption 2 in \citealp{faw2022power})  or that $\E[\|\nabla_\theta f(\theta,Z)\|^2]\leq C+C\|\nabla F(\theta)\|^2$   \citep[see e.g][Assumption 2]{khaled2020better}. Convergence results for SG that hold under the weaker condition that $\E[\|\nabla_\theta f(\theta,Z)\|^{1+\alpha}]$ for some $\alpha\in (0,1]$ have been obtained by some authors, see for instance \citet{wang2021convergence,patel2022global}.

These   assumptions  used to establish the convergence of SG algorithms  may however be violated even for simple optimization problems. For instance, as shown in \citet{patel2022global}, both the assumption that $F$ is $L$-smooth and the assumption that the variance of $\|\nabla_\theta f(\theta,Z)\|$ is bounded, as well as their variants mentioned in the previous paragraph, may not hold  for parameter inference in Poisson regression models. More precisely, \citet[][Proposition 3]{patel2022global} show  the following result:
 \begin{proposition}\label{prop:Poisson}
 Consider the $d=1$ dimensional optimization problem \eqref{eq:optim_prob} where $Z=(X,Y)$, with $X$ and $Y$  two independent Poisson random variables such that $\E[X]=\E[Y]=1$, and where  the function $f:\R^d\times\setZ\rightarrow\R$ is defined by
$$
f(\theta,z)=-z_2z_1\theta+\exp(\theta z_1),\quad \theta\in\R,\quad z\in\setZ:=\R^2.
$$
Then, the function $F$ is strictly convex, twice continuously differentiable and  $\mathrm{Var}\big(\nabla_\theta f(\theta,Z)\big)<\infty$ for all $\theta\in\R$. However, $F$ is neither  globally $L$-smooth nor $(L_0,L_1)$-smooth, and there exists no finite constant $C$ such that $\mathrm{Var}\big(\nabla_\theta f(\theta,Z)\big)\leq C+C  |\nabla F(\theta) |^2$ for all $\theta\in\R$ or such that $\E\big[|\nabla_\theta f(\theta,Z)|^2\big]\leq C+C |\nabla F(\theta)|^2$ for all $\theta\in\R$ (and thus $\sup_{\theta\in\R}\mathrm{Var}\big(\nabla_\theta f(\theta,Z)\big)=\infty$).
 \end{proposition}

To the best of our knowledge, \citet{patel2022global} provide the only existing general convergence result for SG on unbounded spaces that holds under the assumption that $F$ is only locally $L$-smooth, and which allows the variance of $\|\nabla_\theta f(\theta,Z) \|$ to be unbounded.    \citet{patel2022global} actually only assume  that, for some $\alpha\in (0,1]$, the gradient of $F$ is locally $\alpha$-H\"older continuous and that $\E[\|\nabla_\theta f(\theta,Z)\|^{1+\alpha}]$ is finite for all $\theta\in\R^d$. However, to guarantee the stability of SG,  their  result requires that the  H\"older constant of $\nabla F$ on a ball centred at $\theta$ and of size $\E[\|\nabla_\theta f(\theta,Z) \|^{1+\alpha}]^{\frac{1}{1+\alpha}}$ does not increase too quickly as $\|\theta\|\rightarrow\infty$ \citep[see Assumption 5 in][]{patel2022global}. This condition appears to be quite strong since, as we show  in the following proposition, it is not satisfied for the optimization problem considered in Proposition \ref{prop:Poisson}. 
\begin{proposition}\label{prop:Poisson22}
 Consider the set-up of Proposition \ref{prop:Poisson}. Then, there exists no $\alpha\in(0,1]$ such that Assumption 5 used in \citet{patel2022global} to establish the convergence of SG   holds.
 \end{proposition}

Finally, we   note that   \citet{fehrman2020convergence}  study SG algorithms   under the assumption that $F$ is locally  $L$-smooth and allowing $\E[\|\nabla_\theta f(\theta,Z) \|^2]$ to be unbounded but, informally speaking, this reference only shows that SG will converge to $\theta_\star$ if the starting value of the algorithm is sufficiently close to that point. 

While we previously focused on SG methods, it is worth mentioning that stochastic proximal point methods \citep[see][and references therein]{asi2019stochastic} form another popular set of tools for solving optimization problems as defined in \eqref{eq:optim_prob}. These  algorithms have the advantage of having stronger theoretical properties than SG approaches but, on the other hand, they usually have the  drawback to require, at each iteration, to use a numerical method to approximate the solution of an optimization problem. The Poisson regression example introduced in Proposition \ref{prop:Poisson} provides a simple  illustration of this claim: convergence results for the stochastic proximal  point method applied to this problem have been obtained by \citet{asi2019stochastic} but, for this example,  there is no explicit solution to the optimization problem that needs to be solved at each iteration of the algorithm. We finally note that the results derived in \citet{asi2019stochastic} in fact apply to a broader class of algorithms, called  model-based   stochastic proximal  point methods. In these algorithms, the choice of the model determines the optimization problem that needs to be solved at each iteration, and \citet[][Example 3]{asi2019stochastic} use the  Poisson regression example as a motivating example for what they call the truncated model. However, the authors do not show that the resulting algorithm is stable, that is  that their key condition (C.iv) used to establish the stability of  model-based   stochastic proximal  point methods is satisfied.

Consequently, as illustrated by the Poisson regression example of Proposition \ref{prop:Poisson}, there are seemingly simple noisy optimization problems   on $\R^d$, where the objective function $F$ is strictly convex and $\E[\|\nabla_\theta f(\theta,Z) \|^2]$ is finite for all $\theta\in\R^d$,  which cannot be solved (easily) with theoretical guarantees using existing methods.

\subsection{Contribution of the paper}\label{sub:intro}

In this paper we first show that   gradient descent   with backtracking line search (GD-BLS) can be used to solve \eqref{eq:optim_prob} 
even in situations where $F$ is not $L$-smooth and where $\sup_{\theta\in\R^d}\E[\|\nabla_\theta f(\theta,Z)\|^{1+\alpha}]=\infty$ for all $\alpha\in(0,1]$.

 More precisely, we assume throughout this work that there is a computational cost $C_{\mathrm{grad}}\in \mathbb{N}$ for computing $\nabla_\theta f(\theta,z)$ and a computational cost $C_{\mathrm{eval}}\in \mathbb{N}$  for computing  $f(\theta,z)$  for a  pair $(\theta,z)\in\R^d\times\setZ$. Then, we first show that, given  the available computational budget $B\in\mathbb{N}$, by  applying GD-BLS to minimize the  random function  $F_{n_1(B)}:=\frac{1}{n_1(B)}\sum_{i=1}^{n_1(B)} f(\cdot, Z_i)$ we obtain a consistent  estimator $\hat{\theta}_{1,B}$ of $\theta_\star$. The convergence rate of this approach, which is a particular type of sample  average approximation algorithms \citep[see e.g.][]{kim2015guide}, depends on the choice of the function  $n_1$ and,  assuming for now that $\E[\|\nabla_\theta f(\theta_\star,Z)\|^{2}]<\infty$, the upper bound we obtain for this rate is  minimized when  $n_1(B)\approx B^{1/2}$, in which case $\|\hat{\theta}_{1,B}-\theta_\star\|=\bigO_\P(B^{-0.25})$.   Remark that the $B^{-0.25}$ convergence rate obtained for $\hat{\theta}_{1,B}$ is slower than the standard $B^{-0.5}$ convergence  rate (see Proposition \ref{prop:rate}).

In order to save computational resources  gradient descent algorithms are often stopped if a pre-specified convergence criterion is reached, typically when the gradient of the objective function is sufficiently close to zero. Following this approach,  we propose to allow  the optimization of $F_{n_1(B)}$ to stop earlier (i.e.~before the computational budget $B$ is exhausted) and, if any, to use  the remaining computational resources   to improve   the estimation of $\theta_\star$. In this work the optimization of  $F_{n_1(B)}$ is stopped  if its gradient  falls below some threshold $\tau_1(B)$ and the residual computational budget is used   to apply  GD-BLS  on a finer approximation $F_{n_2(B)}$ of $F$ (i.e.~$n_2(B)>n_1(B)$),  using   the estimate of $\theta_\star$ obtained by minimizing  $F_{n_1(B)}$ as initial value.  With  a careful choice for the functions $n_1$, $\tau_1$ and $n_2$ we show that the resulting estimate $\hat{\theta}_{2,B}$ of $\theta_\star$ has an error of size $B^{-0.5(1-\delta^2) }$, with $\delta\in(1/2,1)$  a user-specified parameter.

The  approach outlined in the previous paragraph can be generalized as follows to any $J\geq 2$ applications of GD-BLS on increasingly finer approximations of $F$: For $j=1,\dots,J-1$ we  stop earlier the optimization of the function $F_{n_j(B)}$ by GD-BLS if its gradient becomes smaller than some threshold $\tau_j(B)$ and, if any,  the remaining computational resources are used to repeat this process to the function $F_{n_{j+1}(B)}$, using  as starting value the estimate $\tilde{\theta}_{j,B}$ of $\theta_\star$ obtained by minimizing  $F_{n_j(B)}$. For a suitable choice of functions $(n_j)_{j\geq 1}$  and $(\tau_j)_{j\geq 1}$, and a user-specified parameter $\delta\in (1/2,1)$, we prove that the estimator $\hat{\theta}_{J,B}$ of $\theta_\star$ defined by this approach has an error of size $\bigO_\P(B^{-\frac{1}{2}(1-\delta^J)})$  while, for all $j\in\{1,\dots,J-1\}$, the estimator $\tilde{\theta}_{j,B}$ converges to $\theta_\star$ at rate $B^{-\frac{1}{2}(1-\delta^j)}$.  The assumptions under which these results are obtained are shown to be satisfied for the optimization problem considered in Proposition \ref{prop:Poisson}, and   to achieve the aforementioned  convergence rates it is not needed to tune the parameters of the algorithms according to the specific functions $F$ and $f$ at hand. 

 In practice, to observe the $B^{-\frac{1}{2}(1-\delta^J)}$ convergence rate of $\hat{\theta}_{J,B}$ it is necessary that the optimization of $F_{n_j(B)}$ stops earlier for all $j\in\{1,\dots,J-1\}$, and thus that GD-BLS is applied to the function $F_{n_J(B)}$. Informally speaking, if we denote by $J_B$ the smallest value of $j\in\{1,\dots,J\}$ such that  $\tilde{\theta}_{j,B}=\hat{\theta}_{J,B}$, we therefore expect the observed convergence rate for  $\hat{\theta}_{J,B}$ to be $B^{-\frac{1}{2}(1-\delta^{J_B})}$. In this work we do not study theoretically the behaviour of   $J_B$ but the numerical experiments suggest that, as $B$ increases, this random variable `grows' at speed $\log(B)$.

In some applications the variance of the noisy gradient is not finite, and in this work the  procedure described above is studied under the weaker assumption that $\E[\|\nabla_\theta f(\theta_\star,Z)\|^{1+\alpha}]<\infty$ for some $\alpha\in(0,1]$. In this scenario, the user needs to choose a noise parameter $\alpha'\in (0,1]$ and a $\delta\in (2\alpha'/(1+3\alpha'),1)$, and the proposed approach for computing $\theta_\star$ is proved to be consistent for any $\alpha'\in (0,1]$. In addition, if $\alpha'=\alpha$ then its convergence   rate for computing $\theta_\star$ is shown to be $B^{-\frac{\alpha}{1+\alpha}(1-\delta^{J})}$. By contrast, in this context  and as shown in Proposition \ref{prop:rate},  under the assumptions imposed on \eqref{eq:optim_prob} in the next section  $B^{- \frac{\alpha}{1+\alpha}}$ is the best possible convergence rate for estimating $\theta_\star$.

\subsection{Connection with retrospective approximation}

The proposed $J$-step procedure   can be seen as a retrospective approximation (RA) algorithm with a fixed computational budget and using GD-BLS as deterministic optimizer \citep[see.][for a survey on RA methods]{kim2015guide}. Indeed,  given a sequence $(n'_j)_{j\geq 1}$ in $\mathbb{N}$ and a sequence $(\tau_j')_{j\geq 1}$ in $(0,\infty)$,  RA based on GD-BLS amounts to applying the above  $J$-step procedure with $n_j(B)=n'_j$ and $\tau_j(B)=\tau'_j$ for all $B\in\mathbb{N}$ and $j\in\mathbb{N}$, and with $B=\infty$ (so that the optimization of $F_{n_j'}$ is performed for all $j\in\{1,\dots,J\}$). Unlike in the proposed approach, the computational budget $B_J$ consumed by an RA algorithm is therefore random. The behaviour of $\E[B_J]$ and the convergence rate, as $J\rightarrow\infty$, of the resulting estimator $\theta_{J}$ of $\theta_\star$ have been recently studied by \citet{newton2024retrospective} under general conditions on the optimization routine $\mathcal{R}$  used to optimize $F_{n'_j}$ for all $j\in\{1,\dots,J\}$. However,  the results presented in this   reference are obtained under conditions which seem hard to fulfil if  $F$ is not (globally) $L$-smooth.

\subsection{Outline of the paper}

The rest of the paper is organized as follows. In Section \ref{sec:not_prelim} we state the assumptions we impose on \eqref{eq:optim_prob} and present some preliminary results. The use of  GD-BLS for solving noisy optimization problems is studied in Section \ref{sec:GD1}. In  Section \ref{sec:example} we illustrate the convergence results derived in this work with some numerical experiments and Section  \ref{sec:conclusion} concludes. All the proofs are gathered in the appendix.

\section{Notation, assumptions and some preliminary results\label{sec:not_prelim}}
 
 \subsection{Notation}

To simplify the notation, for all $(\theta,z)\in\R^d\times\setZ$ we use the shorthand $\nabla^k f(\theta,z)$ for $\nabla^k_\theta f(\theta,z)$ for all $k\in\{0,1,2\}$, with the convention that $\nabla^k f(\theta,z)= f(\theta,z)$ when $k=0$. For all $n\in\mathbb{N}$ and compact set $\Theta\subset\R^d$ we let $\hat{\theta}_{\Theta,n}\in\argmin_{\theta\in\Theta}F_n(\theta)$ and we denote by $\lambda_{\min}(M)$ and $\lambda_{\max}(M)$ the smallest and largest eigenvalue of a   $d\times d$ matrix $M$, respectively. To introduce some further notation let $(X_t)_{t\geq 1}$ be a sequence of $\R$-valued random variables. Then, we write $X_t=\smallo_\P(1)$ if and only if $X_t\rightarrow 0$ in $\P$-probability, that is if and only if $\lim_{t\rightarrow\infty}\P(|X_t|\geq \epsilon)=0$ for all $\epsilon>0$, and for a sequence $(a_t)_{t\geq 1}$ in $(0,\infty)$ we write $X_t=\bigO_\P(a_t)$ if and only if for all $\epsilon\in(0,1)$ there exists a finite constant $M_\epsilon$ such that $\limsup_{t\rightarrow\infty}\P\big(|X_t/a_t|\geq M_\epsilon)\leq \epsilon$.

\subsection{Assumptions\label{sub:assumptions}}

We consider the following two assumptions on the optimization problem \eqref{eq:optim_prob}:

\begin{assumption}\label{assumption1}
The function $F$ is locally strongly convex, that is, for every compact set $\Theta\subset\R^d$ there exists a constant $c_\Theta>0$ such that 
$$
F\big(a\theta_1+(1-a)\theta_2\big)\leq a F(\theta_1)+(1-a)F(\theta_2)-\frac{c_\Theta }{2}a(1-a)\|\theta_1-\theta_2\|^2,\quad\forall a\in[0,1],\quad\forall (\theta_1,\theta_2)\in\Theta^2.
$$
\end{assumption}
Remark that under   \ref{assumption1} we have $F\big(a\theta_1+(1-a)\theta_2\big)< a F(\theta_1)+(1-a)F(\theta_2)$ for all $a\in[0,1]$ and all $(\theta_1,\theta_2)\in\R^d\times \R^d$, meaning that the function $F$ is strictly convex on $\R^d$. Consequently, under \ref{assumption1}, the optimization problem \eqref{eq:optim_prob} is well-defined.

\begin{assumption}\label{assumption2}

The function $\theta\mapsto f(\theta,z)$ is  twice continuously differentiable on $\R^d$ for all $z\in \setZ$, and 
$$
\E[\|\nabla^2 f(\theta_\star, Z)\|]<\infty,\quad \E\big[ \|\nabla f(\theta_\star,Z)\|^{1+\alpha}\big]<\infty
$$
for some $\alpha\in(0,1]$. Moreover,  for all compact  set $\Theta\subset\R^d$  there exists a measurable function $M_\Theta:\mathsf{Z}\mapsto [0,\infty)$ such that $\E[M_{\Theta}(Z)]<\infty$ and such that
$$
\|\nabla^2 f(\theta,Z)-\nabla^2 f(\theta',Z)\|\leq M_\Theta(Z)\|\theta-\theta'\|,\quad \forall (\theta,\theta')\in\Theta^2,\quad \P-a.s.
$$
\end{assumption}

These two assumptions are quite weak. In particular, as shown in the following proposition, they are  satisfied for the optimization problem considered in Propositions \ref{prop:Poisson}-\ref{prop:Poisson22}.
\begin{proposition}\label{prop:Poisson3}
 Consider the set-up of Proposition \ref{prop:Poisson}. Then, Assumptions \ref{assumption1}-\ref{assumption2} are satisfied with $\alpha=1$.
\end{proposition}

\subsection{Some implications of Assumptions \ref{assumption1}-\ref{assumption2}\label{sub:prelim_assume} }

The following lemma notably implies  that, under \ref{assumption2}, the random function $\theta\mapsto f(\theta,Z)$ is $\P$-a.s.~locally $L$-smooth.
\begin{lemma}\label{lemma:assume1}
Assume that \ref{assumption2} holds. Then, for all compact set $\Theta\subset\R^d$   there exists a measurable function  $M_\Theta':\mathsf{Z}\mapsto [0,\infty)$   such that $\E[M_\Theta'(Z)]<\infty$ and such that, for all $k\in\{0,1,2\}$, we have
$$
\|\nabla^k f(\theta,Z)-\nabla^k f(\theta',Z)\|\leq M_\Theta'(Z)\|\theta-\theta'\|,\quad \forall (\theta,\theta')\in\Theta^2,\quad \P-a.s.
$$
\end{lemma}

Remark that, in the set-up of Lemma \ref{lemma:assume1}, for any $\epsilon>0$ we may have   $\E[M_\Theta'(Z)^{1+\epsilon}]=\infty$. Therefore, under \ref{assumption2}, the quantity $\E\big[\|\nabla f(\theta,Z)\|^{1+\alpha}\big]$ is assumed to be finite for some $\alpha\in(0,1]$ only if $\theta=\theta_\star$. In addition, it is worth noting that no conditions  are imposed on the behaviour of the function $M'_\Theta$ as the size of the set $\Theta$ increases.

The next lemma shows that, under \ref{assumption2}, we can swap the expectation and the differential operators.

\begin{lemma}\label{lemma:assume2}
Assume that  \ref{assumption2} holds. Then, the function $F$ is twice  continuously differentiable on $\R^d$ and 
$$
\nabla^k F(\theta)=\E\big[\nabla^k f(\theta,Z)\big], \quad\forall \theta\in\R^d,\quad\forall k\in\{0,1,2\}.
$$
\end{lemma}
By combining Lemma \ref{lemma:assume1} and Lemma \ref{lemma:assume2}  it follows   that under \ref{assumption2} the function $F$ is locally $L$-smooth. More precisely, under \ref{assumption2}, for any   compact set $\Theta\subset\R^d$ we have
$$
\|\nabla F(\theta)-\nabla F(\theta')\|\leq \E[M'_\Theta(Z)]\|\theta-\theta'\|,\quad\forall (\theta,\theta')\in\Theta^2
$$
with the function $M'_\Theta$ as in Lemma \ref{lemma:assume1}. We stress that, unlike in \citet{patel2022global},  no conditions  are imposed on the behaviour of the Lipschitz constant $\E[M'_\Theta(Z)]$ as the size of the set $\Theta$ increases.

The following lemma shows that the random function $F_n$  converges uniformly to $F$ on any compact subset of $\R^d$.
\begin{lemma}\label{lemma:uniform_conv}
Assume that  \ref{assumption2} holds. Then, for all   compact set $\Theta\subset\R^d$ we have
$$
\lim_{n\rightarrow\infty}\sup_{\theta\in\Theta}\big\|\nabla^k F_n(\theta)- \nabla^k F(\theta)\big\|=0,\quad \P-a.s.\quad \forall k\in\{0,1,2\}.
$$
\end{lemma}

The next lemma provides a useful result on the asymptotic behaviour of the largest and   smallest eigenvalues of the random matrix  $\nabla^2 F_n(\theta)$.

\begin{lemma}\label{lemma:e-value}
Assume that  \ref{assumption2} holds. Then, for all  compact set $\Theta\subset\R^d$  there exists a $\theta_\Theta\in\Theta$ such that
$$
\liminf_{n\rightarrow\infty}\Big(\inf_{\theta\in\Theta}\lambda_{\min}\big(\nabla^2 F_n(\theta)\big)\Big)\geq \lambda_{\min}\Big(\E\big[\nabla^2 f(\theta_\Theta,Z)\big]\Big),\quad\P-a.s.
$$
and a  $\theta'_\Theta\in\Theta$ such that
$$
\limsup_{n\rightarrow\infty}\Big(\sup_{\theta\in\Theta}\lambda_{\max}\big(\nabla^2 F_n(\theta)\big)\Big)\leq \lambda_{\max}\Big(\E\big[\nabla^2 f(\theta'_\Theta,Z)\big]\Big),\quad\P-a.s.
$$
\end{lemma}

Under \ref{assumption1} and by combining Lemma \ref{lemma:assume2} and Lemma \ref{lemma:e-value}, we readily obtain the following result which notably shows that if   $\Theta\subset\R^d$ is a   compact and convex set then the random function $F_n$ is strictly convex on $\Theta$ with $\P$-probability tending to one.

\begin{corollary}\label{cor:convex}
Assume that \ref{assumption1}-\ref{assumption2} hold. Then, for all  compact   set $\Theta\subset\R^d$  there exists a constant $c_\Theta\in(0,1)$ such that
\begin{align*}
\liminf_{n\rightarrow\infty}\Big(\inf_{\theta\in\Theta}\lambda_{\min}\big(\nabla^2 F_n(\theta)\big)\Big)\geq c_\Theta,\quad \limsup_{n\rightarrow\infty}\Big(\sup_{\theta\in\Theta}\lambda_{\max}\big(\nabla^2 F_n(\theta)\big)\Big)\leq\frac{1}{c_\Theta},\quad\P-a.s.
\end{align*}
\end{corollary}

The following lemma  establishes that,  under \ref{assumption1} and \ref{assumption2}, by using the first $n$ elements of $(Z_i)_{i\geq 1}$ we can learn $\theta_\star$  at rate  $n^{-\frac{\alpha}{1+\alpha}}$, with $\alpha\in (0,1]$ as in \ref{assumption2}. In addition to be used to establish our main results, this lemma will allow us to discuss    the convergence rate of the proposed estimators.

\begin{lemma}\label{lemma:MLE}
Assume that \ref{assumption1}-\ref{assumption2} hold and let $\Theta\subset\R^d$ be a  compact and convex   set such that $\theta_\star\in\mathring{\Theta}$. Then, $\|\theta_\star-\hat{\theta}_{\Theta,n}\|=\bigO_\P(n^{-\frac{\alpha}{1+\alpha}})$, with $\alpha\in  (0,1]$ as in \ref{assumption2}.
\end{lemma}

Finally, the following proposition  shows that, under   \ref{assumption1}-\ref{assumption2} and in the context of this work (i.e.~noisy optimization under fixed computations budget) the best possible convergence rate for an estimator of $\theta_\star$ is $B^{-\frac{\alpha}{1+\alpha}}$, with $\alpha\in(0,1]$ as in \ref{assumption2}.

\begin{proposition}\label{prop:rate}
Let $\alpha'\in (0,1]$. Then, for any sequence  $(\hat{\theta}_B)_{B\geq 1}$  of estimators of $\theta_\star$ such that, for all $B\in\mathbb{N}$,  $\hat{\theta}_B$ depends on $(Z_i)_{i\geq 1}$ only through $k_B\in\mathbb{N}_0$ evaluations of $f(\cdot,\cdot)$ and $k_B'\in\mathbb{N}_0$ evaluations of $\nabla f(\cdot,\cdot)$, where $k_B C_{\mathrm{eval}}+k_B'C_{\mathrm{grad}}\leq B$, there is a noisy optimization problem \eqref{eq:optim_prob} such that   \ref{assumption1}-\ref{assumption2} hold with $\alpha=\alpha'$ and such that,   for some constants $\delta\in(0,1)$ and $C\in(0,\infty)$, we have
\begin{align*}
\limsup_{B\rightarrow\infty}\P\Big(\|\hat{\theta}_{B}-\theta_\star\|> C B^{-\frac{\alpha}{1+\alpha}}\Big)\geq \delta.
\end{align*}
\end{proposition}

\subsection{Gradient descent with backtracking line search\label{sub:prelim_gd}}

Algorithm \ref{algo:GD} provides a generic description  of a gradient descent algorithm that can be used to minimize a differentiable function $g:\R^d\rightarrow\R$. The convergence behaviour of Algorithm \ref{algo:GD} depends  on the value of the step-size $v_t$ chosen on its Line \ref{SS}, and our theoretical analysis of gradient descent applied to  noisy optimization problems builds on the following  result.

\begin{proposition}\label{prop:GD_result_0}
Let $g:\R^d\rightarrow\R$ be a  differentiable function,    $x_0\in\R^d$,  and  for all $T\in\mathbb{N}$  let $x_T$ be as computed by Algorithm \ref{algo:GD} with $(v_t)_{t\geq 1}$ such that,  for some constant $c>0$, 
\begin{align}\label{eq:vt}
g(x_t)\leq g(x_{t-1}) -\frac{v_t}{2}\|\nabla g(x_{t-1})\|^2,\quad v_t\geq c,\quad\forall t\geq 1.
\end{align}
Assume that there exists a convex  set $K\subseteq\R^d$ such that (i)  $x_t\in K$ for all $t\geq 0$, (ii) the function $g$ is strictly convex on $K$ and (iii) there exists an $x_\star\in K$ such that $g(x_\star)= \inf_{x\in K}g(x)$.   Then,
$$
g(x_T)- g(x_\star)\leq \frac{\|x_0-x_\star\|^2}{2 c  T},\quad\forall T\geq 1.
$$
\end{proposition}

If the function $g$ is $L$-smooth (assuming without loss of generality that $L>1$) then \eqref{eq:vt} holds for instance when $v_t=v$ for all $t\geq 1$ and  some $v\in(0,1/L]$ \citep[see e.g.~the proof of][Theorem 2.1.14, page 80]{nesterov2018lectures}. In  practice, the value of $L$ is however rarely known and   backtracking line search, described in Algorithm \ref{algo:Back}, is a simple  procedure that can be used to automatically compute, at each iteration $t$, a value for $v_t$ which ensures that  \eqref{eq:vt} holds with $c= \beta/L$, where  $\beta\in (0,1)$ is a  parameter of Algorithm \ref{algo:Back}. This property  of backtracking line search is formalized in the following proposition, which plays a key role in the proofs of our main results.

\begin{proposition}\label{prop:GD_result}
Let $g:\R^d\rightarrow\R$ be a   differentiable function,   $x_0\in\R^d$, $\beta\in (0,1)$, and  for all $T\in\mathbb{N}$  let $x_T$ be as computed by Algorithm \ref{algo:GD} where Algorithm \ref{algo:Back} is used to compute $v_t$ for all $t\geq 1$. In addition, let   $\tilde{K}\subseteq\R^d$ be a  convex set containing the set $\{x\in\R^d:\, g(x)\leq g(x_0)\}$ and let
$$
K=\Big\{x\in\R^d:\,\exists x'\in \tilde{K}\text{ such that }\|x-x'\|\leq \sup_{\tilde{x}\in \tilde{K}}\|\nabla g(\tilde{x})\|\Big\}. 
$$
Assume that the function $g$ is strictly convex on $K$ and that there exists a constant $L_K\in[1,\infty)$  such that  $\|\nabla g(x)-\nabla g(x')\|\leq L_K\|x-x'\|$ for all $(x,x')\in K^2$. Then, \eqref{eq:vt} holds with $c=\beta/L_K$.
\end{proposition}

Given the large literature on convex optimization, the conclusions of Propositions \ref{prop:GD_result_0}-\ref{prop:GD_result} are probably known. We however  failed to find a reference establishing   similar results and therefore the two propositions are proved in the appendix.

\begin{algorithm} [t]
\begin{algorithmic}[1]
\Require Starting value $x_0\in\R^d$ and number of iterations $T\in\mathbb{N}$.
\vspace{0.01cm}

\For{$t=1,\dots,T$}
\State\label{SS}Choose a set-size $v_t$
\State $x_{t}=x_{t-1}-v_t \nabla g(x_{t-1})$
\EndFor

\hspace{-1.1cm}\textbf{Return:} $x_T$
\end{algorithmic}
\caption{(Generic gradient descent algorithm)\label{algo:GD}}
\end{algorithm}

\begin{algorithm} [!t] 
\begin{algorithmic}[1]
\Require $x_{t-1}\in\R^d$ and backtracking parameter $\beta\in(0,1)$.
\vspace{0.01cm}

\State $g_{t-1}= \nabla g(x_{t-1})$
\State $v_t=1$
\While {$g\big(x_{t-1}-v_t\, g_{t-1} \big)>g(x_{t-1})-\frac{v_t}{2}\|g_{t-1}\|^2$}
\State $v_t\gets \beta v_t$
\EndWhile

\State \textbf{return:} $v_t$
\end{algorithmic}
\caption{ (Backtracking line search)\label{algo:Back}}
\end{algorithm}

\begin{algorithm}[!t]
\begin{algorithmic}[1]
\Require Starting value $\theta_0\in\R^d$, sample size $n\in\mathbb{N}_0$, computational budget $B\in\mathbb{N}_0$, tolerance

parameter $\tau\in[0,\infty)$  and backtracking parameter $\beta\in(0,1)$.
\vspace{0.2cm}

\State Let  $t=0$, $\theta_{B,t}=\theta_0$ and $\tilde{B}=B$
\If{$\tilde{B}\geq nC_{\mathrm{grad}}$}
\State Let $G_t=\nabla F_n(\theta_{B,t})$
\State $\tilde{B}\gets\tilde{B}-nC_{\mathrm{grad}}$
\EndIf
\If{$\tilde{B}\geq nC_{\mathrm{eval}}$}
\State  Let  $g_t=F_n(\theta_{B,t})$
\State $\tilde{B}\gets \tilde{B}-n C_{\mathrm{eval}}$
\EndIf
\While{   $\|G_t\|> \tau $  and $\tilde{B}\geq  nC_{\mathrm{eval}}$}
\State Let   $g'_t=F_n(\theta_{B,t}-G_t)$ and $v_{t+1}=1$
\State $\tilde{B}\gets \tilde{B}-  nC_{\mathrm{eval}}$
\While {$g'_t>g_t-\frac{v_{t+1}}{2}\|G_t\|^2$ and $\tilde{B}\geq n C_{\mathrm{eval}}$}
\State $v_{t+1}\gets \beta v_{t+1}$
\State $g'_t\gets F_n(\theta_{B,t}-v_{t+1}G_t)$
\State $\tilde{B}\gets \tilde{B}-n C_{\mathrm{eval}}$
\EndWhile
\If{$g'_t\leq g_t-\frac{v_{t+1}}{2}\|G_t\|^2$}
\State\label{compute} Let $\theta_{B,t+1}=\theta_{B,t}-v_{t+1} G_t$ and $g_{t+1}= g'_t$
\State $t\gets t+1$
\If{$\tilde{B}\geq nC_{\mathrm{grad}}$}
\State Let $G_t\gets \nabla F_n(\theta_{B,t})$
\State $\tilde{B}\gets\tilde{B}-nC_{\mathrm{grad}}$
\Else
\State \textbf{break}
\EndIf
\EndIf
\EndWhile 

\State \textbf{return:} $( \theta_{B,t}, \tilde{B}\vee 0)$

\end{algorithmic}
\caption{\label{algo:GD1}}
\end{algorithm}

\section{Using gradient descent for noisy optimization\label{sec:GD1}}

\subsection{A simple application of gradient descent to noisy optimization\label{sub:GD_rate}}

Recall that it is assumed in this work   that there is a computational cost  $C_{\mathrm{grad}}\in \mathbb{N}$ for computing $\nabla f(\theta,z)$  and a computational cost $C_{\mathrm{eval}}\in \mathbb{N}$  for computing  $f(\theta,z)$  for a given pair $(\theta,z)\in\R^d\times\setZ$.  Then, for a given  computational budget $B\in\mathbb{N}$, a simple way to estimate $\theta_\star$ is to apply Algorithm \ref{algo:GD} to minimize the function  $g=F_{n(B)}$ for some $n(B)\in\mathbb{N}$,  using at each iteration the backtracking procedure described in Algorithm \ref{algo:Back} for choosing the step-size  and running the algorithm  until the computational resources are exhausted. Algorithm \ref{algo:GD1} (with $\tau=0$) provides a precise description of  this procedure.

 In order for the resulting estimator $\tilde{\theta}_{B}$ of $\theta_\star$ to be consistent as $B\rightarrow\infty$, it is clear that we must have (i) $n(B)\rightarrow\infty$, so that $F_{n(B)}\rightarrow F$, and (ii) $n(B)/B\rightarrow 0$, so that the number of iterations of the algorithm converges to infinity. The following theorem formalizes this intuitive reasoning and provides an upper bound on the convergence rate of $\|\tilde{\theta}_{B}-\theta_\star\|$.

\begin{theorem}\label{thm:rate_GD} 
Assume that \ref{assumption1}-\ref{assumption2} hold. Let $\theta_0\in\R^d$, $\beta\in(0,1)$ and $n:\mathbb{N}\rightarrow \mathbb{N}$ be such that  $\lim_{B\rightarrow\infty} n(B)/B=0$ and such that $\lim_{B\rightarrow\infty} n(B)=\infty$. In addition, for all $B\in\mathbb{N}$ let 
$$
(\tilde{\theta}_{B}, \tilde{B})=\mathrm{Algorithm\, \ref{algo:GD1} }\big(\theta_0,n(B), B, 0, \beta\big),\quad r(B)=\max\big(\sqrt{n(B)/B}, n(B)^{-\frac{\alpha}{1+\alpha}}\big)
$$
with $\alpha\in (0,1]$ as in \ref{assumption2}. Then, $\|\tilde{\theta}_{B}-\theta_\star\|=\bigO_\P\big(r(B)\big)$ and $\liminf_{B\rightarrow\infty} B^{\frac{\alpha}{1+3\alpha}}\,r(B)>0$.
\end{theorem}

The expression for the rate $r(B)$ has a very intuitive and simple explanation. Letting    $\Theta$ be a compact set containing $\theta_\star$,  we have
\begin{align}\label{eq:rate_B}
\|\tilde{\theta}_{B}-\theta_\star\|\leq \|\tilde{\theta}_{B}-\hat{\theta}_{\Theta,n(B)}\|+\|\hat{\theta}_{\Theta,n(B)}-\theta_\star\| 
\end{align}
where the first term on the right-hand side   is the numerical error resulting from using GD-BLS to minimize  $F_{n(B)}$. Under the assumptions of the theorem, if $T$ iterations of  GD-BLS  are performed then the numerical error is  of size $T^{-1/2}$. In the proof of the theorem we show that, under \ref{assumption1}-\ref{assumption2}, the cost per iteration of Algorithm \ref{algo:GD1} is approximatively constant, implying that $\tilde{\theta}_{B}$ is computed by running $T\approx B/n(B)$ iterations of GD-BLS. Consequently, the   numerical error $\|\tilde{\theta}_{B}-\hat{\theta}_{\Theta,n(B)}\|$ is of size $\sqrt{B/n(B)}$, which is the first term appearing in the definition of $r(B)$. The second term in \eqref{eq:rate_B} is the statistical error which, by Lemma \ref{lemma:MLE}, is of size $\bigO_\P(n(B)^{-\frac{\alpha}{1+\alpha}})$. Therefore, the rate $r(B)$ given in Theorem \ref{thm:rate_GD} is simply the maximum between the numerical error and the statistical error. 

The last part of Theorem \ref{thm:rate_GD} follows from the fact that the fastest rate at which $r(B)$ can decrease to zero is obtained when the numerical error and the statistical error vanish at the same speed. It is direct to see that this happens when $n(B)\approx B^{\frac{1+\alpha}{1+3\alpha}}$, in which case $r(B)=\bigO(B^{-\frac{\alpha}{1+3\alpha}})$. This result is formalized in the following corollary.
\begin{corollary}\label{cor:1step}
Consider the set-up of Theorem \ref{thm:rate_GD} with $n:\mathbb{N}\rightarrow \mathbb{N}$ such that
$$
0<\liminf_{B\rightarrow\infty}\frac{n(B)}{B^{\frac{1+\alpha}{1+3\alpha}}}\leq \limsup_{B\rightarrow\infty}\frac{n(B)}{B^{\frac{1+\alpha}{1+3\alpha}}}<\infty.
$$
Then,  $\|\tilde{\theta}_{B}-\theta_\star\|=\bigO_\P\big(n(B)^{-\frac{\alpha}{1+\alpha}}\big)=\bigO_\P(B^{-\frac{\alpha}{1+3\alpha}})$.
\end{corollary}

Algorithm \ref{algo:GD1} belongs to the well-known class of sample average approximation   algorithms  \citep[see][for a  survey of sample average approximation methods]{kim2015guide}, which aim at computing an estimate of $\theta_\star$ by minimizing the function $F_{n(B)}$ with a deterministic optimization routine $\mathcal{R}$. The  optimal allocation of the available computational resources (i.e.~of the optimal choice of $n(B)$) is studied in \citet{royset2013optimal} under general conditions on the optimizer  $\mathcal{R}$ but, unlike Theorem \ref{thm:rate_GD} and Corollary \ref{cor:1step}, the results presented in this reference assume that the search space is compact and they are limited to the case where $\alpha=1$.

\subsection{An improved application of gradient descent to noisy optimization\label{sec:GD2}}

\subsubsection{The algorithm\label{sub:algo}}

The procedure sketched in the introductory section to improve  upon the  convergence rate of $\tilde{\theta}_{B}$ obtained in Corollary  \ref{cor:1step} is presented in Algorithm \ref{algo:GD2}. Before providing, in Theorem \ref{thm:rate_GD_main} below, a convergence result for this algorithm  some comments are in order.

Firstly,  in Algorithm \ref{algo:GD2}, for all $j$ a specific form is imposed on  the number  $n_j(B)$ of elements of $(Z_i)_{i\geq 1}$  and on the tolerance parameter $\tau_j(B)$ that are used in the $j$-th call of Algorithm \ref{algo:GD1}. This is only to simplify the presentation of the algorithm and  the conclusions of Theorem \ref{thm:rate_GD_main}  hold  for any alternative functions $\tilde{n}_j$ and $\tilde{\tau}_j$ satisfying
$$
0<\liminf_{B\rightarrow\infty}\frac{n_j(B)}{\tilde{n}_j(B)}\leq \limsup_{B\rightarrow\infty}\frac{n_j(B)}{\tilde{n}_j(B)}<\infty,\quad  0<\liminf_{B\rightarrow\infty}\frac{\tau_j(B)}{\tilde{\tau}_j(B)}\leq \limsup_{B\rightarrow\infty}\frac{\tau_j(B)}{\tilde{\tau}_j(B)}<\infty.
$$

Secondly,  in Algorithm \ref{algo:GD2} we have $\tau_j(B)>0$ even when $j=J$, that is even in the last call of Algorithm \ref{algo:GD1} the optimization of the function $F_{n_j(B)}$ is stopped earlier when its gradient is small. However,  there is no reasons for stopping earlier the minimization of  $F_{n_J(B)}$ and our convergence result (Theorem \ref{thm:rate_GD_main})   remains valid when, for $j=J$, line 5 of  Algorithm \ref{algo:GD2} is replaced by
\begin{quote}
5:\hspace*{2cm} $(\hat{\theta}_{j,B}, B_j)=\mathrm{Algorithm\, \ref{algo:GD1} }\Big(\hat{\theta}_{j-1,B},\lceil \kappa B^{\gamma_j}\rceil , B_{j-1}, 0, \beta\Big)$.
\end{quote}

Nevertheless, in its current form, Algorithm \ref{algo:GD2} has the following advantage: if we let $J=J_1$ for some $J_1>1$ but, for some   $J_2\in\{1,\dots,J_1-1\}$,  the computational budget is exhausted after the $J_2$-th call of   Algorithm \ref{algo:GD1} then the estimate of $\theta_\star$ returned by Algorithm \ref{algo:GD2} is the one we would have obtained if we had chosen $J=J_2$. From a practical point of view, this means that we can choose a large value for $J$  and let the algorithm determines how many calls to Algorithm \ref{algo:GD1} can be made given the available computational resources.

\subsubsection{Convergence result\label{sub:conv}}

In fact, expanding on the latter point, the only reason why Algorithm \ref{algo:GD2} imposes an  upper bound $J$ on the number of calls of Algorithm \ref{algo:GD1} is to prove the following result:

\begin{algorithm}[t]
\begin{algorithmic}[1]
\Require Starting value $\theta_0\in\R^d$, computational budget $B\in\mathbb{N}$, rate parameters  $\alpha'\in(0,1]$ and 

$\delta\in [0,1)$, maximum number of GD-BLS calls $J\in\mathbb{N}$, sample size parameters  $\kappa\in(0,\infty)$, 

tolerance parameter $\tau\in(0,\infty)$  and backtracking parameter $\beta\in(0,1)$.
\vspace{0.2cm}

\State Let $j=0$, $B_0=B$ and $\hat{\theta}_{0,B}=\theta_0$

\While{$B_{j}>0$ and $j<J$}
\State Let  $j=j+1$ and $\gamma_j=1- \delta^j$
 
\State Let $(\hat{\theta}_{j,B}, B_j)=\mathrm{Algorithm\, \ref{algo:GD1} }\Big(\hat{\theta}_{j-1,B},\lceil \kappa B^{\gamma_j}\rceil , B_{j-1},  \tau B^{-  \frac{\alpha'}{1+ \alpha'}\gamma_j}, \beta\Big)$

\EndWhile
\State \textbf{return:} $\hat{\theta}_{j,B}$

\end{algorithmic}
\caption{\label{algo:GD2}}
\end{algorithm}

\begin{theorem}\label{thm:rate_GD_main}
Assume that \ref{assumption1}-\ref{assumption2} hold. Let $\alpha'\in(0,1]$, $\theta_0\in\R^d$,   $J\in\mathbb{N}$, $(\kappa,\tau)\in (0,\infty)^2$, $\beta\in(0,1)$, $\delta\in \big( 2\alpha'/(1+3\alpha'),1\big)$ and 
$$
\big(\hat{\theta}_{J,B}, \tilde{B}\big)=\mathrm{Algorithm\, \ref{algo:GD2} }\big(\theta_0,B,\alpha',\delta, J,\kappa,\tau,\beta\big),\quad\forall B\in\mathbb{N}.
$$
Then,  $\|\hat{\theta}_{J,B}-\theta_\star\|=\smallo_\P(1)$ and if \ref{assumption2} holds for $\alpha=\alpha'$ then
\begin{align*}
\|\hat{\theta}_{J,B}-\theta_\star\|=\bigO_\P\Big(B^{-\frac{\alpha'}{1+\alpha'}(1-\delta^J)}\Big).
\end{align*}
\end{theorem}

Assuming that \ref{assumption2} holds for $\alpha=\alpha'$, it follows from Theorem \ref{thm:rate_GD_main} that taking a small value for $J$ can only make slower the speed at which $\hat{\theta}_{J,B}$ converges to $\theta_\star$. However, as mentioned in the introductory section,  to observe in practice the $\bigO\big(B^{-\frac{\alpha'}{1+\alpha'}(1-\delta^J)}\big)$ convergence rate it is needed to minimize the function  $F_{n_J(B)}$, which may not be the case if $J$ is too large given the available computational budget. For this reason, we expect the observed convergence rate to be  $B^{-\frac{\alpha'}{1+\alpha'}(1-\delta^{J_{B}})}$, with $J_{B}$ the number of calls of Algorithm \ref{algo:GD1} needed for computing $\hat{\theta}_{J,B}$, that is with  $J_{B}$ the smallest value of $j\in\{1,\dots,J\}$ such that $\hat{\theta}_{J_{B},B}=\hat{\theta}_{J,B}$. A sensible approach is therefore to let the algorithm decides how many calls to Algorithm \ref{algo:GD1} are needed by choosing    $J$ large enough so that   $J_{B}<J$ with probability close to one. The numerical experiments presented in the next section suggest that $J_{B}$ increases with $B$ at speed $\log(B)$, and verifying theoretically this observation is left for future research.

Noting that, for any $j\geq 1$, increasing $\delta$ makes the $j$-th call of  Algorithm \ref{algo:GD1} computationally less expensive, it follows that $J_{B}$    increases with $\delta$. Hence, the larger $\delta$ the slower is the convergence rate of $\hat{\theta}_{j,B}$ for any given $j\geq 1$, but the larger the number of calls  to  Algorithm \ref{algo:GD1} can be made.  The numerical experiments of Section \ref{sec:example} suggest that the choice of $\delta$ has little impact on the convergence behaviour of  $\hat{\theta}_{J,B}$.

Finally, if   \ref{assumption2} does not hold for $\alpha=\alpha'$ then the proof of Theorem \ref{thm:rate_GD_main} shows that  $\|\hat{\theta}_{J,B}-\theta_\star\|$ converges to zero at least at speed $B^{-\frac{\alpha}{1+\alpha}(1-\delta)}$, for any $\alpha$ such that  \ref{assumption2} holds.  Studying more precisely the convergence rate of $\hat{\theta}_{J,B}$ in this context is left for future research, but we expect this rate to be slower than the one obtained  when \ref{assumption2} holds for the chosen value of $\alpha'$. To understand why, remark that in Algorithm \ref{algo:GD2} the parameter $\alpha'$ only impacts the value of the tolerance parameter used in each call of Algorithm \ref{algo:GD1}, and that the larger $\alpha'$ the smaller the tolerance parameter. Then, if $\alpha'$ is too large (i.e.~if \ref{assumption2} does not hold for $\alpha=\alpha'$), for any $j\in\{1,\dots,J\}$ too much computational budget is allocated to the $j$-th call of Algorithm \ref{algo:GD1}, resulting in a waste of computational resources and thus in a slower convergence rate as $B\rightarrow\infty$. In particular, it may be the case that even for large values of $B$ we have $J_B<J$ with high probability.

\subsubsection{Understanding  where the fast convergence rate comes from\label{subsubsec:fast}}

To understand why Algorithm \ref{algo:GD2} allows to significantly improve upon the convergence rate obtained in Section \ref{sub:GD_rate}, consider the set-up of Theorem \ref{thm:rate_GD} with the fixed starting value $\theta_0$ replaced by a sequence $(\theta_{0,B})_{B\geq 1}$ of starting values such that $\|\theta_{0,B}-\theta_\star\|=\bigO(B^{-c})$ for some constant $c\geq 0$, that is let
$$
(\tilde{\theta}_{c,B}, \tilde{B})=\mathrm{Algorithm\, \ref{algo:GD1} }\big(\theta_{0,B},n_c(B), B, 0, \beta\big),\quad\forall B\in\mathbb{N}
$$
for some   function $n_c$ verifying the assumptions of Theorem \ref{thm:rate_GD}. In this case, our computations show that
\begin{align}\label{eq:crate}
\|\tilde{\theta}_{c,B}-\theta_\star\|=\bigO_\P(r_c(B)),\quad r_c(B)=\max\Big(B^{-c}\sqrt{n_c(B)/B}, n_c(B)^{-\frac{\alpha}{1+\alpha}}\Big)
\end{align}
where, as in Section \ref{sub:GD_rate} and for some $\Theta\subset\R^d$, the first term in the definition of $r_c(B)$ is the numerical error  $\|\tilde{\theta}_{c,B}-\hat{\theta}_{\Theta,n_c(B)}\|$ and the second one is the statistical error $\|\hat{\theta}_{\Theta,n_c(B)}-\theta_\star\|$. From \eqref{eq:crate} we observe that the value of $c$  modifies the  trade-off computational error versus statistical error we are facing when choosing $n_c(B)$, and these two sources of errors converge  to zero at the same speed when 
\begin{align}\label{eq:cond_nj}
0<\liminf_{B\rightarrow\infty}\frac{n_c(B)}{B^{\frac{(1+\alpha)(1+2c)}{1+3\alpha}}}\leq \limsup_{B\rightarrow\infty}\frac{n_c(B)}{B^{\frac{(1+\alpha)(1+2c)}{1+3\alpha}}}<\infty
\end{align}
in which case 
$$
\|\tilde{\theta}_{c,B}-\theta_\star\|=\bigO_\P\Big(B^{-\frac{\alpha(1+2c)}{1+3\alpha}}\Big).
$$
Therefore, initializing  Algorithm \ref{algo:GD1} with a starting value that converges to $\theta_\star$ (i.e.~having $c>0$) allows to improve the rate at which we learn $\theta_\star$. In particular, for $c=0$ we recover the result of Corollary \ref{cor:1step}, and the larger $c$ the faster is the convergence rate of $\tilde{\theta}_{c,B}$.

\subsubsection{Understanding the convergence rate given in Theorem \ref{thm:rate_GD_main}}

The convergence rate  given in Theorem \ref{thm:rate_GD_main} is more easily understood  by studying an alternative (but less practical, see below) estimator of $\theta_\star$.

To define this estimator let $\delta_{\alpha'}=2\alpha'/(1+3\alpha')$,   $c_j=(1-\delta_{\alpha'}^j)\alpha'/(1+\alpha')$ for all $j\geq 0$ and let $\theta_0\in\R^d$. In addition, let $J\in\mathbb{N}$ and $\{\kappa_{J,j}\}_{j=1}^J$ be such that $\sum_{j=1}^J\kappa_{J,j}=1$ and such that $\kappa_{J,j}>0$ for all $j\in\{1,\dots,J\}$. Finally, for all $j\in \{1,\dots,J\}$ and using the convention that $\tilde{\theta}_{j-1,B,J}=\theta_0$ for all $B\in\mathbb{N}$ when $j=1$, let
$$
(\tilde{\theta}_{j,B,J}, \tilde{B}_j)=\mathrm{Algorithm\, \ref{algo:GD1} }\Big(\tilde{\theta}_{j-1,B,J},n_{c_{j-1}}(B),\lfloor \kappa_{J,j} B\rfloor, 0, \beta\Big),\quad\forall B\in\mathbb{N}
$$
with the function $n_{c_{j-1}}$ such that \eqref{eq:cond_nj}  holds (with $c=c_{j-1}$). Then, assuming that \ref{assumption2} holds for $\alpha=\alpha'$ and using the   computations of Section \ref{subsubsec:fast}, it is easily checked that for all $j\in\{1,\dots,J\}$ we have $\|\tilde{\theta}_{j,B,J}-\theta_\star\|=\bigO_\P(B^{-c_j})$, where
$$
B^{-c_j}=\max\Big(B^{-c_{j-1}}\sqrt{n_{c_{j-1}}(B)/B}, n_{c_{j-1}}(B)^{-\frac{\alpha'}{1+\alpha'}}\Big),\quad\forall B\in\mathbb{N}.
$$
Informally speaking, and following the reasoning of Section \ref{subsubsec:fast}, this implies that  $\tilde{\theta}_{j,B,J}$ is computed from   $\tilde{\theta}_{j-1,B,J}$ using the number $n_{c_{j-1}}(B)$ of elements of $(Z_i)_{i\geq 1}$ that maximizes the speed at which $\|\tilde{\theta}_{j,B,J}-\theta_\star\|$ converges to zero.  Using the definition of $(c_j)_{j\geq 1}$, it is readily checked   that, for all $j\in\{1,\dots,J\}$, 
$$
\|\tilde{\theta}_{j,B,J}-\theta_\star\|=\bigO_\P\Big(B^{-\frac{ \alpha'}{1+\alpha'}(1-\delta_{\alpha'}^j)}\Big)
$$
while Theorem \ref{thm:rate_GD_main}  establishes that $\|\hat{\theta}_{j,B}-\theta_\star\|=\bigO_\P\big( B^{-\frac{\alpha'}{1+\alpha'}(1-\delta^i)}\big)$ for $\delta\in (\delta_{\alpha'},1)$.

In addition to give some insight on this latter result, the above calculations show that   $\tilde{\theta}_{j,B,J}$ converges to $\theta_\star$ (slightly) faster than   $\hat{\theta}_{j,B}$ for any $j\in\{1,\dots,J\}$. However, these alternative estimators have the important drawback to require that the computational budget  allocated to each call  of Algorithm \ref{algo:GD1}  is chosen beforehand,  and in a way that  depends on the chosen value of $J$. By contrast, in Algorithm \ref{algo:GD2} this allocation  is done automatically and does not depend on $J$, and if only $j<J$ calls to Algorithm \ref{algo:GD1} are made then the algorithm returns the estimate of $\theta_\star$ we would have obtained if we had chosen $J=j$. In our calculations, this desirable properties of  Algorithm \ref{algo:GD2} are obtained at the price of imposing that $\delta$ is strictly bigger than $\delta_{\alpha'}$.

\subsection{Validity of GD-BLS for solving noisy optimization problems: sketch of proof}

We end this section with a sketch of proof of the validity of  GD-BLS for solving noisy optimization problems, the goal being to provide the reader with some insights about how  Assumptions \ref{assumption1}-\ref{assumption2} and  the properties of the GD-BLS algorithm are used to prove our main results. Specifically, in this subsection we let $\theta_0\in\R^d$ and consider   the problem of solving the optimization problem \eqref{eq:optim_prob} by applying the GD-BLS algorithm (i.e.~Algorithms \ref{algo:GD}-\ref{algo:Back})  on the random function $F_n$, using  $\theta_0$ as starting value. We denote by $(\theta^{(n)}_t)_{t\geq 1}$ the resulting sequence of iterates and below we show, in an informal way, that the results of Sections \ref{sub:prelim_assume}-\ref{sub:prelim_gd} imply that $\|\lim_{t\rightarrow\infty}\theta_t^{(n)}- \theta_\star\|=\smallo_\P(1)$.

To this aim remark first that, by using the descent property of GD-BLS, all the elements of $(\theta^{(n)}_t)_{t\geq 1}$  belong to the (random) set $\tilde{\Theta}_{0,n}:=\{\theta\in\R^d:\, F_n(\theta)\leq F_n(\theta_0)\}$. Next, by using Lemma \ref{lemma:uniform_conv} and under \ref{assumption1}, it can be proved  that there exists a compact and convex set $\tilde{\Theta}_0$ such that   $\tilde{\Theta}_{0,n}\subset \tilde{\Theta}_0$ with $\P$-probability tending to one. To proceed further we let 
\begin{align*}
&\Theta_{0,n}=\Big\{\theta\in\R^d:\,\exists \theta'\in \tilde{\Theta}_{0}\text{ such that }\|\theta-\theta'\|\leq \sup_{\tilde{\theta}\in \tilde{\Theta}_{0}}\|\nabla   F_n(\tilde{\theta})\|\Big\}\\
&\Theta_0=\Big\{\theta\in\R^d:\,\exists \theta'\in \tilde{\Theta}_0\text{ such that }\|\theta-\theta'\|\leq \sup_{\tilde{\theta}\in \tilde{\Theta}_0}\| \nabla   F(\tilde{\theta})\|+1\Big\}
\end{align*}
and remark that, by Lemma   \ref{lemma:uniform_conv}, with $\P$-probability tending to one we have $\Theta_{0,n}\subset \Theta_0$. Under \ref{assumption1} the set $\Theta_0$ is compact and convex  and thus, by   Lemma  \ref{lemma:assume1} and Corollary \ref{cor:convex}, it follows that with $\P$-probability tending to one the function $F_n$ is strictly convex and $L$-smooth on $\Theta_0$ (with $L\in(0,\infty)$ a finite constant). Without loss of generality we can assume that $\Theta_0$ is sufficiently large so that $\theta_\star$ belongs to the interior of this set, in which case   $\|\hat{\theta}_{\Theta_0,n}-\theta_\star\|=\smallo_\P(1)$ by Lemma \ref{lemma:MLE}. To sum-up, with $\P$-probability tending to one (i) the function $F_n$ is strictly convex and $L$-smooth on the compact and convex set $\Theta_{0}$, (ii)  $ F_n(\hat{\theta}_{\Theta_0,n})=\inf_{\theta\in\Theta_0}F_n(\theta)$ and (iii)  the sequence of iterates $(\theta^{(n)}_t)_{t\geq 1}$ generated by the GD-BLS algorithm applied on $F_n$ is a sequence in $\tilde{\Theta}_0$. Then, by using Propositions \ref{prop:GD_result_0}-\ref{prop:GD_result}, it readily follows  that  $\|\lim_{t\rightarrow\infty}\theta_t^{(n)}- \hat{\theta}_{\Theta_0,n}\|=\smallo_\P(1)$, and since $\|\hat{\theta}_{\Theta_0,n}-\theta_\star\|=\smallo_\P(1)$   we can conclude that   $\|\lim_{t\rightarrow\infty}\theta_t^{(n)}- \theta_\star\|=\smallo_\P(1)$.

\section{Numerical experiments\label{sec:example}}

In this section we study the behaviour of $\E[\|\hat{\theta}_{J,B}-\theta_\star\|]$ and of $\E[J_{B}]$ as $B$ increases, as estimated from 100 independent realizations of $(Z_i)_{i\geq 1}$. In all the experiments presented below,  Algorithm \ref{algo:GD2} is implemented with $\beta=0.5$ and with $C_{\mathrm{grad}}=C_{\mathrm{eval}}=\kappa=\tau=1$, and with   $\lceil\kappa B^{\gamma_j}\rceil$ replaced by $\max(100, \lceil\kappa B^{\gamma_j}\rceil)$ for all $j\geq 1$. It is trivial that Theorem \ref{thm:rate_GD_main} remains valid with this  modification of the algorithm, which is introduced to prevent to minimize  $F_{n_j(B)}$ for a value of $n_j(B)$ that is  too small. In addition, following the discussion of  Section \ref{sec:GD2},   we let $J=10^4$ so that  $J_{B}<J$ with probability close to one. Finally, in what follows we denote by $\rho_{J_B}$ the coefficient  of correlation between the variables $(\E[J_{B}])_{B\geq 1}$ and $(\log(B))_{B\geq 1}$, as     estimated from 100 independent realizations of $(Z_i)_{i\geq 1}$. Remark that a value of $\rho_{J_B}$ close to one implies that $\E[J_{B}]\approx a+b\log(B)$ for some constants $a\in\R$ and $b\in(0,\infty)$.

\subsection{Example of Proposition \ref{prop:Poisson}\label{sub:ex1}}

We first consider the univariate optimization problem defined in Proposition \ref{prop:Poisson}, which we recall amounts to solving the $d=1$ dimensional noisy optimization problem \eqref{eq:optim_prob} with  $Z=(X,Y)$, where $X$ and $Y$ are two independent Poisson random variables such that $\E[X]=\E[Y]=1$, and with the function $f:\R^d\times\setZ\rightarrow\R$  defined by
$$
f(\theta,z)=-z_2z_1\theta+\exp(\theta z_1),\quad \theta\in\R,\quad z\in\setZ:=\R^2.
$$
As shown above in Proposition \ref{prop:Poisson3},  for this optimization problem Assumptions \ref{assumption1}-\ref{assumption2} hold with $\alpha=1$    and, to the best of our knowledge, SG is not guaranteed to converge  when used to solve it (see Propositions \ref{prop:Poisson}-\ref{prop:Poisson22}). Simple computations show that the global minimizer of the resulting objective function $F$ is $\theta_\star=0$. For this example we let $\theta_0=1$ and  $\alpha'=1$, in which case Theorem \ref{thm:rate_GD_main} imposes to have $\delta\in (0.5,1)$.

\begin{figure}
\centering
\includegraphics[scale=0.25]{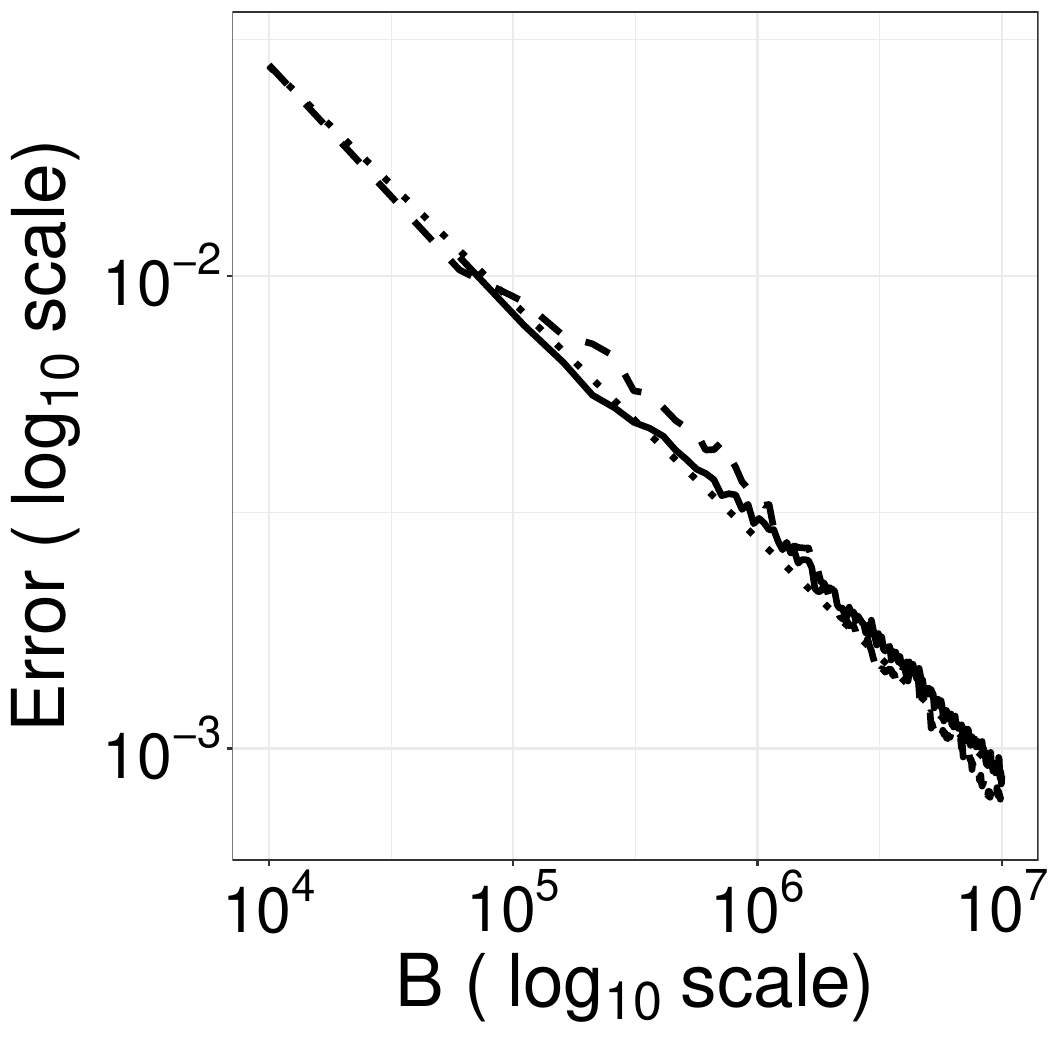} \includegraphics[scale=0.25]{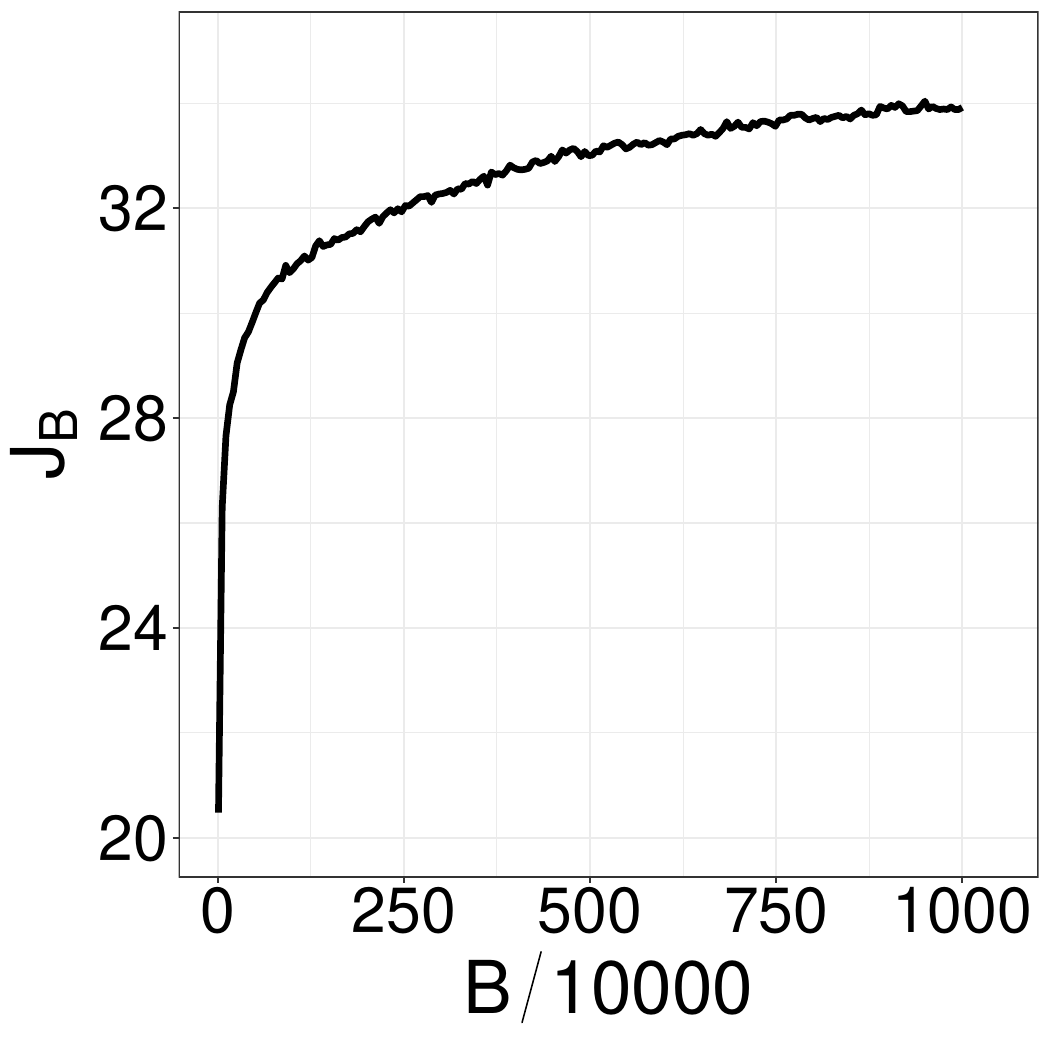} \includegraphics[scale=0.25]{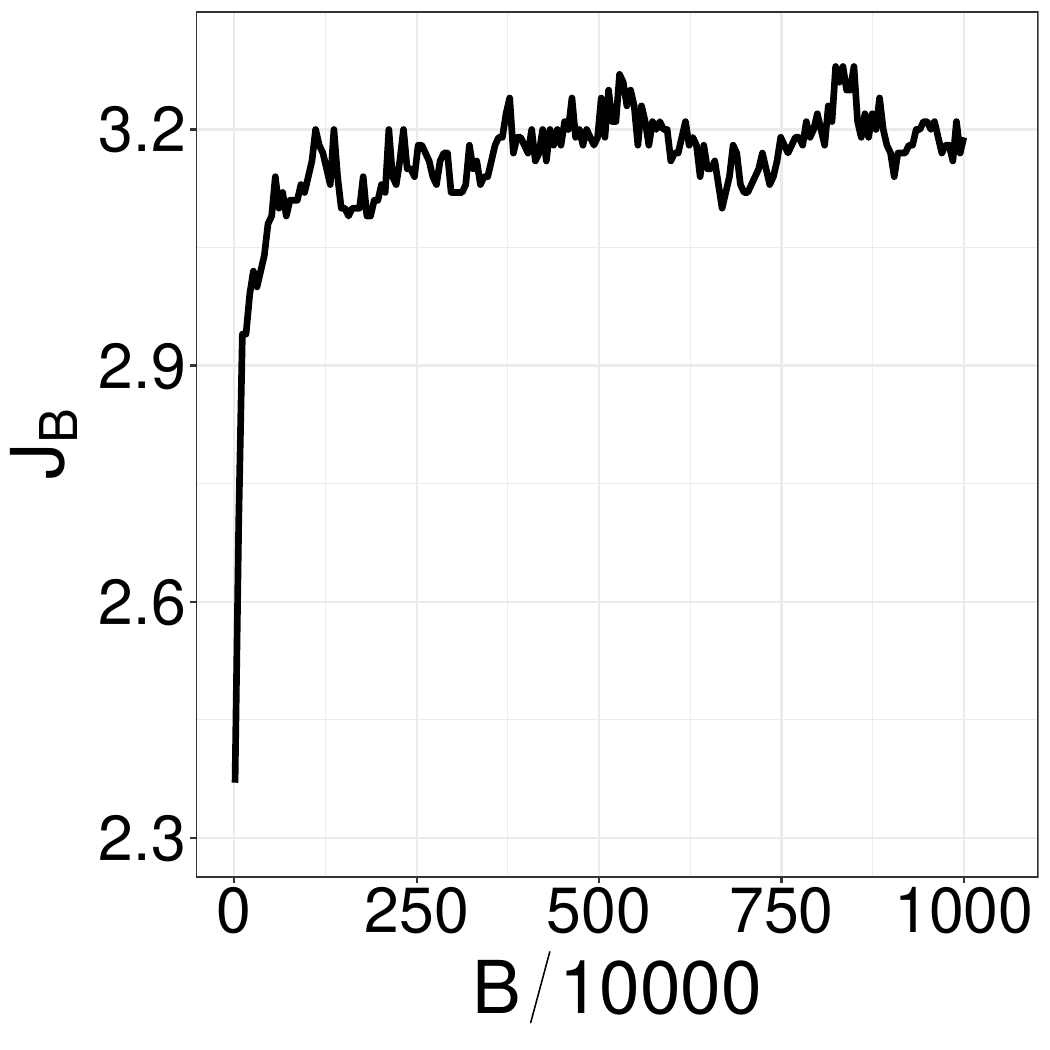} 
\caption{Results for the example of Section \ref{sub:ex1}. The left plot shows $\E[\|\hat{\theta}_{J,B}-\theta_\star\|]$ as a function of $B$, where the solid line is for $\delta=0.95$ and the dashed line for $\delta=0.51$, and with the two dotted  line representing the $B^{-1/2}$ convergence rate. The middle and right plots show the evolution of $\E[J_{B}]$  as a function of $B$ for $\delta=0.95$ (middle plot)  and  for $\delta=0.51$ (right plot). All the results are obtained from 100 independent  realizations of $(Z_i)_{i\geq 1}$. \label{Figure:example1}}
\end{figure}

Figure \ref{Figure:example1} shows,  for $\delta=0.51$ and for $\delta=0.95$, the evolution of $\E[\|\hat{\theta}_{J,B}-\theta_\star\|]$ and of $\E[J_{B}]$ as  $B$ increases. For the two considered values of $\delta$  the observed   rate at which $\E[\|\hat{\theta}_{J,B}-\theta_\star\|]$  decreases to zero is  indistinguishable from the  optimal $B^{-1/2}$ convergence rate. In addition, in  this experiment the value of $\delta$ appears to have little impact on $\E[\|\hat{\theta}_{J,B}-\theta_\star\|]$. As expected, we see in Figure \ref{Figure:example1} that the number of calls to Algorithm \ref{algo:GD1} is much larger for $\delta=0.95$ than for $\delta=0.51$, with a value of $\E[J_{B}]$ ranging from 20.48 to 34.04 for the former choice of $\delta$ and  ranging from 2.37 to 3.28 for the latter. Visually, we observe (at least for $\delta=0.95$) that  $\E[J_{B}]$ grows with $B$ as $\log (B)$. The coefficient of correlation between the variables $(\E[J_{B}])_{B\geq 1}$ and $(\log (B))_{B\geq 1}$ confirms this observation,  since for this experiment we have  $\rho_{J_B}=0.99$ when   $\delta=0.95$ and  $\rho_{J_B}=0.78$ when $\delta=0.51$.

\subsection{A multivariate extension of the example of Proposition \ref{prop:Poisson}\label{sub:ex2}}

We now consider  the $d=20$ dimensional extension of the example of Proposition \ref{prop:Poisson} obtained by letting, in  \eqref{eq:optim_prob},    $Z=(W,X,Y)$ be such that 
$$
W\sim\mathrm{Poisson}(1),\quad X|W\sim\mathcal{U}([-1,1]^{d-1}),\quad Y|(W,X)\sim \mathrm{Poisson}(a^\top _\star X)
$$
for some $a_\star\in\R^{d-1}$, and  the function $f:\R^d\times\setZ\rightarrow\R$  be defined by
$$
f(\theta,z)=-z_2\theta^\top z_1+\exp(\theta^\top z_1),\quad \theta\in\R^d,\quad z_1\in\R^d,\quad z_2\in\R,\quad   z=(z_1,z_2)\in\setZ:=\R^{d+1}.
$$

It is easily checked that the resulting objective function $F$ is such that \ref{assumption1}-\ref{assumption2} hold with $\alpha=1$ and that $\theta_\star=(0,a_\star)$.  Compared to the example of Proposition \ref{prop:Poisson}, we expect that by increasing the dimension of the search space we make the optimization problem more challenging and thus that the conclusions of Proposition \ref{prop:Poisson}-\ref{prop:Poisson22} hold  for this example. For this experiment the value $a_\star$ is a random draw  from the $\mathcal{N}_{d-1}(0,\ind_{d-1})$ distribution and, as in the previous subsection, the results are presented for $(\alpha',\delta)=(1,0.51)$ and for $(\alpha',\delta)=(1,0.95)$. Finally, Algorithm \ref{algo:GD2} is implemented using $\theta_0=(1,\dots,1)$ as starting value.

From Figure \ref{Figure:example2}, which shows the evolution of $\E[\|\hat{\theta}_{J,B}-\theta_\star\|]$ and of $\E[J_{B}]$ as  $B$ increases, we see that the behaviour of these two quantities is similar to what we observed in the previous subsection. The main difference is however that, because of the larger dimension of $\theta$, the $B^{-1/2}$ convergence rate for $\hat{\theta}_{J,B}$ is observed only for $B$ large enough (roughly for $B>10^{6.5}$). The study of the dependence of the performance of Algorithm \ref{algo:GD2} to the dimension $d$ of $\theta$ is left for future research. Finally,  in Figure \ref{Figure:example2} we also observe that $\E[J_{B}]$ appears to increase  with $B$ as $\log(B)$,  and indeed,  for this example, the coefficient of correlation between $(\E[J_{B}])_{B\geq 1}$ and $(\log(B))_{B\geq 1}$  is $\rho_{J_B}=0.98$ for   $\delta=0.95$ and  $\rho_{J_B}=0.70$ for $\delta=0.51$.

\begin{figure}
\centering
\includegraphics[scale=0.25]{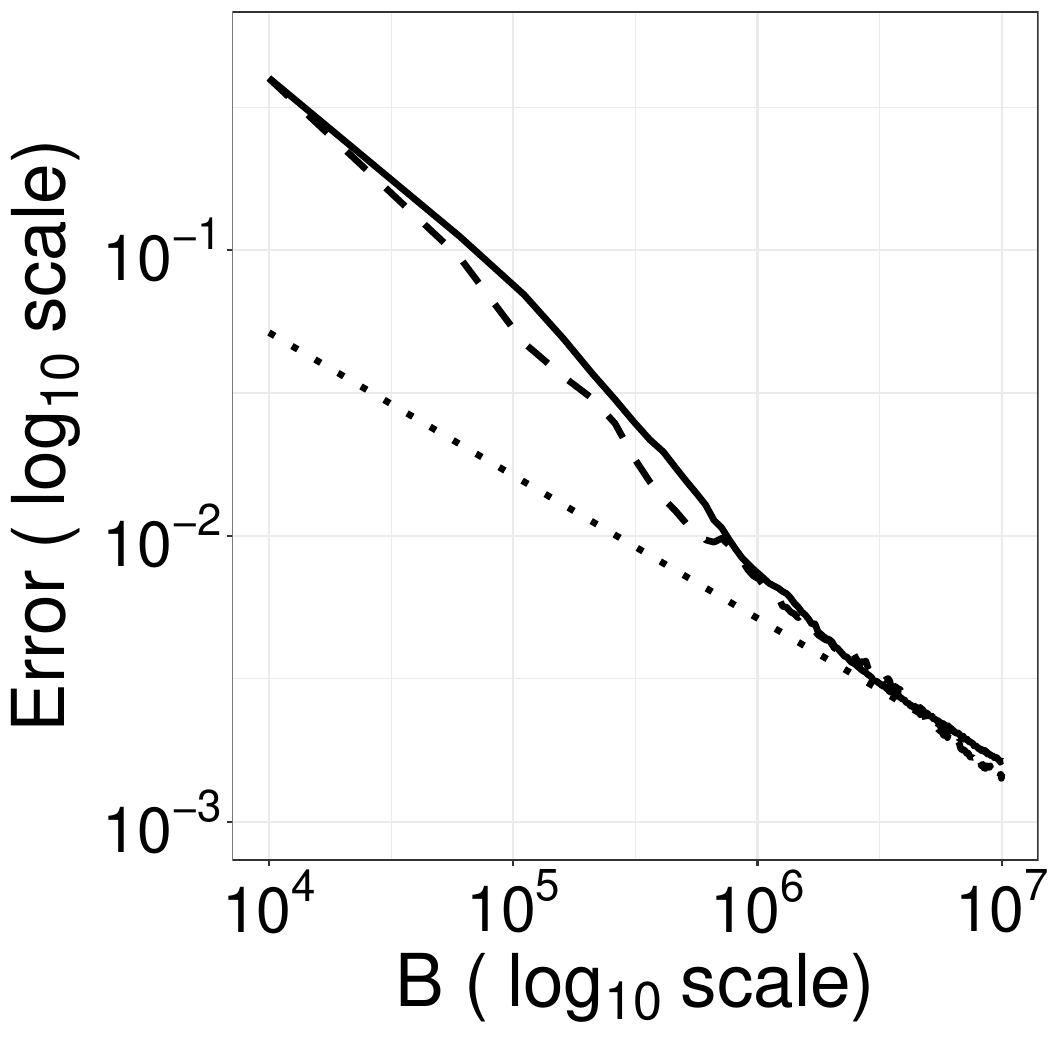}\includegraphics[scale=0.25]{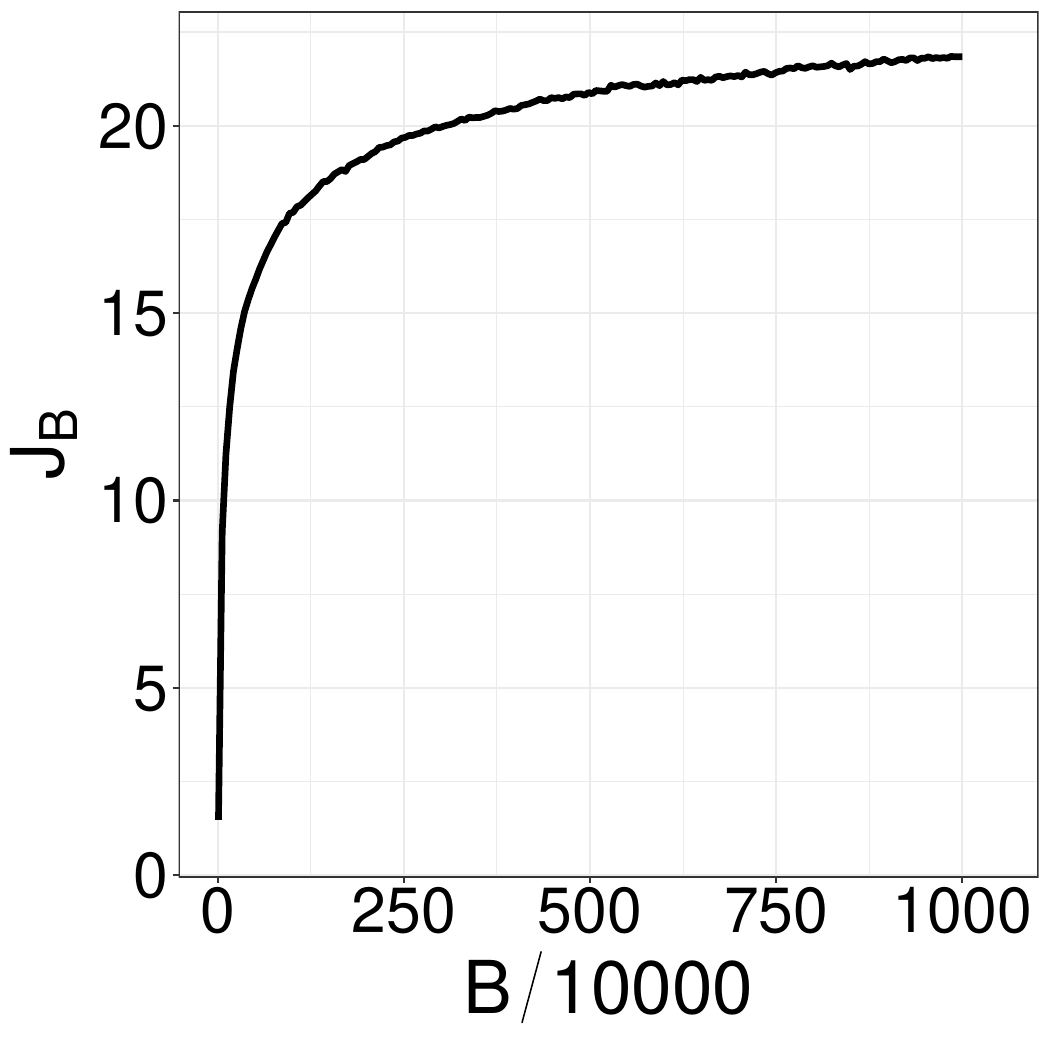}\includegraphics[scale=0.25]{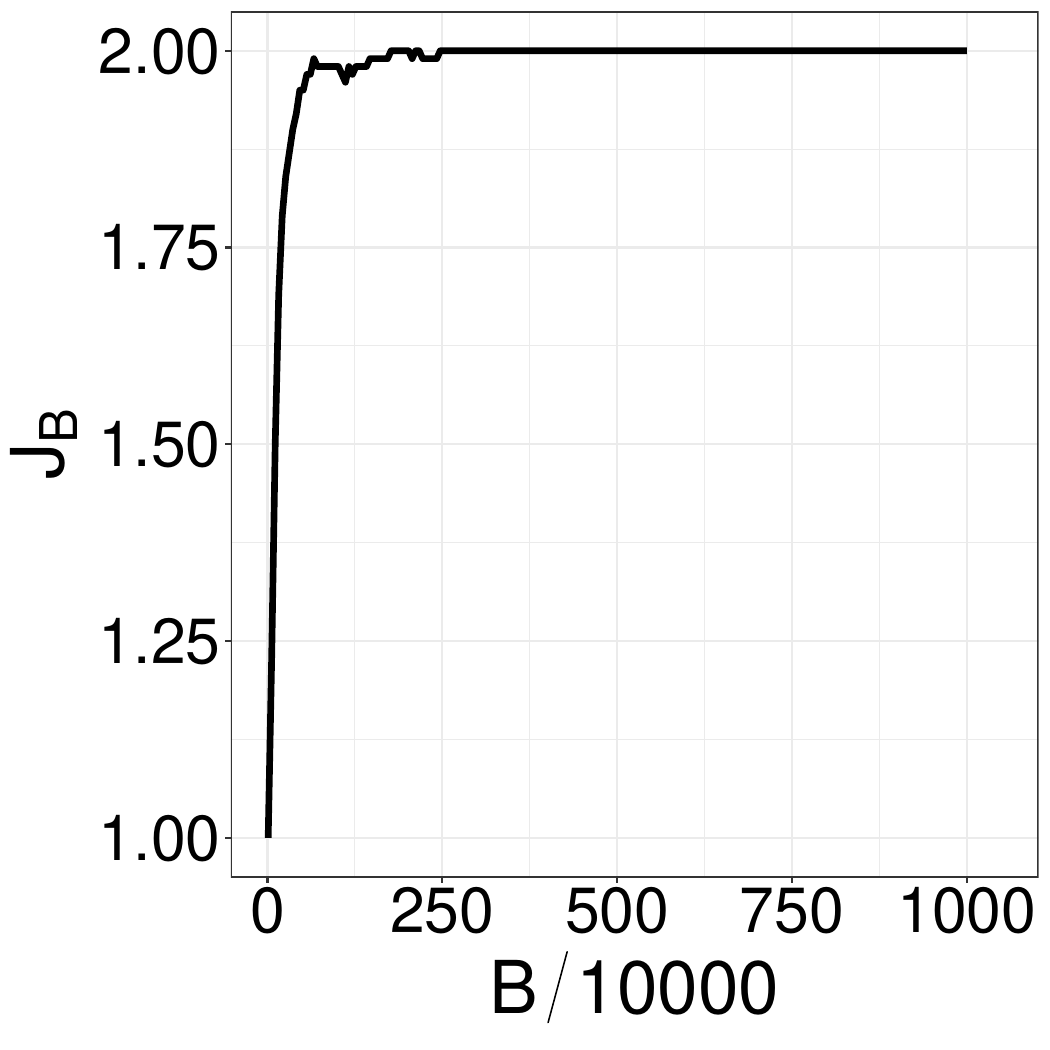} 
\caption{Results for the example of Section \ref{sub:ex2}.  The left plot shows $\E[\|\hat{\theta}_{J,B}-\theta_\star\|]$ as a function of $B$, where the solid line is for $\delta=0.95$ and the dashed line for $\delta=0.51$, and with the  dotted  line  representing the $B^{-1/2}$ convergence rate. The middle and right plots show the evolution of $\E[J_{B}]$  as a function of $B$ for $\delta=0.95$ (middle plot)  and  for $\delta=0.51$ (right plot). All the results are obtained from 100 independent realizations of $(Z_i)_{i\geq 1}$ \label{Figure:example2}}
\end{figure}

\begin{figure}[t]
\centering
\includegraphics[scale=0.25]{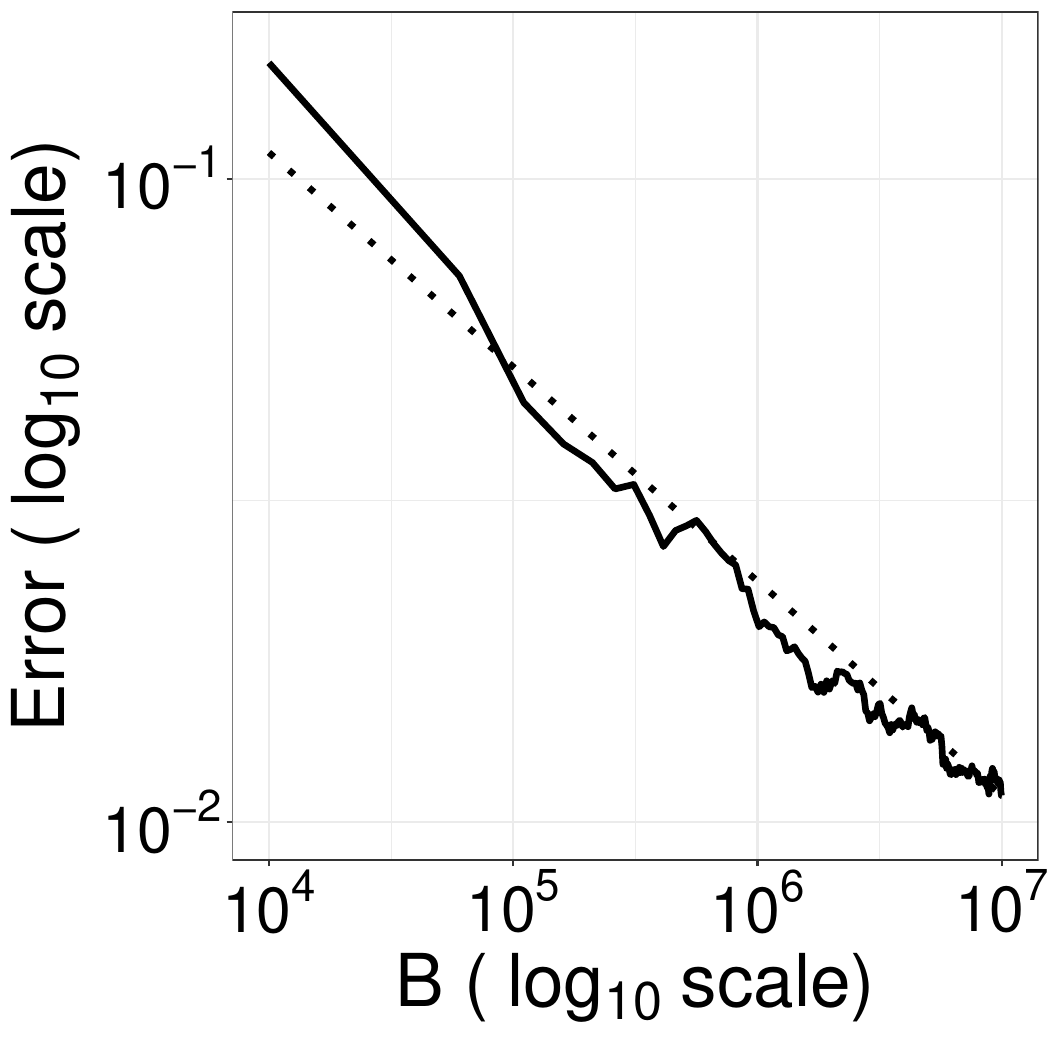}\includegraphics[scale=0.25]{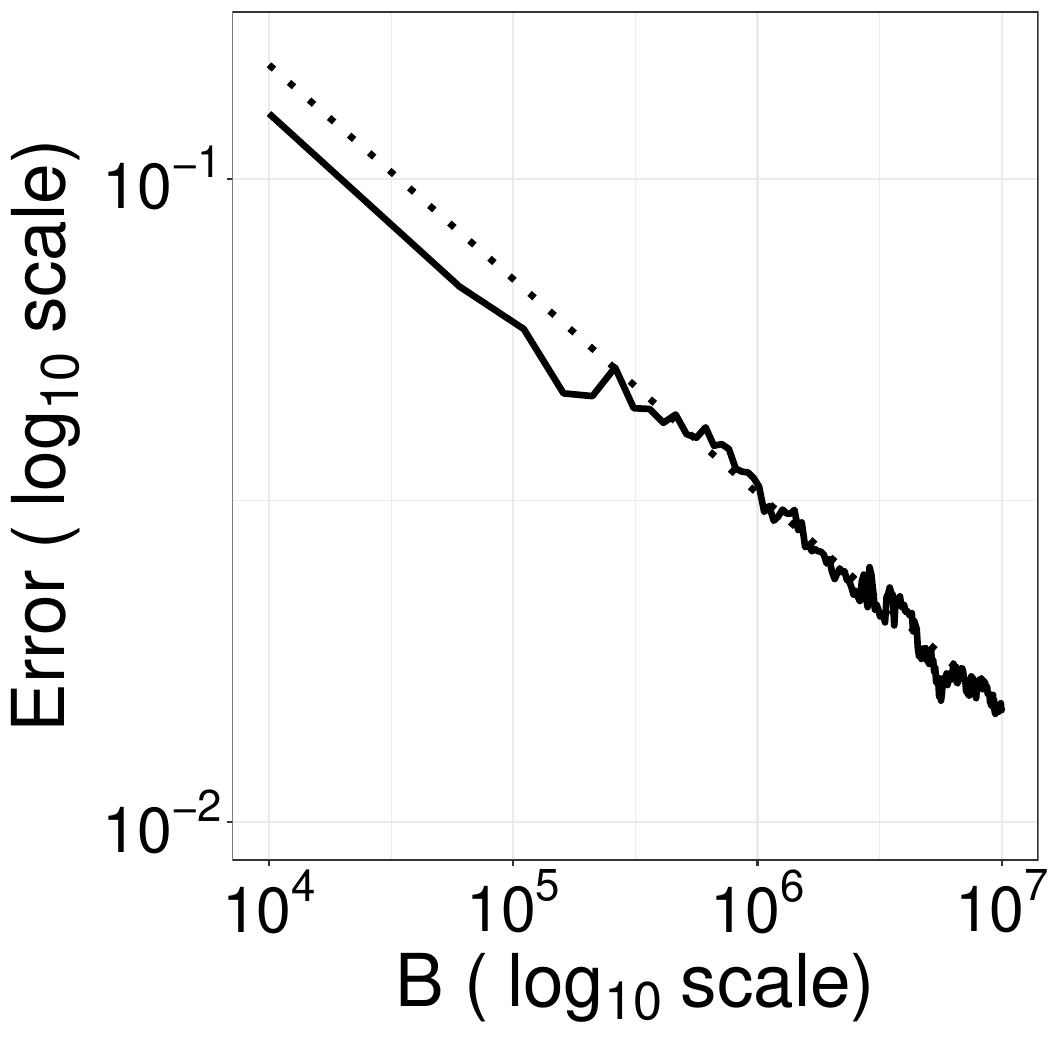}\includegraphics[scale=0.25]{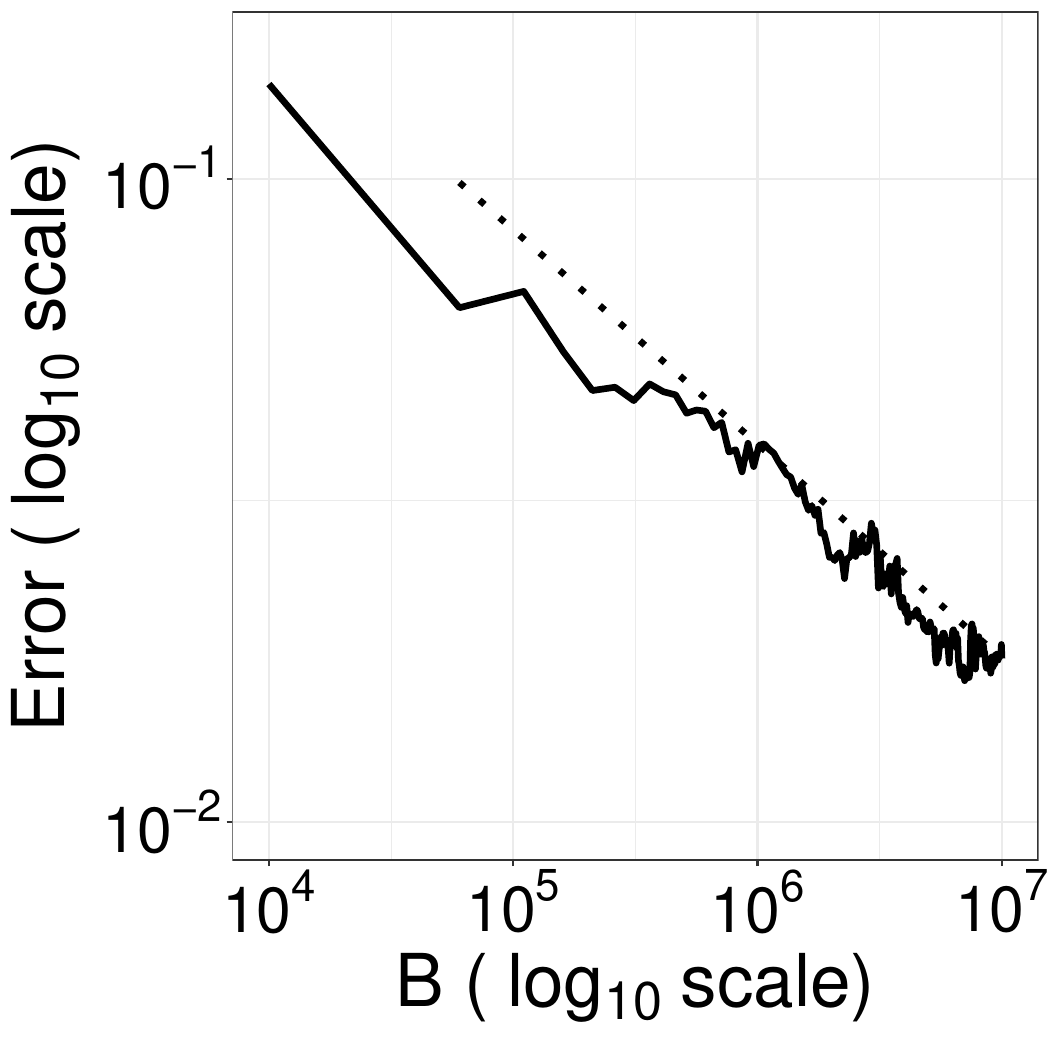} 

\includegraphics[scale=0.25]{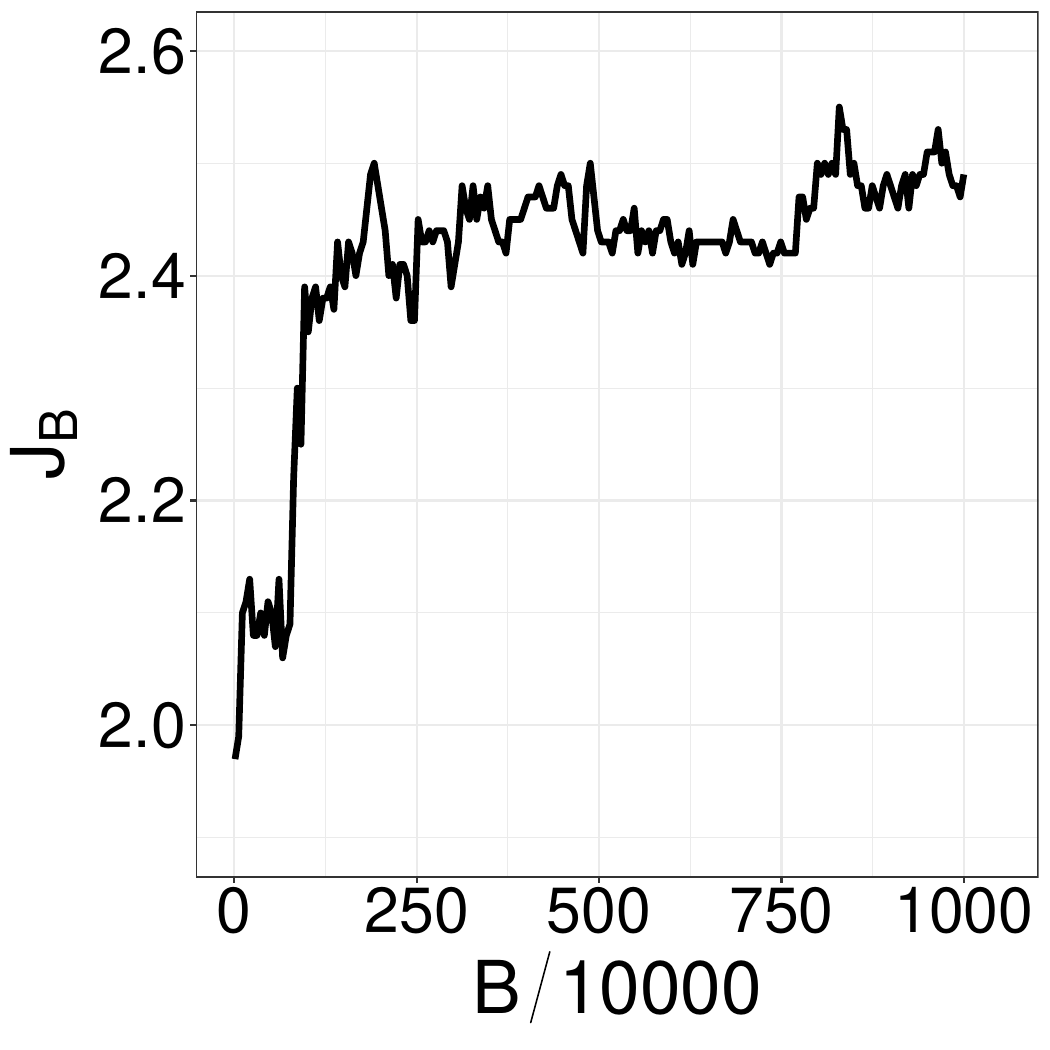}\includegraphics[scale=0.25]{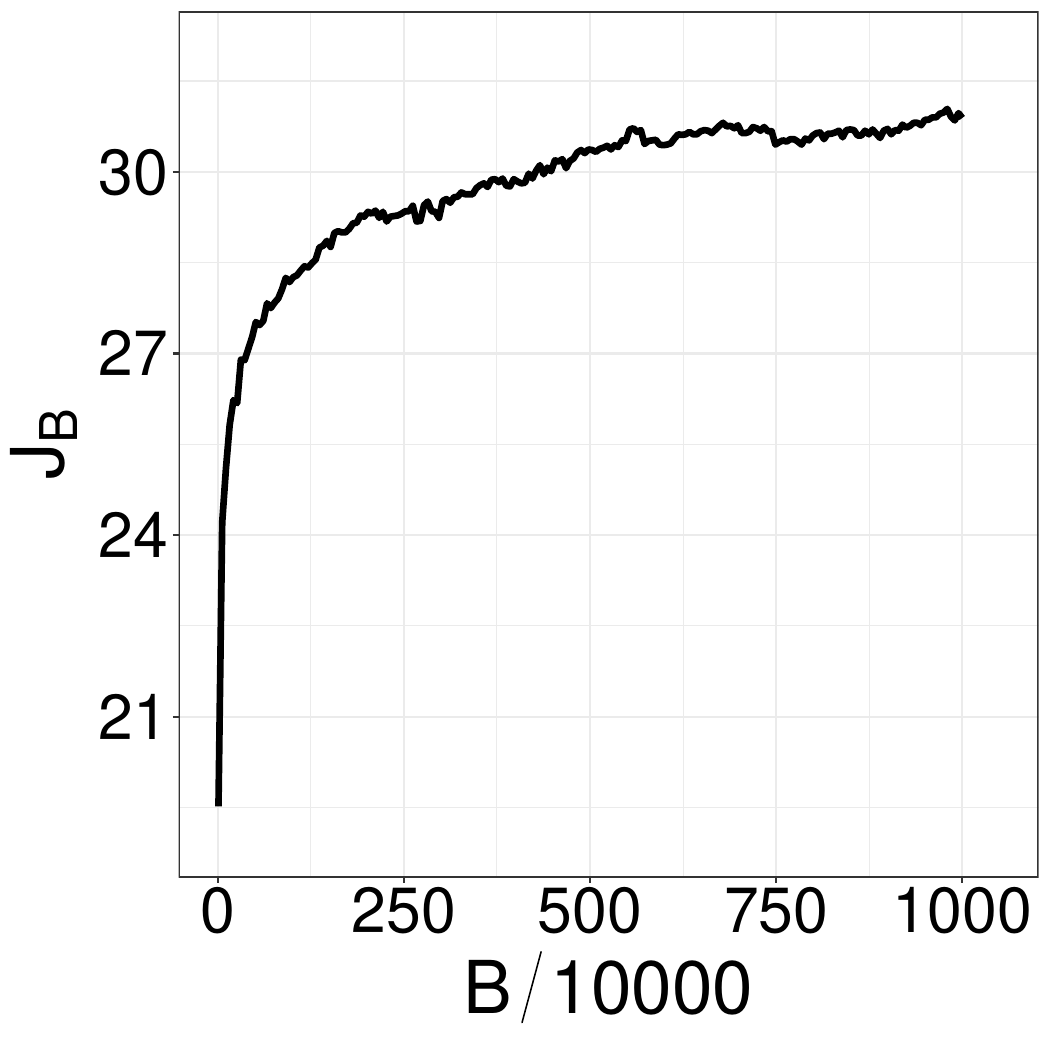} \includegraphics[scale=0.25]{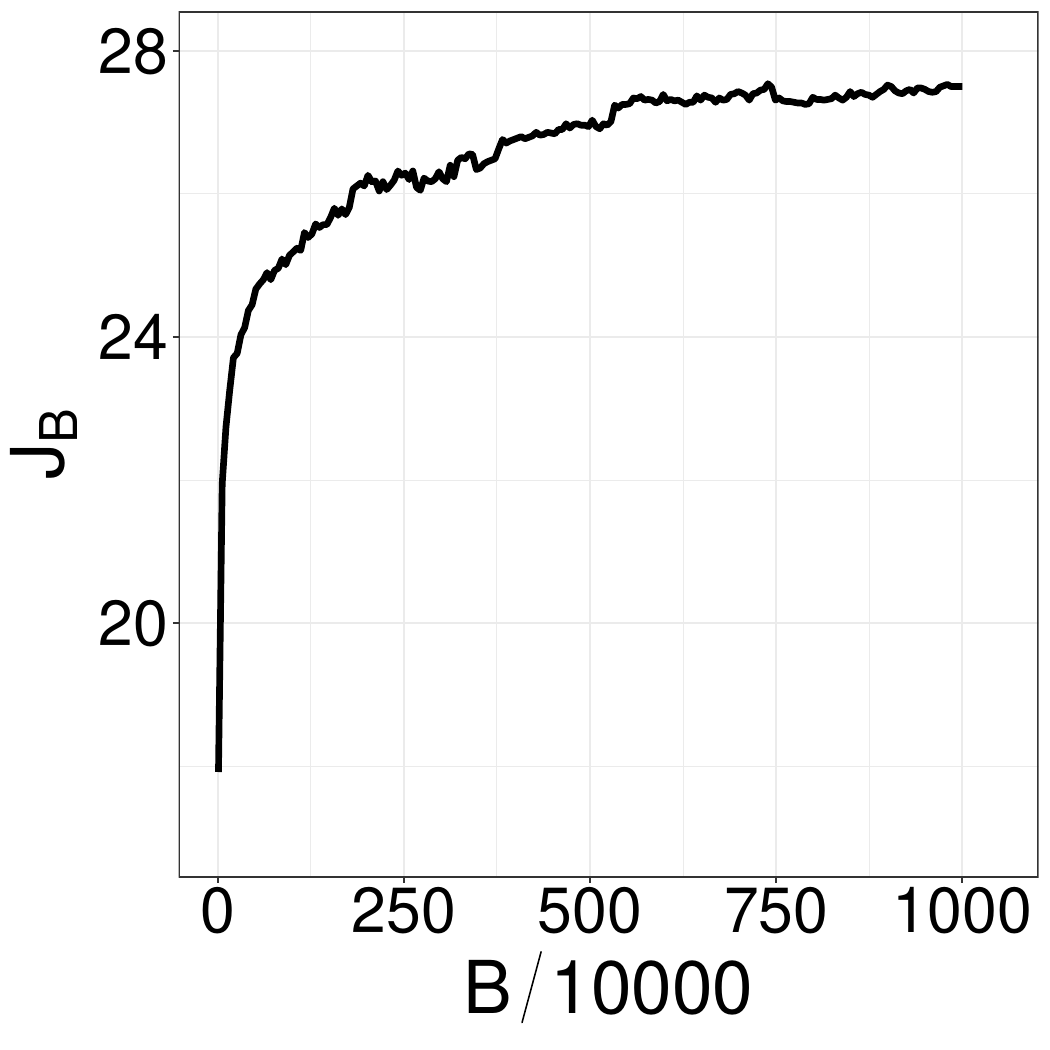} 
\caption{Results for the example of Section \ref{sub:ex3}.  The top plots shows $\E[\|\hat{\theta}_{J,B}-\theta_\star\|_{0.1}]$ as a function of $B$ and  the bottom   plots  show    $\E[J_{B}]$ as a function of $B$. The left plots are for $(\alpha',\delta)=(0,5,0.41)$, the middle plots  for $(\alpha',\delta)=(0.5,0.95)$ and the right plots for $(\alpha',\delta)=(1,0.95)$. In the top plots the dotted  lines  represent  the   $B^{-1/3}$ convergence rate and all the results are obtained  from 100 independent realizations of $(Z_i)_{i\geq 1}$ \label{Figure:example3}}
\end{figure}

\subsection{Example of Proposition \ref{prop:Poisson} with  an extra heavy tailed noise\label{sub:ex3}}

We finally consider the example of Proposition \ref{prop:Poisson} with an extra heavy tailed noise. More precisely, we aim at solving the $d=1$  noisy optimization problem \eqref{eq:optim_prob} with  $Z=(W, X,Y)$, where $W\sim t_\nu(0,1)$ and where $X$ and $Y$ are  two independent Poisson random variables which are independent  of $W$ and such that $\E[X]=\E[Y]=1$,  and with the function $f:\R^d\times\setZ\rightarrow\R$  defined by
$$
f(\theta,z)=-z_3z_2\theta+\exp(\theta z_2)+z_1\theta,\quad \theta\in\R,\quad z\in\setZ:=\R^3.
$$

We let $\nu=1.501$ so that \ref{assumption2} holds for any $\alpha<\nu-1=0.501$. Remark that since $\E[W]=0$ it follows that the resulting objective function  $F$ is the same as for the example of Proposition \ref{prop:Poisson}, and thus $\theta_\star=0$. It is also immediate to see that the conclusions of Proposition \ref{prop:Poisson}-\ref{prop:Poisson22} remain  valid for this example, and therefore to be best of our knowledge SG is not guaranteed to converge when applied to solve it. In addition, the results of \citet{nguyen2024improved}  for clipped SG do not apply to this optimization problem, notably because the function $F$ is not $L$-smooth (see Proposition \ref{prop:Poisson}). For this example we implement Algorithm \ref{algo:GD2}   with $\theta_0=1$ and the results are presented for  $\alpha'=0.5$ and for $\alpha'=1$. Remark that \ref{assumption2} holds with $\alpha=\alpha'$ for the former value of $\alpha'$ but not for the latter. Remark also that, by Theorem \ref{thm:rate_GD_main},   we must have $\delta\in (0.4,1)$ when $\alpha'=0.5$ and $\delta\in(0.5,1)$ when $\alpha'=1$.

Due to the heavy tailed noise, for this example it is not guaranteed that we have $\E[\|\hat{\theta}_{J,B}\|]<\infty$ and the numerical experiments suggest that this may indeed not be the case. For this reason, below we study, in Figure \ref{Figure:example3}, the behaviour  as $B$ increases   of $\E[\|\hat{\theta}_{J,B}-\theta_\star\|_{0.1}]$, the expected 10\% trimmed mean of  $\|\hat{\theta}_{J,B}-\theta_\star\|$.

The plots in the first two columns of Figure \ref{Figure:example3} are obtained for $(\alpha',\delta)=(0.5,0.41)$ and for $(\alpha',\delta)=(0.5,0.95)$. From these plots  we see that  the convergence rate of $\hat{\theta}_{J,B}$ is undistinguishable from the $B^{-\frac{\alpha'}{1+\alpha'}}=B^{-1/3}$ convergence rate, confirming the result of Theorem \ref{thm:rate_GD_main}. As for the examples of Sections \ref{sub:ex1}-\ref{sub:ex2}, the choice of $\delta$ appears to have only a limited impact on the estimation error. We also observe that for the two considered values  of $\delta$ the quantity  $\E[J_{B}]$ seems to grow with $B$ as $\log (B)$, an observation which is confirmed  by the fact that $\rho_{J_B}= 0.85$ for $\delta=0.41$ and  $\rho_{J_B}=0.98$ for  $\delta=0.95$. The plots in last column  of Figure \ref{Figure:example3} study the behaviour of $\hat{\theta}_{J,B}$ when $(\alpha',\delta)=(1,0.95)$, in which case   \ref{assumption2} does not hold for $\alpha=\alpha'$. If    $\hat{\theta}_{J,B}$ appears to roughly converge to $\theta_\star$ at speed $B^{-1/3}$, we observe that the convergence becomes slower for values of $B$ greater than $10^{6.7}$. Remark that, from the   discussion of Section \ref{sub:conv}, it was expected that the convergence rate of $\hat{\theta}_{J,B}$ is slower when \ref{assumption2} does not hold for the chosen value of $\alpha'$ than when it does. Finally, for this value of $(\alpha',\delta)$ the quantity   $\E[J_{B}]$ also seems to grow as $\log (B)$, an observation which is confirmed by the large value $\rho_{J_B}=0.98$ obtained for the coefficient of correlation between $(\E[J_{B}])_{B\geq 1}$ and $(\log (B))_{B\geq 1}$.

\section{Concluding remarks\label{sec:conclusion}}

In this paper we show, theoretically and empirically, that gradient descent with backtracking line search can be used to solve noisy optimization problems with a nearly optimal convergence rate, and without assuming that the objective function is globally $L$-smooth.

The proposed algorithm (Algorithm \ref{algo:GD2}) requires that GD-BLS can be called at most $J$ times, with $J$ a user-specified parameter. We conjecture that this condition can be removed and, assuming that \ref{assumption2} holds for $\alpha=\alpha'$, that the resulting estimator $\hat{\theta}_B$ of the global  minimiser $\theta_\star$  converges at speed $B^{-\frac{\alpha'}{1+\alpha'}(1-j_B)}$ for some sequence $(j_B)_{B\geq 1}$ such that $\lim_{B\rightarrow\infty}j_B=\infty$. Based on our numerical experiments, we predict that $j_B=a+b\log(B)$ for some constants $a\in[0,\infty)$ and $b\in(0,\infty)$ depending on the choice of   $\delta$. Confirming these results theoretically is left for future research.

Another interesting question that deserves further work  is the possibility to replace, in  Algorithm \ref{algo:GD2},   GD-BLS   by   GD-BLS with  Nesterov momentum (GD-BLS-NM). We recall the reader that, while the convergence of the former algorithm for computing $\min_{x\in \R^d}g(x)$ for a convex and $L$-smooth function $g$ is $1/T$, with $T$ the number of iterations,    GD-BLS-NM converges to the global minimum of $g$ at the faster (and optimal) $1/T^2$ rate \citep[see][Section 2.2]{nesterov2018lectures}. Based on this result, it is reasonable to expect that replacing GD-BLS  by   GD-BLS-NM will improve the non-asymptotic behaviour of the proposed procedure for computing $\theta_\star$. The extension of our convergence results to the case where GD-BLS-NM is used  is however non-trivial, notably because  it is  unclear if the aforementioned fast convergence rate of this algorithm remains  valid when the function $g$ is only locally $L$-smooth. The  $1/T$ convergence rate of GD-BLS  holds for such functions because  the backtracking procedure ensures that the estimate of $\min_{x\in \R^d}g(x)$ improves from one iteration to the next, and thus that the sequence in $\R^d$   generated by the algorithm stays in a compact set. By contrast,  GD-BLS-NM is not a descent algorithm, and proving that the sequence in $\R^d$ it generates remains in compact set is non-trivial.

Finally,  it is worth mentioning that, instead of backtracking line search, the procedure introduced by \citet{malitsky2019adaptive} can be used for choosing step-sizes $(v_t)_{t\geq 1}$ which guarantee that the resulting gradient descent algorithm generates a   sequence in $\R^d$ that  stays in a compact set, when the  objective function is convex. Under  the additional assumption that the function $f(\cdot,Z)$ is $\P$-a.s.~convex, it is  easily   checked that all the theoretical results presented in this work are valid if, instead of performing backtracking line search,  in  Algorithm \ref{algo:GD1} the step-sizes $(v_t)_{t\geq 1}$ are chosen as proposed by \citet{malitsky2019adaptive}. The key advantage of this approach is that, in the context of this work, it does not consume any computational budget and therefore allows to reduce the cost per iteration of Algorithm \ref{algo:GD1}. However, this procedure has the drawback to require to choose the initial element $v_1$ of the sequence of $(v_t)_{t\geq 1}$ which usually (and as confirmed by some unreported experiments) has a large impact on the behaviour of the algorithm.

\bibliographystyle{apalike}
\bibliography{complete}

\appendix

\section{Proofs}

\subsection{Proof of Proposition   \ref{prop:Poisson22}}

\begin{proof}[Proof of Proposition \ref{prop:Poisson22}]

For all $\theta\in\R$ we have
\begin{equation}\label{eq:poisson}
\begin{split}
&F(\theta)=-\theta+\exp\big(\exp(\theta)-1\big)\\
&\nabla F(\theta)=-1+\exp\big(\exp(\theta)+\theta-1\big)\\
&\nabla^2 F(\theta)=\big(\exp(\theta)+1\big)\exp\big(\exp(\theta)+\theta-1\big).
\end{split}
\end{equation}

We now show that there exists no $\alpha\in(0,1]$ such that  Assumption 5 of \citet{patel2022global} holds. To this aim, for all  $\alpha\in(0,1]$ we let  
\begin{align}\label{eq:def_g}
G_\alpha(\theta)=\E[|\nabla f(\theta,Z)|^{1+\alpha}\big],\quad\forall \theta\in\R
\end{align}
and, for all $\epsilon>0$ and all $\theta\in\R$, we let $\mathcal{L}_{\alpha,\epsilon}(\theta)\in(0,\infty)$ be the smallest constant $L\in(0,\infty)$ such that
\begin{align*}
\big|\nabla F(\theta)-\nabla F(\theta')\big|\leq L\|\theta-\theta'\|^{1+\alpha},\quad\forall\theta'\in B_{r_{\alpha,\epsilon,\theta}}(\theta)
\end{align*}
where  $r_{\alpha,\epsilon,\theta}=(G_\alpha(\theta)\vee \epsilon)^{\frac{1}{1+\alpha}}$ and where    $B_r(\theta)$ denote the closed ball centred at $\theta\in\R$ and with radius $r>0$.

Then, to show that Assumption 5 of \citet{patel2022global} does not hold it is enough to show that 
\begin{align}\label{eq:ToShow_F}
\liminf_{\theta\rightarrow\infty}\frac{\mathcal{L}_{\alpha,\epsilon}(\theta)G_\alpha(\theta)}{F(\theta)+|\nabla F(\theta)|^2}=\infty,\quad\forall \epsilon>0,\quad\forall\alpha\in (0,1].
\end{align}

To show \eqref{eq:ToShow_F} let $\alpha\in (0,1]$ and $\epsilon>0$, and we first  compute a lower bound for  $G_\alpha(\theta)$ and for $\mathcal{L}_{\alpha,\epsilon}(\theta)$. Using H\"older inequality, for all $\theta\in\R$ we have
\begin{align*}
\E[|\nabla f(\theta,Z)|\big]&\leq \E[|\nabla f(\theta,Z)|^{1+\alpha}\big]^{\frac{1}{1+\alpha}},\quad
\end{align*}
where 
\begin{align*}
 \E[|\nabla f(\theta,Z)|\big] =\E\big[|X\exp(\theta X)-X Y|\big]&\geq \E[X\exp(\theta X)\big]-\E[X Y]\\
 &= \E[X\exp(\theta X)\big]-1\\
 &=\exp\big(\exp(\theta)+\theta-1\big)
\end{align*}
and thus
\begin{align}\label{eq:boundG}
G_\alpha(\theta)\geq \Big(\exp\big(\exp(\theta)+\theta-1\big)-1\Big)^{1+\alpha},\quad\forall \theta\geq 1.
\end{align}

To proceed further let $\theta\in\R$. To obtain a lower bound for $\mathcal{L}_{\alpha,\epsilon}(\theta)$ let $r>\log 2$ and, using the fact that
$$
\exp(\theta')\geq \exp(\theta)+(\theta'-\theta)\exp(\theta),\quad\forall \theta'\in\R,
$$
remark that
$$
\exp\big(\exp(\theta)-\exp(\theta')+\theta-\theta'\big)\leq \frac{1}{2},\quad\forall \theta' \in  (\theta+r,\infty). 
$$
Therefore, for all $\theta'  \in (\theta+r,\infty)$ we have,  using \eqref{eq:poisson},
\begin{align*}
\big|\nabla F(\theta)-\nabla F(\theta')\big|&=\Big|\exp\big(\exp(\theta)+\theta-1\big)-\exp\big(\exp(\theta')+\theta'-1\big)\Big|\\
&= \exp\big(\exp(\theta')+\theta'-1\big)\Big(1-\exp\big(\exp(\theta)-\exp(\theta')+\theta-\theta'\big)\Big)\\
&\geq \frac{1}{2}\exp\big(\exp(\theta')+\theta'-1\big)\\
&= \frac{\exp\big(\exp(\theta')+\theta'-1\big)}{2|\theta'-\theta|^{\alpha}}|\theta'-\theta|^\alpha 
\end{align*}
and thus 
\begin{align}\label{eq:lower_D}
\big|\nabla F(\theta)-\nabla F(\theta')\big|\geq \frac{\exp\big(\exp(\theta+r)+\theta+r-1\big)}{2 (2r)^{\alpha}}|\theta'-\theta|^\alpha,\,\,\forall \theta'\in \Big(B_{2r}(\theta)\setminus B_{r}(\theta)\Big)\cap(\theta,\infty).
\end{align}
This shows that if $D_{r,\alpha}(\theta)\in(0,\infty)$ is such that
$$
\big|\nabla F(\theta)-\nabla F(\theta')\big|\leq D_{r,\alpha}(\theta)|\theta-\theta'|^\alpha,\quad\forall\theta'\in B_{2r}(\theta)
$$
then 
$$
D_{r,\alpha}(\theta)\geq K_{r,\alpha}(\theta):=\frac{\exp\big(\exp(\theta+r)+\theta+r-1\big)}{2 (2r)^{\alpha}}.
$$

To proceed further let $c_{\alpha,\epsilon}\in(1,\infty)$ be such that
$$
\exp\big(\exp(\theta)+\theta-1\big)-1\geq  R(\theta):=\frac{1}{2}\exp\big(\exp(\theta)+\theta-1\big)\geq \epsilon^{\frac{1}{1+\alpha}},\quad\forall \theta\geq c_{\alpha,\epsilon}
$$
and note that, by   \eqref{eq:boundG}, we have  $r_{\alpha,\epsilon,\theta}\geq R(\theta)$ for all $\theta\geq c_{\alpha,\epsilon}$.

We now let $c_{\alpha,\epsilon}'\in\R$ be such that $c'_{\alpha,\epsilon}>c_{\alpha,\epsilon}$ and such that $R(\theta)>2$ for all $\theta\geq c'_{\alpha,\epsilon}$. Then, $r_{\alpha,\epsilon,\theta}>2>\log 2$ for all $\theta\geq c_{\alpha,\epsilon}'$ and thus, using the above computations,
\begin{align}\label{eq:boundL}
\mathcal{L}_{\alpha,\epsilon}(\theta)\geq K_{\frac{1}{2}r_{\alpha,\epsilon,\theta},\alpha}(\theta),\quad\forall \theta\geq c_{\alpha,\epsilon}'.
\end{align}

 Therefore, using \eqref{eq:boundG} and \eqref{eq:boundL}, and noting that for all $\theta>0$ we have $K_{r',\alpha}(\theta)\geq K_{r,\alpha}(\theta)$ for all $r'\geq r\geq 1$, it follows that for all $\theta\geq c_{\alpha,\epsilon}'$ we have
\begin{align*}
\mathcal{L}_{\alpha,\epsilon}(\theta) G_\alpha(\theta)&\geq R(\theta)^{1+\alpha} K_{\frac{1}{2}r_{\alpha,\epsilon,\theta},\alpha}\\
&\geq  R(\theta)^{1+\alpha} K_{\frac{1}{2}R(\theta),\alpha}\\
&= R(\theta)^{1+\alpha}\frac{\exp\big(\exp(\theta+R(\theta)/2)+\theta+R(\theta)/2-1\big)}{2 R(\theta)^{\alpha}}\\
&=\frac{R(\theta)}{2} \exp\big(\exp(\theta+R(\theta)/2)+\theta+R(\theta)/2-1\big)
\end{align*}
and thus
\begin{align*}
\liminf_{\theta\rightarrow\infty}\frac{\mathcal{L}_{\alpha,\epsilon}(\theta)G_\alpha(\theta)}{F(\theta)+|\nabla F(\theta)|^2}=\infty.
\end{align*}
Since $\alpha\in (0,1]$ and $\epsilon>0$ are arbitrary, \eqref{eq:ToShow_F} follows and the proof of the proposition is complete.
\end{proof}

\subsection{Proofs of the results of Section \ref{sec:Intro}}

\subsubsection{Proof of Proposition \ref{prop:Poisson3}}

\begin{proof}[Proof of Proposition \ref{prop:Poisson3}]
From \eqref{eq:poisson}  we observe that $\nabla^2 F(\theta)>0$ for all $\theta\in\R$ and thus the function $F$ satisfies \ref{assumption1}. Remark also that, from \eqref{eq:poisson}, the unique minimizer of $F$ is $\theta_\star=0$.

The function $\theta\mapsto f(\theta,z)$ is twice continuously differentiable on $\R$ for all $z\in\R^2$ where, for all $\theta\in\R$ and $z\in\R^2$, 
\begin{equation}\label{eq:poisson2}
\begin{split}
&\nabla f(\theta,z)=-z_1 z_2+z_1\exp(\theta z_1),\quad \nabla^2 f(\theta,z)= z_1^2\exp(\theta z_1).
\end{split}
\end{equation}
Using \eqref{eq:poisson2} and the fact that $\theta_\star=0$, we obtain
\begin{align}\label{eq:A21}
\E[(\nabla f(\theta_\star,Z))^2\big]=\E\Big[\big(-X Y+X\exp(\theta_\star X)\big)^2\Big]=\E\Big[X^2(1-Y)^2\Big]=2
\end{align}
and
\begin{align}\label{eq:A22}
\E[|\nabla^2 f(\theta_\star,Z)|\big]=\E\big[X^2\exp(\theta_\star X) \big]=2.
\end{align}

We now let $\Theta\subset\R$ be a non-empty compact set, $C_\Theta\in(0,\infty)$ be such that $|\theta|\leq C_\Theta$ and such that $|\theta-\theta'|\leq C_\Theta$ for all $\theta,\theta'\in\Theta$, and we let $M_\Theta:\R^2\rightarrow\R$ be defined by
$$
M_\Theta(z)= |z_1^3|\exp(2C_\Theta |z_1|),\quad z\in\R^2.
$$
Then, using the fact that $|e^x-1|\leq |x|e^{|x|}$ for all $x\in\R$, for all $(\theta,\theta')\in\Theta^2$ and $z\in\R^2$ we have
\begin{align*}
\big|\nabla^2 f(\theta,z)-\nabla^2 f(\theta',z)\big|&=z_1^2\big|\exp(\theta z_1)-\exp(\theta' z_1)\big|\\
&=z_1^2\exp(\theta z_1)\big|1-\exp((\theta'-\theta) z_1)\big|\\
&\leq |z_1^3|\exp(\theta z_1) \big|\theta'-\theta\big|\exp(|\theta'-\theta| |z_1|)\\
&\leq  M_\Theta(z) \big|\theta'-\theta\big| 
\end{align*}
where, using Cauchy-Schwartz inequality,
\begin{align*}
\E\big[M_\Theta(Z)\big]=\E\Big[X^3\exp(2C_\Theta X)\Big]&\leq\Big(\E[X^6]\E\big[\exp(4C_\Theta X)\big]\Big)^{1/2}\\
&=\E[X^6]^{1/2}\exp\Big(\exp(4C_\Theta)-1\Big)^{1/2}\\
&<\infty.
\end{align*}
Together with \eqref{eq:A21}-\eqref{eq:A22} this shows that \ref{assumption2} holds and the proof of Proposition \ref{prop:Poisson3} is complete.

\end{proof}

\subsubsection{Proof of Lemma \ref{lemma:assume1}}

\begin{proof}[Proof of Lemma \ref{lemma:assume1}]

Let $\Theta\subset\R^d$ be a (non-empty) closed ball containing $\theta_\star$ and  let $(\theta,\theta')\in\Theta^2$. 

Then, under \ref{assumption2}, by Taylor's theorem  there exists  a $\Theta$-valued random variable   $\tilde{\theta}$ such that, $\P$-a.s., 
\begin{equation}\label{eq:lemma_as}
\begin{split}
\big\| \nabla f(\theta,Z)-\nabla f(\theta',Z)\big\|&=\big\|\nabla^2 f(\tilde{\theta},Z)(\theta-\theta')\big\|\\
&\leq \|\theta-\theta'\| \,\|\nabla^2 f(\tilde{\theta},Z) \|\\
&\leq \|\theta-\theta'\|\Big(\|\nabla^2 f(\tilde{\theta},Z)-\nabla^2 f(\theta_\star,Z) \|+\|\nabla^2 f(\theta_\star,Z) \|\Big)\\
&\leq \|\theta-\theta'\|\Big(\|\tilde{\theta}-\theta_\star\| M_\Theta(Z)+\|\nabla^2 f (\theta_\star,Z) \|\Big)\\
&\leq  \|\theta-\theta'\| M^{(1)}_\Theta(Z)
\end{split}
\end{equation}
where
$$
M^{(1)}_\Theta(z)=\sup_{\theta_1,\theta_2\in\Theta} \|\theta_1-\theta_2\| M_\Theta(z)+\|\nabla^2 f(\theta_\star,z)\|,\quad\forall z\in \mathsf{Z}.
$$
Under \ref{assumption2} we have both $\E[M_\Theta(Z)]<\infty$ and $\E[\|\nabla^2 f(\theta_\star,Z)\|]<\infty$, and thus $\E[M^{(1)}_\Theta(Z)]<\infty$.

Similarly, under \ref{assumption2}, by Taylor's theorem  there exists  a $\Theta$-valued random variable $\tilde{\theta}$  such that, $\P-a.s.$,
\begin{align*}
\big|f(\theta,Z)-f(\theta',Z)\big|&=\big|(\theta-\theta')^\top \nabla f(\tilde{\theta},Z)\big|\\
&\leq \|\theta-\theta'\| \,\|\nabla f(\tilde{\theta},Z) \|\\
&\leq \|\theta-\theta'\|\Big(\|\nabla f(\tilde{\theta},Z)-\nabla f(\theta_\star,Z) \|+\|\nabla f(\theta_\star,Z) \|\Big)\\
&\leq \|\theta-\theta'\|\Big(\|\tilde{\theta}-\theta_\star\| M^{(1)}_\Theta(Z)+\|\nabla f(\theta_\star,Z) \|\Big)\\
&\leq \|\theta-\theta'\|   M^{(2)}_\Theta(Z)
\end{align*}
where the third inequality uses   \eqref{eq:lemma_as} and where
$$
 M^{(2)}_\Theta(z)=\sup_{\theta_1,\theta_2\in\Theta} \|\theta_1-\theta_2\|  M^{(1)}_\Theta(z)+\|\nabla f(\theta_\star,z)\|,\quad\forall z\in \mathsf{Z}.
$$
Since, as shown above, $\E[ M^{(1)}_\Theta(Z)]<\infty$ while, by assumption,  $\E[\|\nabla f(\theta_\star,Z)\|]<\infty$, it follow that $\E[M^{(2)}_\Theta(Z)]<\infty$. The result of the lemma follows.
\end{proof}

\subsubsection{Proof of Lemma \ref{lemma:assume2}}

\begin{proof}[Proof of Lemma \ref{lemma:assume2}]
Let $\theta_\star$ be as in \ref{assumption2} and $\Theta$ be a (non-empty) closed ball containing $\theta_\star$. Next, let $\theta\in\Theta$, $M'_{\Theta}(\cdot)$ be as in Lemma \ref{lemma:assume1} and
$$
G(z)=\sup_{\theta_1,\theta_2\in\Theta}\|\theta_1-\theta_2\|  M_{\Theta}'(z)+\|\nabla  f(\theta_\star,z)\|+\|\nabla^2 f(\theta_\star,z)\|,\quad\forall z\in\mathsf{Z}.
$$

Then, for all $k\in\{1,2\}$ we  $\P$-a.s.~have,  by Lemma \ref{lemma:assume1},
\begin{align*}
\sup_{\theta\in\Theta}\|\nabla^k f(\theta,Z)\|&\leq \sup_{\theta\in\Theta} \| \nabla^k f(\theta,Z)- \nabla^k f(\theta_\star,Z)\|+ \|\nabla^k f(\theta_\star,Z)\|\\
&\leq \sup_{\theta_1,\theta_2\in\Theta}\|\theta_1-\theta_2\| M_{\Theta}'(Z)+\|\nabla^k f(\theta_\star,Z)\|\\
&\leq G(Z).
\end{align*}
Since by Lemma \ref{lemma:assume1} and under \ref{assumption2} we have $\E[G(Z)]<\infty$, it follows from the dominated convergence theorem that, for all $k\in\{1,2\}$, we have   $\nabla^k F(\theta)=\E\big[\nabla^k f(\theta,Z)\big]$   for all $\theta\in\ring{\Theta}$  and thus, since $\Theta$ is arbitrary, $\nabla^k F(\theta)=\E\big[\nabla^k f(\theta,Z)\big]$ for all $\theta\in\R^d$.
 
Since $F$ is twice differentiable on $\R^d$, the mapping $\nabla F:\R^d\rightarrow\R^d$ is continuous and thus to conclude the proof of the lemma it remains to show that the mapping  $\nabla^2 F:\R^d\rightarrow\R^{d\times d}$ is continuous as well. To this aim, let $\Theta$ be a (non-empty) closed ball containing $\theta_\star$ and let $\theta,\theta'\in\Theta$. Then,  
\begin{align*}
\big\|\nabla^2 F(\theta)-\nabla^2 F(\theta')\|&=\big\|\E\big[\nabla^2 f(\theta,Z)\big]-\E\big[\nabla^2 f(\theta',Z)\big] \big\|\\
&\leq \E \Big[\|\nabla^2 f(\theta,Z)-\nabla^2 f(\theta',Z)\|\Big]\\
&\leq \|\theta-\theta'\|\,  \E[M_\Theta(Z)]
\end{align*}
where the first inequality holds by Jensen's inequality and the second one holds under \ref{assumption2}. Since $\E[M_\Theta(Z)]<\infty$, this shows that the mapping  $\nabla^2 F:\R^d\rightarrow\R^{d\times d}$ is locally Lipschitz and thus continuous. The result of the lemma follows.
\end{proof}

\subsubsection{Proof of Lemma \ref{lemma:uniform_conv}}

\begin{proof}[Proof of Lemma \ref{lemma:uniform_conv}]

Let $k\in\{0,1,2\}$ and note that,  using Lemma \ref{lemma:assume1}, we have  $\E[\|\nabla^k f(\theta,Z)\|]<\infty$ for all $\theta\in\R^d$ under \ref{assumption2}. Therefore, by the law of large numbers, for all $\theta\in\R^d$ we have
$$
\lim_{n\rightarrow\infty}  \Big\|\nabla^k F_n(\theta)-\E\big[\nabla^k f(\theta, Z)\big]\Big\|=0,\quad \P-a.s.
$$
To proceed further let $\Theta\subset\R^d$ be a (non-empty) compact set. Then,  by Lemma \ref{lemma:assume1}, there exists a  measurable function $M'_\Theta(\cdot)$ such that $\E[M'_\Theta(Z)]<\infty$ and such that
$$
\big\|\nabla^k f(\theta, Z)-\nabla^k f(\theta',Z)\big\|\leq  M'_\Theta(Z)\|\theta-\theta'\|,\quad\forall \theta,\theta'\in\Theta,\quad \P-a.s.
$$
Therefore, $\P$-a.s., for all $\theta,\theta'\in\R^d$ and $n\geq 1$ we have
\begin{equation}\label{eq:eigen_lem1}
\begin{split}
   \|\nabla^k F_n(\theta)-\nabla^k F_n(\theta')\|&=\frac{1}{n}\Big\|\sum_{i=1}^n \Big(\nabla^k f(\theta, Z_i)-\nabla^k f(\theta',Z_i)\Big)\Big\|\\
    &\leq\frac{1}{n}\sum_{i=1}^n \|\nabla^k f(\theta, Z_i)-\nabla^k f(\theta', Z_i)\|\\
    &\leq \|\theta-\theta'\|\,  \frac{1}{n}\sum_{i=1}^n  M'_\Theta(Z_i).
\end{split}
\end{equation}
 By the law of large numbers, $\lim_{n\rightarrow\infty} \frac{1}{n}\sum_{i=1}^n  M'_\Theta(Z_i)=\E[ M'_\Theta(Z)]<\infty$, $\P$-.a.s, and thus
\begin{equation}\label{eq:eigen_lem2}
\frac{1}{n}\sum_{i=1}^n  M'_\Theta(Z)=\bigO(1),\quad\P-a.s.
\end{equation}

Using \eqref{eq:eigen_lem1}-\eqref{eq:eigen_lem2}, as well as the fact that $\Theta$ is compact and the fact that, by Lemma \ref{lemma:assume2}, the mapping $\nabla^kF$ is uniformly continuous on the compact set $\Theta$, it is readily checked that
$$
\lim_{n\rightarrow\infty}\sup_{\theta\in\Theta}\Big\|\nabla^k F_n(\theta)-\E\big[\nabla^k f(\theta, Z)\big]\Big\|=0.
$$
The proof of the lemma is complete.
\end{proof}

\subsubsection{Proof of Lemma \ref{lemma:e-value}}

\begin{proof}[Proof of Lemma \ref{lemma:e-value}]

To simplify the notation in what follows for all $\theta\in\R^d$ we let $H(\theta)=\E[\nabla^2 f(\theta,Z)]$ and $H_n(\theta)=\nabla^2 F_n(\theta)$ for all $n\geq 1$. 

Let $\Theta$ be a (non-empty) compact set, $\theta\in\Theta$ and $n\geq 1$. Note that the matrices $H_n(\theta)$ and $H(\theta)$ are symmetric and thus, by the min-max theorem,
\begin{align}\label{eq:val_H}
&\lambda_{\min}\big(H_n(\theta)\big)=\min\Big\{x^\top H_n(\theta) x:\,\|x\|=1\Big\},\notag\\
&\lambda_{\min}\big(H(\theta)\big)=\min\Big\{x^\top H(\theta) x:\,\|x\|=1\Big\}.
\end{align}
Since for all $x\in\R^d$ such that $\|x\|=1$ we have
$$
x^\top H_n(\theta) x=x^\top H(\theta) x+x^\top \big(H_n(\theta)-H(\theta))x\geq \lambda_{\min}\big(H(\theta)\big)-  \|H_n(\theta)-H(\theta)\big\|
$$
it follows from Lemma \ref{lemma:uniform_conv} that
\begin{align*}
\liminf_{n\rightarrow\infty}\inf_{\theta\in\Theta}\lambda_{\min}\big(H_n(\theta)\big)&\geq \inf_{\theta\in\Theta}\lambda_{\min}\big(H(\theta)\big)- \limsup_{n\rightarrow\infty}\sup_{\theta\in\Theta}\|H_n(\theta)-H(\theta)\big\|=\inf_{\theta\in\Theta}\lambda_{\min}\big(H(\theta)\big).
\end{align*}
By Lemma \ref{lemma:assume2}, the mapping $H:\R^d\rightarrow\R^{d\times d}$ is continuous and thus, using \eqref{eq:val_H} and the maximum theorem, it follows that the mapping $\theta\mapsto  \lambda_{\min}\big(H(\theta)\big)$ is continuous. Therefore, since $\Theta$ is compact, there exists a $\theta_\Theta\in\Theta$ such that
$$
\inf_{\theta\in\Theta}\lambda_{\min}\big(H(\theta)\big)=\lambda_{\min}\big(H(\theta_\Theta)\big) 
$$
which concludes the proof of the first part of the lemma.

Using the fact, by the min-max theorem,
\begin{align*}
&\lambda_{\max}\big(H_n(\theta)\big)=\max\Big\{x^\top H_n(\theta) x:\,\|x\|=1\Big\},\notag\\
&\lambda_{\max}\big(H(\theta)\big)=\max\Big\{x^\top H(\theta) x:\,\|x\|=1\Big\} 
\end{align*}
the second part of the lemma is proved in a similar way and its proof is therefore omitted to save space.
\end{proof}

\subsubsection{Proof of Lemma \ref{lemma:MLE}}

\begin{proof}[Proof of  Lemma \ref{lemma:MLE}]
 Let $\Theta$ be a   compact and convex    set such that $\theta_\star\in\mathring{\Theta}$  and remark  that, by Lemma \ref{lemma:assume2}, we have  $\E[\nabla f(\theta_\star,Z)]=0$. Hence, since $\E[\|\nabla f(\theta_\star,Z)\|^{1+\alpha}]<\infty$ by assumption, it follows from \citet[][Lemma 3]{bubeck2013bandits} that there exists a constant $C^*<\infty$ such that
\begin{align}\label{eq:bound_alpha}
\P\big( \|\nabla F_{n}(\theta_\star)\|\geq M\big)\leq  \frac{C^*}{n^\alpha M^{1+\alpha}},\quad\forall M>0,\quad\forall n\geq 1.
\end{align}

Remark also that, since $\sup_{\theta\in\Theta}\|F_n(\theta)-F(\theta)\|=\smallo_\P(1)$ by Lemma \ref{lemma:uniform_conv} while,  under \ref{assumption1},  there exists a constant $c_\Theta>0$ such that $F(\theta)-F(\theta_\star)\geq c_\Theta\|\theta-\theta_\star\|$ for all $\theta\in\Theta$, it follows from \citet[][Theorem 5.7, page 45]{VanderVaart2000} that
\begin{align}\label{eq:c_MLE}
\|\theta_\star-\hat{\theta}_{\Theta,n}\|=\smallo_\P(1).
\end{align}

We now let
$$
\lambda_{\Theta}=\frac{1}{2}\inf_{\theta\in \Theta}\lambda_{\min}\big(\E\big[\nabla^2 f(\theta,Z)\big]\big)
$$
and,  to simplify the notation in what follows, for  all $\theta\in\R^d$   and $n\geq 1$ we   let $H_n(\theta)=\nabla^2 F_n(\theta)$. Finally, for all  $\epsilon>0$ and $n\geq 1$  we let
\begin{align*}
\Omega_{\epsilon,n}=\Big\{\omega\in\Omega:\,\, \hat{\theta}_{\Theta,n}^\omega\in\mathring{\Theta},\,\  \big\|\nabla  F^\omega_{n}(\theta_\star)\big\|\leq n^{-\frac{\alpha}{1+\alpha}} (C^*/\epsilon)^{\frac{1}{1+\alpha}} ,\, \inf_{\theta\in\Theta}\lambda_{\min}\big(H^\omega_{n}(\theta)\big)\geq \lambda_\Theta\Big\}.
\end{align*}
Remark that, by \eqref{eq:bound_alpha}-\eqref{eq:c_MLE} and  Lemma \ref{lemma:e-value}, and since $\theta_\star\in\mathring{\Theta}$, we have $\liminf_{n\rightarrow\infty}\P(\Omega_{\epsilon,n})\geq 1-\epsilon$ for all $\epsilon\in(0,1)$.  Remark also that $\lambda_\Theta>0$ under \ref{assumption1} and by Lemma \ref{lemma:assume2}, and thus for all $n\geq 1$ and $\epsilon>0$ the function $F_n^\omega$ is strictly convex on $\Theta$ for all $\omega\in\Omega_{\epsilon,n}$.

Let $n\geq 1$, $\epsilon\in(0,1)$ and $\omega\in\Omega_{\epsilon,n}$, and note that $\nabla F^\omega_{n}(\hat{\theta}^\omega_{\Theta,{n}})=0$ since $F^\omega_{n}$ is strictly convex on $\Theta$ and $\hat{\theta}^\omega_{\Theta,{n}}\in\mathring{\Theta}$. Therefore, since $\theta_\star\in\mathring{\Theta}$, it follows from Taylor's theorem that there exists a $\check{\theta}^\omega_{n_B}\in\mathring{\Theta}$ such that  
\begin{align}\label{eq:FF}
\nabla F^\omega_{n}(\theta_\star)=H_{n}^\omega( \check{\theta}^\omega_{n})(\theta_\star-\hat{\theta}^\omega_{\Theta,{n}})\Leftrightarrow  \theta_\star-\hat{\theta}^\omega_{\Theta,{n}}=\big(H_{n}^\omega( \check{\theta}^\omega_{n})\big)^{-1}\nabla  F^\omega_{n}(\theta_\star).
\end{align}
Remark that the matrix $H_{n}^\omega( \check{\theta}^\omega_{n})$ is indeed invertible since all its eigenvalues are bounded below
by $\lambda_\Theta$, where $\lambda_\Theta>0$ under \ref{assumption1} and by Lemma \ref{lemma:assume2}.

Using \eqref{eq:FF}, we obtain that
\begin{align*}
\|\theta_\star-\hat{\theta}^\omega_{\Theta,n}\|&\leq \|H_{n}^\omega( \check{\theta}^\omega_{n})\|^{-1}\|\nabla  F^\omega_{n}(\theta_\star)\|\leq \Big(\frac{C*}{\epsilon}\Big)^{\frac{1}{1+\alpha}} \lambda^{-1}_\Theta n^{-\frac{\alpha}{1+\alpha}}
\end{align*}
and the result of the lemma follows.

\end{proof}

\subsubsection{Proof of Proposition \ref{prop:rate}}
\begin{proof}[Proof of Proposition \ref{prop:rate}]

We  consider the optimization problem \eqref{eq:optim_prob} with $d=1$, with $\setZ=\R$,  with $f:\R\times\setZ\rightarrow\R$ defined by $f(\theta,z)=\theta^2-2\theta z$, $(\theta,z)\in\R^2$ and with $Z\sim P$ for some probability distribution $P$. Then, it is direct to check that  for any distribution $P$ such that $\E[|Z|^{1+\alpha'}]<\infty$ Assumptions \ref{assumption1}-\ref{assumption2} hold with $\alpha=\alpha'$, and that we have $\theta_\star=\E[Z]$. For all $B\in\mathbb{N}$ let $\hat{\theta}_B$ be an    estimator  of $\theta_\star$ based on $k_B\leq B$  distinct elements of $(Z_i)_{i\geq 1}$. Then, from the proof of \citet[][Theorem 3.1]{Devroye}, it is direct to see that there  exists a distribution $P$ of $Z$ such that $\E[|Z|^{1+\alpha'}]<\infty$ and such that, for some constants $\delta\in(0,1)$ and $C\in(0,\infty)$, we have 
\begin{align*}
\limsup_{B\rightarrow\infty}\P\Big(|\hat{\theta}_{B}-\E[Z]|> C k_B^{-\frac{\alpha}{1+\alpha}}\big)\geq \delta.
\end{align*}
The result of the proposition follows upon noting that $k_B\leq B$ for all $B\in\mathbb{N}$.
\end{proof}

\subsubsection{Proof of Proposition \ref{prop:GD_result_0}}

\begin{proof}[Proof of Proposition \ref{prop:GD_result_0}]

Since $g$ is convex on the convex set $K$ and $x_t\in K$ for all $t\geq 0$, it follows that
 \begin{align*}
 g(x_\star)\geq g(x_t)+\nabla g(x_t)^\top(x_\star-x_t),\quad\forall t\geq 0
 \end{align*}
 and thus
 \begin{align}\label{eq:gd2}
 g(x_t)\leq g(x_\star)-\nabla g(x_t)^\top(x_\star-x_t)=g(x_\star)+\nabla g(x_t)^\top(x_t-x_\star),\quad\forall t\geq 0.
 \end{align}
Remark that
\begin{align}\label{eq:gd3}
\|x_{t+1}-x_\star\|^2=\|x_t-x_\star\|^2+v_t^2\|\nabla g(x_t)\|^2-2 v_t \nabla g(x_t)^\top (x_t-x_\star),\quad\forall t\geq 0 
\end{align}
and let $t\geq 0$. Then, 
\begin{equation}\label{eq:gd4}
\begin{split}
g(x_{t+1})-g(x_\star)&\leq g(x_t)-g(x_\star)-\frac{v_t}{2}\|g(x_t)\|^2\\
&\leq \nabla g(x_t)^\top(x_t-x_\star)-\frac{v_t}{2}\|g(x_t)\|^2\\
&=  \frac{1}{2v_t}\Big(2v_t\nabla g(x_t)^\top(x_t-x_\star)-v_t^2\|g(x_t)\|^2\Big)\\
&=  \frac{1}{2v_t}\Big(2v_t\nabla g(x_t)^\top(x_t-x_\star)-v_t^2\|g(x_t)\|^2-\|x_t-x_\star\|^2+ \|x_t-x_\star\|^2\Big)\\
&=\frac{1}{2v_t}\big(\|x_t-x_\star\|^2-\|x_{t+1}-x_\star\|^2\big)
\end{split}
\end{equation}
where the first inequality uses \eqref{eq:vt}, the second inequality uses \eqref{eq:gd2} and the third equality uses \eqref{eq:gd3}. 

In addition, since $g(x_{t+1})-g(x_\star)\geq 0$ for all $t\geq 0$, it follows from \eqref{eq:gd4} that 
\begin{equation}\label{eq:gd5}
\|x_t-x_\star\|^2-\|x_{t+1}-x_\star\|^2\geq 0,\quad\forall t\geq 1
\end{equation}
 and thus, since for some constant $c>0$ we have $\inf_{t\geq 1}v_t\geq c>0$ by assumption, it follows from \eqref{eq:gd4}-\eqref{eq:gd5} that
$$
g(x_{t+1})-g(x_\star)\leq  \frac{1}{2c}\Big(\|x_t-x_\star\|^2-\|x_{t+1}-x_\star\|^2\Big),\quad\forall t\geq 0.
$$

To conclude the proof let $T\in\mathbb{N}$. Then, using this latter result, we have
\begin{align*}
\sum_{t=1}^{T}\big(g(x_{t})-g(x_\star)\big)&\leq \frac{1}{2c}\sum_{t=1}^{T}\Big(\|x_{t-1}-x_\star\|^2-\|x_{t}-x_\star\|^2\Big)\\
&=\frac{1}{2c}\big(\|x_0-x_\star\|^2-\|x_T-x_\star\|^2\big)\\
&\leq \frac{\|x_0-x_\star\|^2}{2c}
\end{align*}
and thus, since by \eqref{eq:vt} we have $g(x_{t})-g(x_\star)\leq g(x_{t-1})-g(x_\star)$ for all $t\geq 1$, it follows that
\begin{align*}
g(x_T)-g(x_\star)\leq \frac{1}{T}\sum_{t=1}^{T}\big(g(x_{t})-g(x_\star)\big)\leq \frac{\|x_0-x_\star\|^2}{2cT}
\end{align*}
and the proof of the proposition is complete.
\end{proof}

\subsubsection{Proof of Proposition \ref{prop:GD_result}}

\begin{proof}[Proof of Proposition \ref{prop:GD_result}]

Let $K$ and $\tilde{K}$ be as in the statement of the proposition, and remark that since $g$ is convex and $L_K$-smooth on the convex set $K$  we have, by the descent lemma,
\begin{align}\label{eq:prop4_1}
g(y)\leq g(x)+\nabla g(x)^\top (y-x)+\frac{L_K}{2}\|y-x\|^2,\quad\forall (x,y)\in K^2.
\end{align}
Remark also that $x-v\nabla g(x)\in K$ for all $x\in \tilde{K}$ and all $v\in (0,1]$, and let $x\in \tilde{K}\subset K$ and $v\in (0,1]$. Then,   since by \eqref{eq:prop4_1} we have
$$
g\big(x-v \nabla  g(x) \big)\leq g(x)-v \Big(1-\frac{L_K}{2}v\Big)\|\nabla g(x)\|^2 
$$
where
$$
v \Big(1-\frac{L_K}{2}v\Big)\geq \frac{v}{2}\Leftrightarrow v\leq \frac{1}{L_K},
$$
it follows that if $v=\beta^{k_x}$ with 
$$
k_x=\min\Big\{k\in\mathbb{N}_0: g\big(x-\beta^k\|\nabla g(x)\|\big)\leq g(x)-\frac{\beta^k}{2}\|\nabla g(x)\|^2\Big\}
$$
then $v \geq \beta/L_K$. Since  $x_t\in \tilde{K}$ for all $t\geq 0$, this shows that $v_t\geq \beta/L_K$ for all $t\geq 1$  and the proof of the proposition is complete.
\end{proof}

\subsection{Preliminary results for proving  Theorems \ref{thm:rate_GD}-\ref{thm:rate_GD_main}}

\subsubsection{A  technical lemma}

\begin{lemma}\label{lemma:gamma}
Let  $\alpha'\in(0,1]$, $\delta\in \big(0, 2\alpha'/(1+3\alpha')\big)$ and, for all $j\in\mathbb{N}$, let $\gamma_j= 1-\delta^j $. Next, let $w_{\alpha'}=(1+\alpha')/(1+3\alpha')$ and
$$
\gamma'_j=w_{\alpha'}+(1-w_{\alpha'})\gamma_{j-1},\quad\forall j\in\mathbb{N}\setminus\{1\}.
$$
Then, there exists a constant $\pi\in (0,1)$ such that
$$
\gamma_j=\pi\gamma_j'+(1-\pi)\gamma_{j-1},\quad\forall j\in\mathbb{N}\setminus\{1\}.
$$
In addition, if for some constants   $(\kappa,\tau)\in(0,\infty)^2$ and   all $j\in\mathbb{N}$ and $B\in\mathbb{N}$ we let $n_{j,B}=\lceil  \kappa B^{\gamma_j}\rceil$, $\tau_{j,B}= \tau  B^{- \frac{\alpha'}{1+\alpha'} \gamma_j} $ and $r_{j,B}=B^{- \frac{\alpha'}{1+\alpha'}  \gamma_j}$, then
$$
\lim_{B\rightarrow\infty}\frac{n_{j,B} r_{j-1,B}^2}{\tau_{j,B}^2 B}=\lim_{B\rightarrow\infty}\frac{n_{j,B}}{B}=0,\quad\forall j\in\mathbb{N}
$$
with  the convention  that $r_{j-1,B}=1$ for all $B\in\mathbb{N}$ when $j=1$.

\end{lemma}

\begin{proof}[Proof of Lemma \ref{lemma:gamma}]

We show that the first part of the lemma holds for
$$
\pi=(1-\delta)\frac{1+3\alpha'}{1+\alpha'}.
$$
Remark that the conditions on $\delta$ ensure that $\pi\in (0,1)$ and that, by definition, $\pi$ is such that
\begin{align}\label{eq:pi1}
\frac{2\alpha'+(1-\pi)(1+\alpha')}{1+3\alpha'}=\delta.
\end{align}
Then, the first part of the lemma holds   if and only if for all $j\geq 2$ we have
\begin{equation}\label{eq:gg}
\begin{split}
\gamma_j&=\pi w_{\alpha'}+\pi (1-w_{\alpha'})\gamma_{j-1}+(1-\pi)\gamma_{j-1}\\
&=\pi w_{\alpha'}+(1 -\pi w_{\alpha'} )\gamma_{j-1}\\
&=\pi w_{\alpha'}\sum_{i=0}^{j-2}(1-\pi w_{\alpha'})^i+(1-\pi w_{\alpha'})^{j-1}\gamma_1\\
&=1-(1-\pi w_{\alpha'})^{j-1}\big(1-\gamma_1)\\
&=1-\Big(\frac{2\alpha'+(1-\pi)(1+\alpha')}{1+3\alpha'}\Big)^{j-1}(1-\gamma_1)\\
&=1-\delta^{j-1}(1-\gamma_1)\\
&=1-\delta^j
\end{split}
\end{equation}
where the the penultimate equality uses \eqref{eq:pi1} and the last one the fact that $\gamma_1=1-\delta$. This shows the first part of the lemma.

To show the second part of the lemma note that, by the first part of the lemma, we have
\begin{align}\label{eq:in_gamma}
\gamma_j<\gamma_{j+1}<\gamma_{j+1}',\quad\forall j\in\mathbb{N}
\end{align}
and to simplify the notation in the following let $c_{\alpha'}=\alpha'/(1+\alpha')$.

Firstly, the conditions on $\delta$ ensuring  that
$\gamma_1(1+2c_{\alpha'})-1<0$, it follows that
\begin{align}\label{eq:tb1}
&\lim_{B\rightarrow\infty}\frac{ n_{1,B}  }{\tau_{1,B}^2 B}=0.
\end{align}

Next, using the first part of the lemma, we remark that for all $j\in\mathbb{N}\setminus\{1\}$ we have
\begin{align*}
(1+2c_{\alpha'})\gamma_{j}-2c_{\alpha'}\gamma_{j-1}-1&=(1+2c_{\alpha'})\gamma'_{j}-2c_{\alpha'}\gamma_{j-1}-1+(1+2c_{\alpha'})(1-\pi)(\gamma_{j-1}-\gamma_{j}')\\
&=(1+2c_{\alpha'})(1-\pi)(\gamma_{j-1}-\gamma_{j}')\\
&<0
\end{align*}
where the last inequality uses \eqref{eq:in_gamma}, and thus
\begin{align}\label{eq:tb2}
\lim_{B\rightarrow\infty}\frac{ n_{j+1,B} r_{j,B}^2}{\tau^2_{j+1,B}B}=0,\quad\forall j\in\mathbb{N}.
\end{align}
By combining \eqref{eq:tb1} and \eqref{eq:tb2} we obtain that
\begin{align}\label{eq:tb22}
\lim_{B\rightarrow\infty}\frac{ n_{j,B} r_{j-1,B}^2}{\tau^2_{j,B}B}=0,\quad\forall j\in\mathbb{N}.
\end{align}

Finally, since $\gamma_j<1$ for all $j\in\mathbb{N}$, we have 
\begin{align}\label{eq:tb4}
\lim_{B\rightarrow\infty}\frac{ n_{j,B}}{B}=0, \quad\forall j\in\mathbb{N} 
\end{align}
and the   second part of the lemma follows from \eqref{eq:tb22}-\eqref{eq:tb4}. The proof is complete.

\end{proof}

\subsubsection{A first key property of Algorithm \ref{algo:GD1}}

\begin{lemma}\label{lemma:stay_C}
Assume that  \ref{assumption1}-\ref{assumption2} hold. Let $x\in\R^d$, $\beta\in(0,1)$, $(n_B)_{B\geq 1}$ be a sequence in $\mathbb{N}$,   $(\tau_B)_{B\geq 1}$ be a sequence  in $[0,\infty)$ and $(\tilde{B}_{B})_{B\geq 1}$ be a sequence of $\mathbb{N}_0$-valued random variables. In addition, for all $B\in\mathbb{N}$ and $\theta_0\in\R^d$,   let 
$$
(\hat{\theta}_{\theta_0,B}, \tilde{B}_{\theta_0})=\mathrm{Algorithm\, \ref{algo:GD1} }\big(\theta_0,n_B, \tilde{B}_{ B}, \tau_B,   \beta\big) 
$$
and denote by $\{\theta_{\theta_0,B,t}\}_{t=0}^{T_{\theta_0,B}}$   the sequence in $\R^d$  generated by   Algorithm \ref{algo:GD1} to compute $\hat{\theta}_{\theta_0,B}$ from $\theta_0$ and (which is therefore such that $\theta_{\theta_0,B,0}=\theta_0$ and such that $\theta_{\theta_0,B,T_{\theta_0,B}}=\hat{\theta}_{\theta_0,B}$). Finally, for all  $j\in\mathbb{N}_0$, let $\Theta_{j,x}=\big\{\theta\in\R^d:\,F(\theta)\leq F(x)+j\big\}$. Then, for all   $j\in\mathbb{N}_0$, the set $\Theta_{j+1,x}$ is a (non-empty) compact and convex set such that $\theta_\star\in\mathring{\Theta}_{j+1,x}$, and there exists a sequence $(\Omega^{(1)}_{j,B})_{B\geq 1 }$ in $\F$ such that $\lim_{B\rightarrow\infty}\P(\Omega^{(1)}_{j,B})=1$  and such that  
$$
\theta^\omega_{\theta_0,B,t}\in\Theta_{j+1,x},\quad \forall t\in\{0,\dots,T_{\theta_0,B}^\omega\},\quad \forall \omega\in \Omega^{(1)}_{j,B},\quad\forall \theta_0\in\Theta_{j,x},\quad\forall B\in\mathbb{N}.
$$
\end{lemma}

\begin{proof}[Proof of Lemma \ref{lemma:stay_C}]

Let  $j\in\mathbb{N}_0$,
\begin{align}\label{eq:Theta_tilde}
\widetilde{\Theta}_{j,x}=\Big\{\theta\in\R^d:\,\exists \theta'\in\Theta_{j+1,x}\text{ such that }\|\theta-\theta'\|\leq \sup_{\tilde{\theta}\in\Theta_{j+1,x}}\|\nabla F(\tilde{\theta})\|+1\Big\}
\end{align}
 and, for all $B\in\mathbb{N}$, let
\begin{align*}
\Omega^{(1)}_{j,B}=\Big\{\omega\in\Omega:\max_{k\in\{0,1\}} \sup_{\theta\in\tilde{\Theta}_{j,x}}\|\nabla^k F^\omega_{n_B}(\theta)-\nabla^k F(\theta)\|\leq 1/2\Big\}.
\end{align*}

Under \ref{assumption1}  the function $F$ is strictly convex on $\R^d$ and thus the set $\Theta_{j+1,x}$ is convex and bounded. Therefore, since by Lemma \ref{lemma:assume2}  the function $F$ is also continuous on $\R^d$,  the set $\Theta_{j+1,x}$ is compact and such that $\theta_\star\in\mathring{\Theta}_{j+1,x}$. 

To show the second part of lemma remark that  since  $\Theta_{j+1,x}$  is a compact (and convex) set and $F$ is continuous by  Lemma \ref{lemma:assume2}, it follows  that $\sup_{\tilde{\theta}\in\Theta_{j+1,x}}\|\nabla F(\tilde{\theta})\|<\infty$ and thus       $\widetilde{\Theta}_{j,x}$ is also a compact  (and convex) set. Consequently,   by Lemma  \ref{lemma:uniform_conv}, we have $\lim_{B\rightarrow\infty}\P(\Omega^{(1)}_{j,B})=1$.

To proceed further we let $B\in\mathbb{N} $,  $\omega\in \Omega^{(1)}_{j,B}$ and $\theta_0\in\Theta_{j,x}$, and we show that  $\theta_{ \theta_0,B,t}^\omega\in\Theta_{j+1,x}$ for all $t\in\{0,\dots,T_{ \theta_0,B}^\omega\}$. We assume below that $T_{ \theta_0,B}^\omega>0$ since otherwise the result is trivial (since by  $\theta_0\in\Theta_{j,x}\subset\Theta_{j+1,x}$).

To do so let $\theta\in \widetilde{\Theta}_{j,x}$ be such that $F^\omega_{n_B}(\theta)\leq F^\omega_{n_B}(\theta_0)$ and note that
\begin{align*}
F(\theta)-1/2\leq F^\omega_{n_B}(\theta)\leq F^\omega_{n_B}(\theta_0)\leq F(\theta_0)+1/2\leq F(x)+j+1/2
\end{align*}
showing that $F(\theta)\leq F(x)+j+1$  and thus that
\begin{align}\label{eq:inclusion}
\{\theta\in  \widetilde{\Theta}_{j,x}:\,\, F^\omega_{n_B}(\theta)\leq F^\omega_{n_B}(\theta_0)\}\subset \Theta_{j+1,x}.
\end{align}
We now let $t\in \{0,\dots, T_{ \theta_0,B}^\omega-1\}$ and      remark that Algorithm \ref{algo:GD1} ensures that we have
$$
\|\theta_{ \theta_0,B,t+1}^\omega-\theta_{ \theta_0,B,t}^\omega\|\leq \|\nabla F^\omega_{n_B}(\theta_{ \theta_0,B,t}^\omega)\|.
$$
Hence, if $\theta_{ \theta_0,B,t}^\omega \in\Theta_{j+1,x}\subset\tilde{\Theta}_{j,x}$, we have
$$
\|\theta_{ \theta_0,B,t+1}^\omega-\theta_{ \theta_0,B,t}^\omega\|\leq \|\nabla F(\theta_{ \theta_0,B,t}^\omega)\|+1/2\leq \sup_{\tilde{\theta}\in\Theta_{j+1,x}}\|\nabla F(\tilde{\theta})\|+1
$$
showing that $\theta_{\theta_0,B,t+1}^\omega\in \widetilde{\Theta}_{j,x}$. Therefore, since Algorithm \ref{algo:GD1} ensures that we have
$$
F^\omega_{n_B}(\theta_{  \theta_0,B,t+1}^\omega)\leq F^\omega_{n_B}(\theta_{ \theta_0,B,t}^\omega)\leq  F^\omega_{n_B}(\theta_0) 
$$
it follows from \eqref{eq:inclusion} that  $\theta_{ \theta_0,B,t+1}^\omega\in\Theta_{j+1,x}$ if $\theta_{ \theta_0,B,t}^\omega\in\Theta_{j+1,x}$. Consequently, since $\theta_0\in\Theta_{j,x}\subset\Theta_{j+1,x}$, it follows that $\theta_{\theta_0,B,t}^\omega\in\Theta_{j+1,x} $ for all $t\in\{0,\dots,T_{ \theta_0,B}^\omega\}$ and the proof of the lemma is complete.
\end{proof}
 
\subsubsection{A second key property of Algorithm \ref{algo:GD1}}

\begin{lemma}\label{lemma:Error_F}

Consider the set-up of Lemma \ref{lemma:stay_C}. For all $\theta_0\in\R^d$, $B\in\mathbb{N}$ and $t\in\{1,\dots, T_{\theta_0,B}\}$, let $C_{\theta_0,B,t}$ denote the computational cost needed for computing $\theta^\omega_{\theta_0,B,t}$ from $\theta^\omega_{\theta_0,B,t-1}$, with the convention that $C^\omega_{\theta_0,B,t}=0$ when $T^\omega_{\theta_0,B}=0$.  Then, for all $j\in\mathbb{N}_0$,  there exist (i) a compact and convex set $K_{j+1,x}\subset\R^d$ such that $\Theta_{j+1,x}\subseteq K_{j+1,x}$, (ii) a constant  $D_{j}\in(0,\infty)$ and    (iii) a sequence $(\Omega^{(2)}_{j,B})_{B\geq 1}$ in $\F$  such that $\lim_{B\rightarrow\infty}\P(\Omega^{(2)}_{j,B})=1$, such that we have, for all $B\in\mathbb{N}$, all  $\theta_0\in\Theta_{j,x}$ and all $\omega\in\Omega_{j,B}^{(2)}$,   
$$
F^\omega_{n_B}(\hat{\theta}^\omega_{ \theta_0,B})-F^\omega_{n_B}(\hat{\theta}^\omega_{K_{j+1,x},n_B})\leq \frac{D_{j}\|\theta_0-\hat{\theta}_{K_{j+1,x},n_B}\|^2}{ T^\omega_{\theta_0,B}},\quad C^\omega_{\theta_0,B,t}\leq D_{j}\, n_B.
$$
\end{lemma}

\begin{proof}[Proof of Lemma \ref{lemma:Error_F}]

Let  $j\in\mathbb{N}_0$ and  $\widetilde{\Theta}_{j,x}$ be a as defined in \eqref{eq:Theta_tilde}. Recall from the proof of Lemma  \ref{lemma:stay_C} that $\widetilde{\Theta}_{j,x}$ is a compact and convex set, and for all $B\in\mathbb{N}$ 
 let  
\begin{align*}
\Omega^{(2)}_{j,B}=\Omega^{(1)}_{j,B}\cap \ \Big\{ \omega\in\Omega:\, & \text{$F^\omega_{n_B}$  is strictly convex on $\widetilde{\Theta}_{j,x}$},\,\\
&  \|\nabla F^\omega_{n_B}(\theta)-\nabla F^\omega_{n_B}(\theta')\|\leq 2 \E[M'_{\tilde{\Theta}_{j,x}}(Z)] \|\theta-\theta'\|  \text{ for all $(\theta,\theta')\in \widetilde{\Theta}^2_{j,x}$}\Big\}
\end{align*}
with $\Omega^{(1)}_{j,B}$  as in the statement of Lemma \ref{lemma:stay_C}  and with the function $M'_{\tilde{\Theta}_{j,x}}(\cdot)$  as in Lemma \ref{lemma:assume1} (when $\Theta=\tilde{\Theta}_{j,x}$).  Remark that, by Lemma \ref{lemma:assume1}, Corollary \ref{cor:convex}  and Lemma \ref{lemma:stay_C}, we have $\lim_{B\rightarrow\infty}\P(\Omega^{(2)}_{j,B})=1$.

To prove the result of the lemma let $L_{j}=\big(2 \E[M'_{\tilde{\Theta}_{j,x}}(Z)]\vee 1\big)$, $\theta_0\in\Theta_{j,x}$, $B\in\mathbb{N}$ and $\omega\in\Omega_{j,B}^{(2)}$. Remark that the function $F_{n_B}^\omega$ is strictly convex and $L_{j}$-smooth on the compact and convex set $\tilde{\Theta}_{j,x}$, and that since $\omega\in\Omega_{j,B}^{(1)}$ we have $\theta^\omega_{\theta_0,B,t}\in\Theta_{j+1,x}$ for all $t\in\{0,\dots, T_{\theta_0,B}^\omega\}$. 

Assume first  that $T_{\theta_0,B}^\omega>0$  and for all $t\in\{1,\dots,T_{\theta_0,B}^\omega\}$     let $v_{t}^\omega$ be such that $\theta^\omega_{\theta_0,B,t}=\theta^\omega_{\theta_0,B,t-1}+v_{t}^\omega\nabla F_{n_B}(\theta^\omega_{\theta_0,B,t-1})$. Then,  by applying Proposition \ref{prop:GD_result} with  $g=F_{n_B}^\omega$, $K=\tilde{\Theta}_{j,x}$, $\tilde{K}=\Theta_{j+1,x}$ and $L_K=L_{j}$, it follows that $v_t^\omega\geq \beta/L_{j}$ for all 
$t\in\{1,\dots,T_{\theta_0,B}^\omega\}$. From this result, we readily obtain that 
$$
C^\omega_{\theta_0,B,t}\leq D_j n_B,\quad D_j=(\beta^{k_j}+1)(C_{\mathrm{eval}}+C_{\mathrm{grad}}),\quad k_j=\inf\{k\in\mathbb{N}:\,\beta^k <\beta/L_{j}\}.
$$
 In addition, by Proposition \ref{prop:GD_result_0}, we have
$$
F^\omega_{n_B}(\hat{\theta}^\omega_{ \theta_0,B})-F^\omega_{n_B}(\hat{\theta}^\omega_{\tilde{\Theta}_{j,x},n_B})\leq \frac{\beta\|\theta_0-\hat{\theta}_{\tilde{\Theta}_{j,x},n_B}\|^2}{2 L_{j}T^\omega_{\theta_0,B}}.
$$
Noting that this latter inequality also holds if $T_{\theta_0,B}^\omega=0$, the proof of the lemma is complete.
\end{proof}

\subsubsection{A third key property of Algorithm \ref{algo:GD1}}

\begin{lemma}\label{lemma:Error_T}
Consider the set-up of Lemma \ref{lemma:Error_F}. Assume  in addition that there exists a constant $c\in(0,1)$ such that $\P(\tilde{B}_{B}\geq c B)=1$ for  all $B\in\mathbb{N}$ and that  $\lim_{B\rightarrow\infty} n_B/B=0$ and that $\lim_{B\rightarrow\infty} \tau_B=0$. Then, for all $j\in\mathbb{N}_0$, there exist a constant  $D'_{j}\in(0,\infty)$ and a  sequence $(\Omega^{(3)}_{j,B})_{B\geq 1}$ in $\F$  such that $\lim_{B\rightarrow\infty}\P(\Omega^{(3)}_{j,B})=1$ and  such that, with $K_{j+1,x}$ as in Lemma \ref{lemma:Error_F},
$$
\|\hat{\theta}^\omega_{\theta_0,B}-\hat{\theta}^\omega_{K_{j+1,x},n_B}\|\leq \frac{D'_{j}\|\theta_0-\hat{\theta}^\omega_{K_{j+1,x},n_B}\|}{\sqrt{T_{\theta_0,B}^\omega}},\quad\forall\omega\in\Omega^{(3)}_{j,B},\quad\forall\theta_0\in\Theta_{j,x},\quad\forall B\in\mathbb{N}.
$$
\end{lemma}

\begin{proof}[Proof of Lemma \ref{lemma:Error_T}]

Let $j\in\mathbb{N}_0$, and  to simplify the notation in what follows  let $K_x=K_{j+1,x}$ and, for  all $\theta\in\R^d$   and $n\geq 1$, let $H_n(\theta)=\nabla^2 F_n(\theta)$.

We start by showing that
\begin{align}\label{eq:conv_MLE}
\sup_{ \theta_0 \in  \Theta_{j,x}}\|\hat{\theta}_{ \theta_0,B}-\hat{\theta}_{K_{x},n_B}\|=\smallo_\P(1).
\end{align}
To this aim let $M'_{K_{x}}(\cdot)$ be as defined in Lemma \ref{lemma:assume1} when $\Theta=K_{x}$,
$$
\lambda_{K_x}=\frac{1}{2}\inf_{\theta\in K_x}\lambda_{\min}\big(\E\big[\nabla^2 f(\theta,Z)\big]\big)
$$
and, for all $\epsilon>0$ and   $B\in\mathbb{N}$, let
\begin{align*}
\Omega_{\epsilon,B}&=\Omega^{(1)}_{j,B}\cap \Omega^{(2)}_{j,B}\\
&\cap \bigg\{ \omega\in\Omega:\, \,\tilde{B}_B^\omega\geq c B,\,\,\inf_{\theta\in K_x}\lambda_{\min}\big(H^\omega_{n_B}(\theta)\big)\geq \lambda_{K_x},\,\, \|\hat{\theta}^\omega_{K_x,n_B}-\theta_\star\|\leq \epsilon,\,\,\hat{\theta}^\omega_{K_x,n_B}\in\mathring{K}_x,\\
&\hspace*{1cm}\,\,\sup_{\theta\in K_x}|F_{n_B}^\omega(\theta)-F(\theta)|\leq \epsilon,\,\,\frac{1}{n_B}\sum_{i=1}^{n_B}M'_{K_x}(Z^\omega)\leq 2\E[M'_{K_x}(Z)]\bigg\}  
\end{align*}
with   $\Omega^{(1)}_{j,B}$ as in Lemma \ref{lemma:stay_C} and with $\Omega^{(2)}_{j,B}$ as in Lemma \ref{lemma:Error_F}. Remark that $\theta_\star\in\mathring{K}_x$ (see Lemma \ref{lemma:stay_C}-\ref{lemma:Error_F}) and thus, by Lemmas \ref{lemma:assume1}, \ref{lemma:uniform_conv}, \ref{lemma:e-value}, \ref{lemma:MLE}, \ref{lemma:stay_C} and \ref{lemma:Error_F},  we have $\liminf_{B\rightarrow\infty}\P(\Omega_{\epsilon,B})=1$ for all $\epsilon>0$.  

We now let $\theta_0\in \Theta_{j,x}$,  $\epsilon>0$, $B\in\mathbb{N}$ and $\omega\in \Omega_{\epsilon,B}$, and assume first that $\|\nabla  F^\omega_{n_B}(\hat{\theta}^\omega_{ \theta_0,B})\|\leq  \tau_B$. By Lemma \ref{lemma:assume2} and under \ref{assumption1} we have $\lambda_{K_x}>0$,  implying that the function $F^\omega_{n_B}$ is strictly convex on the compact and convex set $K_x$. Consequently, since $\hat{\theta}^\omega_{K_x,n_B}\in\mathring{K}_x$, we have $\nabla F^\omega_{n_B}(\hat{\theta}^\omega_{K_x,n_B})=0$. Therefore,  using Taylor's theorem, the fact that $\|Ax\|\geq \lambda_{\min}(A)\|x\|$ for any $d\times d$ symmetric matrix $A$ and $x\in\R^d$, and noting that $\hat{\theta}^\omega_{ \theta_0,B}\in \Theta_{j+1,x}\subset K_x$, we have
\begin{align*}
\tau_B\geq \|\nabla  F^\omega_{n_B}(\hat{\theta}^\omega_{ \theta_0,B})\|\geq \lambda_{K_x}\,\|\hat{\theta}^\omega_{ \theta_0,B}-\hat{\theta}_{K_{x},n_B}\|
\end{align*}
showing that
\begin{align}\label{eq:case1}
\|\hat{\theta}^\omega_{\theta_0,B}-\hat{\theta}^\omega_{K_{x},n_B}\|\leq \frac{\tau_B}{  \lambda_{K_x}},\quad\text{ if }\|\nabla  F^\omega_{n_B}(\hat{\theta}^\omega_{ \theta_0,B})\|\leq  \tau_B.
\end{align}
Assume now that  $\|\nabla  F^\omega_{n_B}(\hat{\theta}^\omega_{ \theta_0,B})\|>\tau_B$, in which case we have
\begin{align}\label{eq:T_L}
T_{\theta_0,B}^\omega\geq \frac{c\,B}{D_{j} n_B}
\end{align}
with $D_{j}<\infty$ as in Lemma \ref{lemma:Error_F}. Under \ref{assumption1} there exists a constant $\mu_x>0$ such that
$$
 |F(\theta)-F(\theta_\star)\big|\geq \frac{\mu_x}{2}\|\theta-\theta_\star\|^2,\quad\forall \theta\in K_x
$$
and thus
 \begin{equation}\label{eq:split_conv0}
 \begin{split}
\frac{\mu_x}{2}\|\hat{\theta}^\omega_{ \theta_0,B}-\theta_\star\|^2\leq\big|F(\hat{\theta}^\omega_{ \theta_0,B})-F(\theta_\star)\big|&\leq \big|F(\hat{\theta}^\omega_{ \theta_0,B})-F^\omega_{n_B}(\hat{\theta}^\omega_{\theta_0,B})\big|+\big|F^\omega_{n_B}(\hat{\theta}^\omega_{ \theta_0,B})-F(\theta_\star)\big|\\
&\leq \epsilon +\big|F^\omega_{n_B}(\hat{\theta}^\omega_{ \theta_0,B})-F(\theta_\star)\big|.
 \end{split}
 \end{equation}
Recalling that $\theta_\star\in K_{x}$, we have
 \begin{equation}\label{eq:split_conv}
 \begin{split}
 \big|F^\omega_{n_B}(\hat{\theta}^\omega_{ \theta_0,B})-F(\theta_\star)\big|&\leq  \big|F^\omega_{n_B}(\hat{\theta}^\omega_{  \theta_0,B})- F^\omega_{n_B}\big(\hat{\theta}^\omega_{K_{x},n_B}\big)\big|+\big| F^\omega_{n_B}\big(\hat{\theta}^\omega_{K_{x},n_B}\big)-F^\omega_{n_B}\big(\theta_\star)\big|\\
 &+\big|F^\omega_{n_B}\big(\theta_\star)-F(\theta_\star)\big|\\
 &\leq  \big|F^\omega_{n_B}(\hat{\theta}^\omega_{ \theta_0,B})- F^\omega_{n_B}\big(\hat{\theta}^\omega_{K_x,n_B}\big)\big|+\Big(\frac{1}{n_B}\sum_{i=1}^{n_B}M'_{K_x}(Z^\omega_i)\Big)\|  \hat{\theta}^\omega_{K_x,n_B} - \theta_\star\big\|\\
 &+\sup_{\theta\in K_x}\big|F^\omega_{n_B}\big(\theta)-F(\theta)\big|\\
 &\leq \big|F^\omega_{n_B}(\hat{\theta}^\omega_{ \theta_0,B})- F^\omega_{n_B}\big(\hat{\theta}^\omega_{K_x,n_B}\big)\big|+ \epsilon\big(2\E[M'_{K_x}(Z)]+1\big)\\
 &\leq \frac{D_{j} \sup_{(\theta,\theta')\in K_x^2}\|\theta-\theta'\|^2}{T^\omega_{\theta_0,B}}+ \epsilon\big(2\E[M'_{K_x}(Z)]+1\big)\\
 &\leq \frac{n_B\, D^2_{j} \sup_{(\theta,\theta')\in K_x^2}\|\theta-\theta'\|^2}{c B}+ \epsilon\big(2\E[M'_{K_x}(Z)]+1\big)
 \end{split}
 \end{equation}
 where the last inequality holds by \eqref{eq:T_L}.
 
 By combining \eqref{eq:split_conv0} and \eqref{eq:split_conv} we obtain that if  $\|  F^\omega_{n_B}(\hat{\theta}^\omega_{ \theta_0,B})\|> \tau_B $ then
 $$
 \|\hat{\theta}^\omega_{ \theta_0,B}-\hat{\theta}^\omega_{K_{x},n_B}\|\leq \epsilon+ \sqrt{\frac{2}{\mu_x}}\Big(\epsilon+\frac{n_B\, D^2_{j} \sup_{(\theta,\theta')\in K_x^2}\|\theta-\theta'\|^2}{c B}+ \epsilon\big(2\E[M'_{K_x}(Z)]+1\big)\bigg)^{\frac{1}{2}}.
 $$

Together with \eqref{eq:case1}, this shows that for all $\epsilon>0$, $B\in\mathbb{N}$ and $\omega\in\Omega_{\epsilon,B}$, we have
\begin{align*}
\sup_{\theta_0\in \Theta_{j,x}} \|\hat{\theta}^\omega_{ \theta_0,B}-\hat{\theta}^\omega_{K_{x},n_B}\|&\leq \epsilon+\frac{\tau_B}{  \lambda_{K_x}} \\
&+\sqrt{\frac{2}{\mu_x}}\Big(\epsilon+\frac{n_B\, D^2_{j} \sup_{(\theta,\theta')\in K_x^2}\|\theta-\theta'\|^2}{c B}+ \epsilon\big(2\E[M'_{K_x}(Z)]+1\big)\bigg)^{\frac{1}{2}}
 \end{align*}
and \eqref{eq:conv_MLE} follows since, by assumption, $\lim_{B\rightarrow\infty} (n_B/B)=\lim_{B\rightarrow\infty}\tau_B=0$ while $\epsilon>0$ is arbitrary.

To proceed further, let  
$$
 \delta =  \frac{\lambda_{K_x}}{4\E[M'_{K_x}(Z)]}
$$
and, for all $B\in\mathbb{N}$, let
\begin{align*}
\Omega^{(3)}_{j,B}&=\Omega^{(1)}_{j,B}\cap \Omega^{(2)}_{j,B}\\
&\cap\Big\{ \omega\in\Omega: \hat{\theta}^\omega_{K_x,n_B}\in\mathring{K}_x, \sup_{\theta_0\in  \Theta_{j,x}}\|\hat{\theta}^\omega_{\theta_0,B}-\hat{\theta}^\omega_{K_x,n_B}\|\leq \delta,\quad \inf_{ \theta  \in K_x}\lambda_{\min}\big(H^\omega_{n_B}(\theta)\big)\geq \lambda_{K_x},\\
&  \sup_{(\theta,\theta')\in K_x^2\,:\,\theta\neq\theta'}\frac{\|\nabla^2 F^\omega_{n_B}(\theta)-\nabla^2 F^\omega_{n_B}(\theta')\|}{\|\theta-\theta'\|}\leq 2 \E[M'_{K_x}(Z)]\Big\}.
\end{align*}
Note that,  by Lemmas \ref{lemma:assume1}, \ref{lemma:e-value}, \ref{lemma:stay_C} and \ref{lemma:Error_F},  and using  \eqref{eq:conv_MLE}, we have $\liminf_{B\rightarrow\infty}\P(\Omega^{(3)}_{j,B})=1$.

We now let $B\in \mathbb{N}$, $\theta_0\in\Theta_{j,x}$ and $\omega\in\Omega^{(3)}_{j,x,B}$. Then, noting that $\nabla F^\omega_{n_B}(\hat{\theta}^\omega_{K_x,n_B})=0$ since $\hat{\theta}^\omega_{K_x,n_B}\in\mathring{K}_x$ and the function $F^\omega_{n_B}$ is strictly convex on the compact and convex set $K_x$, it follows from Taylor's theorem \citep[see][Theorem 5.8 and Remark 5.9]{amann2005analysis} that
$$
F^\omega_{n_B}(\hat{\theta}^\omega_{  \theta_0,B})-F^\omega_{n_B}(\hat{\theta}^\omega_{K_x,n_B})=\frac{1}{2}\big(\hat{\theta}^\omega_{ \theta_0,B}-\hat{\theta}^\omega_{K_x,n_B})^\top H^\omega_{n_B}(\hat{\theta}^\omega_{K_x,n_B})\big(\hat{\theta}^\omega_{ \theta_0,B}-\hat{\theta}^\omega_{K_x,n_B})^\top+R^\omega_{\theta_0,B}
$$
where
$$
|R^\omega_{ \theta_0,B}|\leq       \| \hat{\theta}^\omega_{  \theta_0,B}-\hat{\theta}^\omega_{K_x,n_B}\|^2  \delta \E[M'_{K_x}(Z)].
$$
Therefore, with $D_{j}<\infty$ as in Lemma \ref{lemma:Error_F}, we have
\begin{align*}
\frac{D_{j}\,\|\theta_0-\hat{\theta}^\omega_{K_x,n_B}\|^2}{T^\omega_{ \theta_0,B}}&\geq F^\omega_{n_B}(\hat{\theta}^\omega_{  \theta_0,B})-F^\omega_{n_B}(\hat{\theta}^\omega_{K_x,n_B})\\
&\geq  \| \hat{\theta}^\omega_{ \theta_0,B}-\hat{\theta}^\omega_{K_x,n_B}\|^2\Big(\frac{\lambda_{K_x}}{2}- \delta \E[M'_\Theta(Z)]\Big)\\
 &=\frac{\lambda_{K_x}}{4} \| \hat{\theta}^\omega_{ \theta_0,B}-\hat{\theta}^\omega_{K_x,n_B}\|^2
\end{align*}
and the result   of the lemma follows from the fact that $\lambda_{K_x}>0$ under \ref{assumption1} and by Lemma \ref{lemma:assume2}.
\end{proof}

\subsubsection{A fourth key property of Algorithm \ref{algo:GD1}}

\begin{lemma}\label{lemma:budget}
Consider the set-up of Lemma \ref{lemma:Error_T}. Assume in addition that for some   sequence $(r_B)_{B\geq 1}$  in $(0,\infty)$ we have $\lim_{B\rightarrow\infty}(n_B r_B^2)/(\tau_B^2 B)=0$. Then, for all $j\in\mathbb{N}_0$ and $\epsilon\in (0,1]$, there exist  a constant $\tilde{C}_{j,\epsilon}\in[1,\infty)$ and a sequence $(\Omega^{(4)}_{j,\epsilon,B})_{B\geq 1}$ in $\F$,  such that $\liminf_{B\rightarrow\infty}\P(\Omega_{j,\epsilon,B}^{(4)})\geq 1-\epsilon$, such that,   for all $B\in\mathbb{N}$, $\omega\in\Omega_{j,\epsilon,B}^{(4)}$ and $\theta_0\in \Theta_{j,x}$ verifying  $\|\theta_0-\hat{\theta}^\omega_{K_{j+1,x},n_B}\|\leq r_B$ (with $K_{j+1,x}$ in Lemma \ref{lemma:Error_F}) we have
\begin{align}
& \hat{\theta}^\omega_{\theta_0,B}\in\Theta_{j+1,x},\label{p1}\\
&\tilde{B}_{\theta_0}^\omega\geq \frac{c B}{2},\label{p2}\\
&\|\hat{\theta}^\omega_{\theta_0,B}-\hat{\theta}^\omega_{K_{j+1,x}, n_B}\|\leq \tilde{C}_{j,\epsilon}\bigg(\tau_B+r_B\sqrt{\frac{n_B}{B}}  \bigg)\label{p3}\\
&\|\hat{\theta}^\omega_{\theta_0,B}-\theta_\star\|\leq \tilde{C}_{j,\epsilon,x}\bigg(\tau_B+ n_B^{-\frac{\alpha}{1+\alpha}}+r_B\sqrt{\frac{n_B}{B}}\bigg)\label{p4}
\end{align}
where $c\in(0,1)$ is as in Lemma \ref{lemma:Error_T} and where $\alpha\in(0,1]$ is as in \ref{assumption2}.

\end{lemma}

\begin{proof}[Proof of Lemma \ref{lemma:budget}]

We let $j\in\mathbb{N}_0$ and to simplify the notation in what follows  we let $K_x=K_{j+1,x}$ and, for all $\theta\in\R^d$   and $n\geq 1$, we let $H_n(\theta)=\nabla^2 F_n(\theta)$.

To show  the result of the lemma,  for all $B\in\mathbb{N}$ and $\theta_0\in\R^d$ we let  
$$
(\tilde{\theta}_{\theta_0,B}, B'_{\theta_0})=\mathrm{Algorithm\, \ref{algo:GD1} }(\theta_0, n_B, \tilde{B}_B, 0,\beta)
$$
and we denote by    $\{\theta'_{\theta_0,B,t}\}_{t=0}^{\tilde{T}_{\theta_0,B}}$   the sequence in $\R^d$  generated by   Algorithm \ref{algo:GD1} to compute $\tilde{\theta}_{\theta_0,B}$ from $\theta_0$ (which is therefore such that $\theta'_{\theta_0,B,0}=\theta_0$ and such that $\theta'_{\theta_0,B,\tilde{T}_{\theta_0,B}}=\tilde{\theta}_{\theta_0,B}$).
In addition, we let
$$
\lambda_{K_x}=\frac{1}{2}\inf_{\theta\in K_x}\lambda_{\min}\big(\E\big[\nabla^2 f(\theta,Z)\big]\big),\quad \bar{\lambda}_{K_x}=2\sup_{\theta\in K_x}\lambda_{\max}\big(\E\big[\nabla^2 f(\theta,Z)\big]\big) 
$$
and, for all $\epsilon>0$, we let $\tilde{D}_{\epsilon}<\infty$ be such that
\begin{align}\label{eq:eps}
\liminf_{B\rightarrow\infty}\Big( \|\hat{\theta}_{K_x,n_B}-\theta_\star\|\leq \tilde{D}_{\epsilon}\, n_B^{-\frac{\alpha}{1+\alpha}}\Big)\geq 1-\epsilon.
\end{align}
Remark that such a finite constant $\tilde{D}_\epsilon$ exists for all $\epsilon>0$ by Lemma \ref{lemma:MLE}.

Finally, for all $B\in\mathbb{N}$ and $\epsilon>0$ we let
\begin{align*}
\tilde{\Omega}^{(4)}_{j,\epsilon,B}&=\Omega^{(1)}_{j,B}\cap \Omega^{(2)}_{j,B}\cap \Omega^{(3)}_{j,B}\\
& \cap \bigg\{ \omega\in\Omega:\, \, \hat{\theta}^\omega_{K_x,n_B}\in\mathring{K}_x,\,\,\,\,\tilde{B}_B^\omega\geq c B,\,\,\inf_{\theta\in K_x}\lambda_{\min}\big(H^\omega_{n_B}(\theta)\big)\geq \lambda_{K_x},\\
&\hspace*{1cm} \|\hat{\theta}^\omega_{K_x,n_B}-\theta_\star\|\leq \tilde{D}_{\epsilon} \,n_B^{-\frac{\alpha}{1+\alpha}},\sup_{\theta\in K_x}\lambda_{\max}\big(H^\omega_{n_B}(\theta)\big)\leq \bar{\lambda}_{K_x}\bigg\}  
\end{align*}
with  $c\in(0,1)$ as in Lemma \ref{lemma:Error_T}, $\Omega^{(1)}_{j,B}$ as in Lemma \ref{lemma:stay_C}, $\Omega^{(2)}_{j,B}$ as in Lemma \ref{lemma:Error_F} and $\Omega^{(3)}_{j,B}$ as in Lemma \ref{lemma:Error_T}. Remark that $\theta_\star\in\mathring{K}_x$ (see Lemmas \ref{lemma:stay_C}-\ref{lemma:Error_F}) and thus, by Lemma \ref{lemma:e-value}, Lemmas \ref{lemma:stay_C}-\ref{lemma:Error_T} and by \eqref{eq:eps}, we have $\liminf_{B\rightarrow\infty}\P(\tilde{\Omega}^{(4)}_{j,\epsilon,B})\geq 1-\epsilon$ for all $\epsilon>0$. In addition, remark that there exists a constant $D_{j}'<\infty$ such that
\begin{align}\label{eq:b_bound}
\|\tilde{\theta}^\omega_{\theta_0,B}-\hat{\theta}^\omega_{K_x,n_B}\|\leq D'_{j}\|\theta_0-\hat{\theta}^\omega_{K_x,n_B}\| \sqrt{\frac{n_B}{B}},\quad\forall \omega\in \tilde{\Omega}^{(4)}_{j,\epsilon,B},\quad\forall\theta_0\in\Theta_{j,x},\quad\forall B\in\mathbb{N},\quad\forall\epsilon>0.
\end{align}

We now let $\epsilon>0$, $B\in\mathbb{N}$, $\omega\in\tilde{\Omega}^{(4)}_{j,\epsilon,B}$ and $\theta_0\in\Theta_{j,x}$ be such that $\|\theta_0-\hat{\theta}^\omega_{K_x,N^\omega_B}\|\leq r_{B}$. Assume first that
$$
\|\nabla  F^\omega_{n_B}(\hat{\theta}^\omega_{ \theta_0,B})\|\leq  \tau_B 
$$
and remark that $\nabla F^\omega_{n_B}(\hat{\theta}^\omega_{K_x,n_B})=0$   since $F^\omega_{n_B}$ is strictly convex on the convex set $K_x$ (as $\lambda_{K_x}>0$ under \ref{assumption1} and by Lemma \ref{lemma:assume2}) and since  $\hat{\theta}^\omega_{K_x,n_B}\in\mathring{K}_x$.

Then, using Taylor's theorem,  the fact that $\|Ax\|\geq \lambda_{\min}(A)\|x\|$ for any $d\times d$ symmetric matrix $A$ and $x\in\R^d$, we have
\begin{align}\label{eq:tilde_1}
\tau_B\geq \|\nabla  F^\omega_{n_B}(\hat{\theta}^\omega_{ \theta_0,B})\|\geq \lambda_{K_x}\,\|\hat{\theta}^\omega_{ \theta_0,B}-\hat{\theta}_{K_{x},n_B}\|\Leftrightarrow  \|\hat{\theta}^\omega_{ \theta_0,B}-\hat{\theta}_{K_{x},n_B}\|\leq \frac{\tau_B }{ \lambda_{K_x}}.
\end{align} 

Assume now that $\|\nabla  F^\omega_{n_B}(\hat{\theta}^\omega_{ \theta_0,B})\|>  \tau_B$. In this case, $\hat{\theta}^\omega_{\theta_0,B}=\tilde{\theta}^\omega_{\theta_0,B}$ and thus 
\begin{equation}\label{eq:tilde_2}
\begin{split}
\|\hat{\theta}^\omega_{\theta_0,B}-\hat{\theta}^\omega_{K_x,n_B}\|&=\|\tilde{\theta}^\omega_{ \theta_0,B}-\hat{\theta}^\omega_{K_x,n_B}\|\leq D'_{j} \|\theta_0-\hat{\theta}^\omega_{K_x,N^\omega_B}\| \sqrt{\frac{n_B}{B}}\leq D'_{j} r_{ B} \sqrt{\frac{n_B}{B}}
\end{split}
\end{equation}
where the first inequality holds by \eqref{eq:b_bound}. Then, \eqref{p3} follows from \eqref{eq:tilde_1} and \eqref{eq:tilde_2}, implying that \eqref{p4} holds since
$$
\|\hat{\theta}^\omega_{\theta_0,B}-\theta_\star\|\leq \|\hat{\theta}^\omega_{\theta_0,B}-\hat{\theta}^\omega_{K_x,n_B}\|+\|\theta_\star-\hat{\theta}^\omega_{K_x,n_B}\|\leq \|\hat{\theta}^\omega_{\theta_0,B}-\hat{\theta}^\omega_{K_x,n_B}\|+  \tilde{D}_\epsilon\, n_B^{-\frac{\alpha}{1+\alpha}}.
$$

To show \eqref{p2}  let $D_{j}'<\infty$ be as in Lemma \ref{lemma:Error_T} and remark that, using Taylor's theorem,
\begin{equation*}
\begin{split}
\|\nabla F^\omega_{n_B}(\hat{\theta}_{\theta_0,B}^\omega)\|\leq \bar{\lambda}_{K_x}\,\|\hat{\theta}^\omega_{ \theta_0,B}-\hat{\theta}^\omega_{K_x,n_B}\|  
  &\leq  \frac{ \bar{\lambda}_{K_x} D'_{j}\|\theta_0-\hat{\theta}^\omega_{K_x,n_B}\|}{\sqrt{T^\omega_{ \theta_0,B}}}\leq  \tau_B \bar{\lambda}_{K_x} D_{j}'\sqrt{\frac{r_{B}^2}{\tau_B^2\,T^\omega_{\theta_0,B}}}
\end{split}
\end{equation*}
so that
\begin{align}\label{eq:boundTB}
  T^\omega_{\theta_0,B}\leq \min\Big(\tilde{T}^\omega_{\theta_0,B},   (r_B/\tau_B)^2 (\bar{\lambda}_{K_x} D'_{j})^2 +1  \Big).
\end{align}
To proceed further let $D_{j}<\infty$ be as in Lemma \ref{lemma:Error_F} and remark that
$$
 \tilde{T}^\omega_{\theta_0,B}\geq \frac{\tilde{B}_B^\omega}{n_B   D_{j}}\geq  \frac{c B}{  n_B  D_{j} }.
$$
Since by assumption, $\lim_{B\rightarrow\infty} (r_{B}/\tau_B)^2 n_B/B =0$,
it follows from \eqref{eq:boundTB} that   for $B$ large enough we have
\begin{align*}
T^\omega_{\theta_0,B}\leq  (r_{B}/\tau_B)^2 (\bar{\lambda}_{K_x} D'_{j})^2
\end{align*}
and thus
\begin{equation*}
\begin{split}
\tilde{B}_{\theta_0}^\omega&\geq \tilde{B}_B^\omega- T^\omega_{\theta_0,B} n_B  D_{j} \geq  c\,B-   \frac{n_B r^2_{B}}{\tau_B^2}    D_{j}(\bar{\lambda}_{K_x} D'_j)^2= B\Big(c-\frac{n_B r^2_{B}}{ \,\tau_B^2 B}\,    D_{j}
 (\bar{\lambda}_{K_x} D'_j)^2\Big).
\end{split}
\end{equation*}
Therefore, since $\lim_{B\rightarrow\infty} n_B r^2_{B}/(B\tau^2_B)=0$ by assumption, it follows that \eqref{p2} holds assuming that $B$ is sufficiently large. The result of the lemma follows.

\end{proof}

\subsection{Proof of Theorem \ref{thm:rate_GD}}

\begin{proof}[Proof of Theorem \ref{thm:rate_GD}]

Let $\epsilon\in (0,1)$. Then, by Lemma \ref{lemma:MLE} and by applying Lemmas \ref{lemma:Error_F}-\ref{lemma:Error_T} with $j=0$, $x=\theta_0$, $c=1$ and with
\begin{align*} 
n_B=n(B),\quad \tau_B=0,\quad\forall B\in\mathbb{N},
\end{align*}
it directly follows that there exists a sequence $(\Omega_{\epsilon,B})_{B\geq 1}$ in $\F$, such that $\liminf_{B\rightarrow\infty}\P(\Omega_{\epsilon,B})\geq 1-\epsilon$, and a constant  $C_{\epsilon}\in [1,\infty)$,  for which we have
\begin{align*}
\|\tilde{\theta}_{B}^\omega-\theta_\star\|\leq C_{\epsilon} \bigg(n(B)^{-\frac{\alpha}{1+\alpha}}+\sqrt{\frac{n_B}{B}}\bigg),\quad\forall \omega \in \Omega_{\epsilon,B},\quad\forall B\in\mathbb{N}.
\end{align*}
The result of the theorem follows.
\end{proof}

\subsection{Proof of Theorem \ref{thm:rate_GD_main}}

\subsubsection{Proof of the first part of Theorem \ref{thm:rate_GD_main}} 

\begin{proof}[Proof of the first part of Theorem \ref{thm:rate_GD_main}]

Noting that if \ref{assumption2} holds for some $\alpha\in (0,1]$ then it also holds for any $\tilde{\alpha}\in (0,\alpha]$, below we can without loss of generality  assume that \ref{assumption2} holds for some $\alpha\leq \alpha'$.

To simplify the notation in what follows, for  all $\theta\in\R^d$   and $n\geq 1$ we let $H_n(\theta)=\nabla^2 F_n(\theta)$ and, for all $j\in\mathbb{N}$ and $B\in\mathbb{N}$  we let $n_{j,B}=\lceil  \kappa B^{\gamma_j}\rceil$ and $\tau_{j,B}= \tau  B^{-\frac{\alpha'}{1+\alpha'} \gamma_j} $. In addition, for all $x\in \R^d$ and $j\in\mathbb{N}_0$ we let the sets $\Theta_{j,x}$ and $K_{j+1,x}$ be as defined in Lemma \ref{lemma:stay_C} and in Lemma \ref{lemma:Error_F}, respectively.

Let $\epsilon\in (0,1)$. Then, by applying Lemma \ref{lemma:budget} with $j=1$, $x=\theta_0$, $c=1$ and with
\begin{align}\label{eq:def_seq}
n_B=n_{1,B},\quad \tau_B=\tau_{1,B},\quad r_B=\sup_{(\theta,\theta')\in K_{2,\theta_0}}\|\theta-\theta'\|,\quad\forall B\in\mathbb{N},
\end{align}
it follows that there exists a sequence $(\Omega_{\epsilon,B})_{B\geq 1}$ in $\F$, such that $\liminf_{B\rightarrow\infty}\P(\Omega_{\epsilon,B})\geq 1-\epsilon$, and a constant  $C_{\epsilon}\in [1,\infty)$  for which we have
\begin{align}\label{eq:Show_indMain_2}
\max\Big(\|\hat{\theta}_{1,B}^\omega-\theta_\star\|,\|\hat{\theta}_{1,B}^\omega-\hat{\theta}^\omega_{K_{2,\theta_0},n_{1,B}}\|\Big)\leq C_{\epsilon}\, B^{-c_{\alpha,\delta}},\quad\forall \omega \in \Omega_{\epsilon,B},\quad\forall B\in\mathbb{N}
\end{align}
with 
$$
c_{\alpha,\delta}=\min\bigg(\frac{(1-\delta) \alpha}{1+\alpha},\frac{\delta}{2}\bigg)=\frac{(1-\delta)\alpha}{1+\alpha}.
$$
Remark that, by Lemma \ref{lemma:gamma}, the sequences $(n_B)_{B\geq 1}$, $(\tau_B)_{B\geq 1}$ and $(r_B)_{B\geq 1}$ defined in \eqref{eq:def_seq} satisfy the assumptions of Lemma \ref{lemma:budget}. Then, the result of the first part of the theorem for $J=1$ follows from \eqref{eq:Show_indMain_2}.

We assume now that $J>1$ and, without loss of generality, we assume  that $K_{j+1,\theta_0}=K$ for all $j\in\{1,\dots,J\}$ and for some compact and convex set $K\subset\R^d$. In addition, we let
$$
\lambda_{J}=\frac{1}{2}\inf_{\theta\in K}\lambda_{\min}\big(\E\big[\nabla^2 f(\theta,Z)\big]\big)
$$
and, for all $\epsilon\in(0,1)$, we let $C'_\epsilon<\infty$ be such that
$$
\liminf_{B\rightarrow\infty}\P\Big(\|\hat{\theta}_{ K,n_{j,B}}-\theta_\star\|\leq C'_\epsilon n_{j,B}^{-\frac{\alpha}{1+\alpha}}\Big)\geq 1-\frac{\epsilon}{J},\quad\forall j\in\{1,\dots,J\}.
$$
Remark that, for all $\epsilon\in(0,1)$, such a constant $C'_\epsilon<\infty$ exists by Lemma \ref{lemma:MLE}. Remark also that, using Fr\'echet's inequality,
\begin{align}\label{eq:lim_MLE}
\liminf_{B\rightarrow\infty}\P\Big(\max_{j\in\{1,\dots,J\}}\|\hat{\theta}_{ K,n_{j,B}}-\theta_\star\|\leq C'_\epsilon n_{j,B}^{-\frac{\alpha}{1+\alpha}}\Big)\geq 1-\epsilon.
\end{align}

Then, for all  $\epsilon\in(0,1)$ and  $B\in\mathbb{N}$,  we let
\begin{align*}
\Omega'_{\epsilon,B}&= \Omega_{\epsilon,B}
\cap_{j=1}^J\Big(\Omega_{j,B}^{(1)}\cap \Omega_{j,B}^{(2)}\Big)\\
&\cap \bigg\{ \omega\in\Omega:\, \min_{j\in\{1,\dots,J\}}\inf_{\theta\in  K}\lambda_{\min}\big(H^\omega_{n_{j,B}}(\theta)\big)\geq \lambda_{J},\,\,\max_{j\in\{1,\dots,J\}}n_{j,B}^{\frac{\alpha}{1+\alpha}}\|\hat{\theta}^\omega_{ K,n_{j,B}}-\theta_\star\|\leq C'_\epsilon,\\
& \hspace{0.7cm}\hat{\theta}^\omega_{ K, n_j}\in \mathring{K}\text{ for all $j\in\{1,\dots,J\}$}\bigg\}
\end{align*}
with $\Omega_{\epsilon,B}$ as above, $\Omega_{j,B}^{(1)}$ as in Lemma \ref{lemma:stay_C} and  $\Omega_{j,B}^{(2)}$ as in Lemma \ref{lemma:Error_F}. Remark that, by   Lemmas \ref{lemma:e-value}, \ref{lemma:stay_C} and \ref{lemma:Error_F}, and by \eqref{eq:lim_MLE}, we have $\liminf_{B\rightarrow\infty}\P(\Omega'_{\epsilon,B})\geq 1-\epsilon$.

Let $\epsilon\in(0,1)$, $B\in\mathbb{N}$ and $\omega\in \Omega'_{\epsilon,B}$. Remark that $\nabla F^\omega_{n_{j,B}}(\hat{\theta}^\omega_{K,n_{j,B}})=0$ for all $j\in\{1,\dots,J\}$,  since for all $j\in\{1,\dots, J\}$ the function $F^{\omega}_{n_{j,B}}$ is strictly convex on $K$ (since $\lambda_J>0$ under \ref{assumption1} and by Lemma \ref{lemma:assume2}) while  $\hat{\theta}^\omega_{K, n_j}\in\mathring{K}$. Therefore, using Taylor's theorem, the fact that $\|Ax\|\geq \lambda_{\min}(A)\|x\|$ for any $d\times d$ symmetric matrix $A$ and $x\in\R^d$, we have
\begin{align}\label{eq:thm2b1}
F^\omega_{n_{j,B}}(\theta)-F^\omega_{n_{j,B}}(\hat{\theta}^\omega_{K,n_{j,B}})\geq \lambda_J \|\theta-\hat{\theta}^\omega_{K,n_{j,B}}\|^2,\quad\forall \theta\in K.
\end{align}

For all $j\in\mathbb{N}$ let $D_j$ be as in Lemma \ref{lemma:Error_F} and assume without loss of generality that $D_j\leq D_{j+1}$ for all $j\in\mathbb{N}$ and that $D_J/\lambda_J\geq 1$.  In addition, let $J_B^\omega$ denote the smallest $j\in\{1,\dots,J\}$ such that $\hat{\theta}^\omega_{j,B}=\hat{\theta}^{\omega}_{J,B}$. Then, noting that $n_{j,B}\leq n_{j+1,B}$ for all $j\in\mathbb{N}$, if $J_B^\omega>1$ we have
\begin{align*}
\|\hat{\theta}_{J_B^\omega,B}^\omega-\hat{\theta}^\omega_{K,n_{J_B^\omega,B}}\|^2&\leq \frac{1}{\lambda_J}\Big(F^\omega_{n_{J_B^\omega,B}}(\hat{\theta}_{J_B^\omega,B}^\omega)-F^\omega_{n_{J_B^\omega,B}}(\hat{\theta}^\omega_{K,n_{J_B^\omega,B}})\Big)\\
&\leq \frac{ D_J\|\hat{\theta}_{J_B^\omega-1,B}^\omega-\hat{\theta}^\omega_{K,n_{J_B^\omega,B}}\|^2}{T^\omega_{J_B^\omega,B}\lambda_J}\\
&\leq \frac{2 D_J}{\lambda_J}\Big(\|\hat{\theta}_{J_B^\omega-1,B}^\omega-\hat{\theta}^\omega_{K,n_{J_B^\omega-1,B}}\|^2+\|\hat{\theta}^\omega_{K,n_{J_B^\omega,B}}-\hat{\theta}^\omega_{K,n_{J_B^\omega-1,B}}\|^2\Big)\\
&\leq   \frac{4 D_J}{\lambda_J}\Big(\|\hat{\theta}_{J_B^\omega-1,B}^\omega-\hat{\theta}^\omega_{K,n_{J_B^\omega-1,B}}\|^2+  (C'_\epsilon)^2 n_{1,B}^{-\frac{2\alpha}{1+2\alpha}}\Big)\\
&\leq \Big(\frac{4 D_J}{\lambda_J}\Big)^{J}\|\hat{\theta}_{1,B}^\omega-\hat{\theta}^\omega_{K,n_{1,B}}\|^2+(C'_\epsilon)^2 n_{1,B}^{-\frac{2\alpha}{1+2\alpha}}\sum_{j=1}^J(4 D_J/\lambda_J)^j\\
&\leq \Big(\frac{4 D_J}{\lambda_J}\Big)^{J} C^2_{\epsilon}\, B^{-2c_{\alpha,\delta}}+(C'_\epsilon)^2 n_{1,B}^{-\frac{2\alpha}{1+2\alpha}}\sum_{j=1}^J(4 D_J/\lambda_J)^j
\end{align*}
where the first inequality uses \eqref{eq:thm2b1} and the last inequality uses \eqref{eq:Show_indMain_2}.

If $J_B^\omega=1$ then, from  \eqref{eq:Show_indMain_2}, we have
$$
\|\hat{\theta}_{J_B^\omega,B}^\omega-\hat{\theta}^\omega_{K,n_{J_B^\omega,B}}\|^2\leq C_\epsilon^2 B^{-2c_{\alpha,\delta}}.
$$
Therefore,
\begin{align*}
\|\hat{\theta}_{J,B}^\omega-\hat{\theta}^\omega_{K,n_{J^\omega,B}}\|=\|\hat{\theta}_{J_B^\omega,B}^\omega-\hat{\theta}^\omega_{K,n_{J^\omega,B}}\| \leq C''_\epsilon  B^{- c_{\alpha,\delta}}
\end{align*}
where
\begin{align*}
 C''_\epsilon=\bigg(\Big(\frac{4 D_J}{\lambda_J}\Big)^{J} C^2_{\epsilon}+(C'_\epsilon)^2\sum_{j=1}^J(4 D_J/\lambda_J)^j\Big)^{1/2}.
\end{align*}
Therefore,
\begin{align*}
\|\hat{\theta}_{J,B}^\omega-\theta_\star\|&\leq \|\hat{\theta}^\omega_{K,n_{J^\omega,B}}-\theta_\star\|+\|\hat{\theta}_{J,B}^\omega-\hat{\theta}^\omega_{K,n_{J^\omega,B}}\|\\
&\leq C'_\epsilon n_{1,B}^{-\frac{\alpha}{1+\alpha}}+C''_\epsilon  B^{- c_{\alpha,\delta}}\\
&\leq (C'_\epsilon+C''_{\epsilon})B^{- c_{\alpha,\delta}}.
\end{align*}
The result of the first part of the theorem follows.
\end{proof}

\subsubsection{Proof of the second part of Theorem \ref{thm:rate_GD_main}}

\begin{proof}[Proof of the 2nd part of Theorem \ref{thm:rate_GD_main}]

Let $c_\alpha=\alpha/(1+\alpha)$ and, for all $j\in\mathbb{N}$ and $B\in\mathbb{N}$,  let $n_{j,B}=\lceil  \kappa B^{\gamma_j}\rceil$, $\tau_{j,B}= \tau  B^{-c_\alpha \gamma_j} $ and $r_{j,B}=B^{- c_\alpha  \gamma_j}$. Finally, let $r_{0,B}=1$ for all $B\in\mathbb{N}$ and, for all $x\in \R^d$ and $j\in\mathbb{N}$, let $\Theta_{j,x}$ be as defined in Lemma \ref{lemma:stay_C} and $K_{j+1,x}$ be as defined in Lemma \ref{lemma:Error_F}.

To prove the second part of the theorem we show by induction    that for all $j\in\mathbb{N}$ there exist, for all $\epsilon\in(0,\infty)$,  a sequence $(\Omega_{\epsilon,j,B})_{B\geq 1}$ in $\F$  such that $\liminf_{B\rightarrow\infty}\P(\Omega_{\epsilon,j,B})\geq 1-\epsilon$, and a constant  $C_{\epsilon,j}\in [1,\infty)$  for which we have
\begin{align}\label{eq:Show_indMain}
B_j^\omega\geq \frac{B}{C_{\epsilon,j}},\quad \|\hat{\theta}_{j,B}^\omega-\theta_\star\|\leq C_{\epsilon,j}\, r_{j,B}, \quad \|\hat{\theta}_{j,B}^\omega-\hat{\theta}^\omega_{K_{j+1,\theta_0},n_{j,B}}\|\leq C_{\epsilon,j}\, r_{j,B},\quad  \hat{\theta}_{j,B}^\omega\in \Theta_{j+1,\theta_0}
\end{align}
for all $w\in\Omega_{\epsilon,j,B}$ and $ B \in\mathbb{N}$.

To do so for all $B\in\mathbb{N}$ let $n_B=n_{1,B}$, $\tau_B=\tau_{1,B}$ and $r_B=\sup_{(\theta,\theta')\in K_{2,\theta_0}}\|\theta-\theta'\|$, and note that, by  Lemma \ref{lemma:gamma}, these sequences $(n_B)_{B\geq 1}$, $(\tau_B)_{B\geq 1}$ and $(r_B)_{B\geq 1}$ satisfy the assumptions of Lemma \ref{lemma:budget}. Therefore, by applying Lemma \ref{lemma:budget} with $c=1$, $x=\theta_0$ and $j=1$, it follows that for all $\epsilon\in(0,\infty)$ there exist  a sequence $(\Omega_{\epsilon,1,B})_{B\geq 1}$ in $\F$ such that   $\liminf_{B\rightarrow\infty}\P(\Omega_{\epsilon,j,B})\geq 1-\epsilon$, and  a constant $C_{\epsilon,1}\in [1,\infty)$, such that  \eqref{eq:Show_indMain} holds for all $w\in\Omega_{\epsilon,1,B}$ and all $ B \in\mathbb{N}$ when $j=1$.

Assume now that for some $s\in\mathbb{N}$   there exist, for all $\epsilon\in(0,\infty)$,  a sequence $(\Omega_{\epsilon,s,B})_{B\geq 1}$ in $\F$ such that $\liminf_{B\rightarrow\infty}\P(\Omega_{\epsilon,s,B})\geq 1-\epsilon$   and a constant  $C_{\epsilon,s}\in[1,\infty)$   for which \eqref{eq:Show_indMain} holds  for all $w\in\Omega_{\epsilon,s,B}$ and all $ B \in\mathbb{N}$ when $j=s$. 

For all $\epsilon>0$ and $B\in\mathbb{N}$ we  now let   $\tilde{B}^\omega_{\epsilon,s}= \max(B^{\omega}_s,B/C_{\epsilon,s}) $ for all $\omega\in\Omega$. In addition, for all $B\in\mathbb{N}$, we let  $n_B=n_{s+1,B}$, $\tau_B=\tau_{s+1,B}$ and $r_B= r_{s,B}$ and remark that, by Lemma \ref{lemma:gamma}, these   sequences $(n_B)_{B\geq 1}$, $(\tau_B)_{B\geq 1}$ and $(r_B)_{B\geq 1}$ satisfy the assumptions of Lemma \ref{lemma:budget}. Let $\epsilon>0$. Then,  by applying Lemma \ref{lemma:budget} with $c=1/C_{\frac{\epsilon}{2},s}$, $x=\theta_0$ and $j=s+1$,  it follows that   there exist  a sequence $(\Omega'_{\epsilon/2,s+1,B})_{B\geq 1}$ in $\F$ such that   $\liminf_{B\rightarrow\infty}\P(\Omega'_{\epsilon/2,s+1,B})\geq 1-\epsilon/2$ and  a constant $\tilde{C}_{\epsilon/2,s+1}\in [1,\infty)$ such that, for all $B\in\mathbb{N}$ and $\omega\in \Omega_{\epsilon/2,s,B}\cap \Omega'_{\epsilon/2,s+1,B}$, we have
\begin{equation}\label{eq:main_thm2}
\begin{split}
& \hat{\theta}^\omega_{s+1,B}\in\Theta_{s+2,\theta_0},\quad B_{s+1}^\omega\geq \frac{B}{2 C_{\frac{\epsilon}{2},s}},\\
&\|\hat{\theta}^\omega_{s+1,B} -\hat{\theta}^\omega_{K_{s+1,\theta_0}, n_{s+1,B}}\|\leq \tilde{C}_{\epsilon/2,s+1} \bigg( \tau_{s+1,B}+r_{s,B}\sqrt{ \frac{n_{s+1,B}}{B}} \bigg)\\
&\|\hat{\theta}^\omega_{s+1,B} -\theta_\star\|\leq \tilde{C}_{\epsilon/2,s+1} \bigg( \tau_{s+1,B}+r_{s,B}\sqrt{ \frac{n_{s+1,B}}{B}} \bigg).
\end{split}
\end{equation}

To proceed further remark that, for some constant $D<\infty$,   
\begin{align}\label{eq:in_tat}
\tau_{s+1,B}+r_{s,B}\sqrt{ \frac{n_{s+1,B}}{B}}&\leq D\bigg(B^{- c_\alpha \gamma_{s+1}}+ B^{\frac{\gamma_{s+1}-1}{2}  - c_\alpha \gamma_s}\bigg),\quad\forall B\in\mathbb{N}.
\end{align}

To proceed further let $\gamma'_{s+1}$ be as defined in Lemma \ref{lemma:gamma} when $\alpha'=\alpha$, and note that  $\gamma'_{s+1}$ is such that
\begin{align*}
B^{- c_\alpha \gamma'_{s+1}}=B^{\frac{\gamma'_{s+1}-1}{2}  - c_\alpha\gamma_s},\quad\forall B\in\mathbb{N}.
\end{align*}
By Lemma \ref{lemma:gamma},  we have $\gamma_{s+1}< \gamma'_{s+1}$  and thus, since for all $B\in\mathbb{N}\setminus\{1\}$ the mapping $\gamma\mapsto   B^{- c_\alpha \gamma}$ is decreasing in $\gamma$ while the mapping $\gamma\mapsto B^{\frac{\gamma-1}{2}  - c_\alpha \gamma_s}$ is increasing in $\gamma$, it follows that
$$
 B^{\frac{\gamma_{s+1}-1}{2}  - c_\alpha \gamma_s}\leq  B^{- c_\alpha \gamma_{s+1}},\quad\forall B\in\mathbb{N}
$$
and thus, by \eqref{eq:in_tat},
\begin{align}\label{eq:thm2_b2}
\tau_{s+1,B}+r_{s,B}\sqrt{ \frac{n_{s+1,B}}{B}}&\leq 2D  B^{- c_\alpha \gamma_{s+1}}=2D r_{s+1,B},\quad\forall B\in\mathbb{N}.
\end{align}

By combining \eqref{eq:main_thm2} and \eqref{eq:thm2_b2}, it follows that there exists a constant $C_{\epsilon,s+1}\in[1,\infty]$ such that, for $j=s+1$,  \eqref{eq:Show_indMain} holds for all  $\omega\in \Omega_{\epsilon/2,s,B}\cap \Omega'_{\epsilon/2,s+1,B}$ and $B\in\mathbb{N}$ where, using Fr\'echet's inequality,
$$
\liminf_{B\rightarrow\infty}\P\big(\Omega_{\epsilon/2,s,B}\cap \Omega'_{\epsilon/2,s+1,B}\big)\geq  \liminf_{B\rightarrow\infty}\P\big(\Omega_{\epsilon/2,s,B}\big)+\liminf_{B\rightarrow\infty}\P\big(\Omega'_{\epsilon/2,s+1,B}\big)-1\geq 1-\epsilon.
$$
Since $\epsilon>0$ is arbitrary, this concludes to   show     that for all $j\in\mathbb{N}$ there exists, for all   $\epsilon\in(0,\infty)$,  a sequence $(\Omega_{\epsilon,j,B})_{B\geq 1}$ in $\F$  such that $\liminf_{B\rightarrow\infty}\P(\Omega_{\epsilon,j,B})\geq 1-\epsilon$, and a constant  $C_{\epsilon,j}\in [1,\infty)$ such that \eqref{eq:Show_indMain} holds for all $\omega\in\Omega_{\epsilon,j,B}$ and all $B\in\mathbb{N}$. The second part of the theorem follows.

\end{proof}

\end{document}